\documentclass[a4paper, 12pt, reqno,final]{amsart}
\usepackage[foot]{amsaddr} 
\usepackage[utf8x]{inputenc}
\usepackage[T1]{fontenc}
\usepackage[english]{babel}
\usepackage{amsmath, amsfonts, amsthm, amssymb, amscd}
\usepackage[final]{graphicx}
\usepackage[below]{placeins}
\usepackage{subfig}
\usepackage{multirow}
\usepackage{mathtools}

\usepackage[hidelinks]{hyperref}
\hypersetup{
  pdfauthor = {Stefan Metzger},
  pdftitle = {SSAV}
}

\usepackage{showkeys} 
\usepackage{esint}
\usepackage{caption}
\usepackage{dsfont}

\usepackage{tikz}
\usetikzlibrary{patterns}
\usetikzlibrary{shapes.geometric}

\usepackage[]{todonotes}
\renewcommand{\todo}[1]{}
\usepackage{diagbox}

\newtheorem{theorem}{Theorem}[section]
\newtheorem{lemma}[theorem]{Lemma}
\newtheorem*{lemma*}{Lemma}
\newtheorem{corollary}[theorem]{Corollary}

\newtheorem{remark}[theorem]{Remark}

\numberwithin{equation}{section}


\newcounter{IMASTYLE}
\makeatletter
\newcommand{\labitem}[2]{%
\def\@itemlabel{\textbf{#1}}
\item
\def\@currentlabel{#1}\label{#2}}
\makeatother

\newcommand{\norm}[1]{\left\|{#1}\right\|}
\newcommand{\abs}[1]{\left|{#1}\right|}

\newcommand{\rkla}[1]{{\left(#1\right)}}
\newcommand{\trkla}[1]{{(#1)}}
\newcommand{\gkla}[1]{{\left\{#1\right\}}}
\newcommand{\tgkla}[1]{{\{#1\}}}
\newcommand{\skla}[1]{{\left\langle#1\right\rangle}}
\newcommand{\ekla}[1]{{\left[#1\right]}}
\newcommand{\tekla}[1]{{[#1]}}
\newcommand{\tabs}[1]{|{#1}|}

\newcommand{\bs}[1]{\boldsymbol{#1}}

\newcommand{\Om}{\mathcal{O}}

\newcommand{\iO}{\int_{\Om}}

\newcommand{\para}[1]{\partial _{#1}}

\newcommand{\dx}{\, \mathrm{d}{x}}

\newcommand{\dt}{\, \mathrm{d}t}
\newcommand{\ds}{\, \mathrm{d}s}
\newcommand{\dr}{\, \mathrm{d}r}

\newcommand{\nn}{^{n}}

\newcommand{\no}{^{n-1}}

\newcommand{\tl}{^{\tau}}
\newcommand{\tp}{^{\tau,+}}
\newcommand{\tm}{^{\tau,-}}
\newcommand{\tpm}{^{\tau,(\pm)}}
\newcommand{\pml}{^{(\pm)}}

\newcommand{\h}{_{h}}
\newcommand{\hj}{_{h_j}}

\newcommand{\weak}{\rightharpoonup}
\newcommand{\weakstar}{\stackrel{*}{\rightharpoonup}}
\newcommand{\weaktop}{{\mathrm{weak}}}

\newcommand{\Uh}{U_{h}}
\newcommand{\Uhj}{U_{h_j}}
\newcommand{\Th}{\mathcal{T}_h}

\newcommand{\Ihop}{\mathcal{I}_h}

\newcommand{\Ih}[1]{\Ihop\gkla{#1}}

\newcommand{\Ihjop}{\mathcal{I}_{h_j}}
\newcommand{\Ihj}[1]{\Ihjop\gkla{#1}}

\newcommand{\restr}[2]{\ensuremath{
  \left.\kern-\nulldelimiterspace 
  #1 
  \vphantom{\big|} 
  \right|_{#2} 
  }}
  
\newcommand{\extend}[2]{\ensuremath{
  \left.\kern-\nulldelimiterspace 
  #1 
  \vphantom{\big|} 
  \right|^{#2} 
  }}
  
\newcommand{\diam}{\operatorname{diam}}

\newcommand{\g}[1]{\mathfrak{g}_{#1}}

\newcommand{\Zh}{\mathds{Z}\h}

\newcommand{\db}[1]{\,\operatorname{d}\!\beta_{#1}}

\newcommand{\p}{{\mathfrak{p}}}

\newcommand{\expected}[1]{\mathds{E}\ekla{#1}}
\newcommand{\expectedt}[1]{\widetilde{\mathds{E}}\ekla{#1}}
\newcommand{\Prob}{\mathds{P}}
\newcommand{\prob}[1]{\Prob\ekla{#1}}

\newcommand{\projop}{\mathcal{P}_{\Uh}}

\newcommand{\Ritzop}{\mathcal{R}_{\Uhj}}
\newcommand{\Ritz}[1]{\Ritzop{#1}}

\DeclareMathOperator*{\esssup}{ess\,sup}

\newcommand{\Ito}{It\^{o}}

\newcommand{\sinc}[1]{\blacktriangle^{#1}\boldsymbol{\xi}^\tau}

\newcommand{\sinctilde}[1]{\blacktriangle^{#1}\widetilde{\boldsymbol{\xi}}^j}

\definecolor{colorONE}{RGB}{0, 0, 0}
\definecolor{colorTWO}{RGB}{56, 125, 184}
\definecolor{colorTHREE}{RGB}{153, 79, 163}
\definecolor{colorFOUR}{RGB}{255, 0, 0}
\definecolor{colorFIVE}{RGB}{255, 128, 0}
\definecolor{colorSIX}{RGB}{0, 170, 0}

\begin{document}
\title[A convergent SSAV method]{A convergent stochastic scalar auxiliary variable method}
\date{\today}
\author[S.~Metzger]{Stefan Metzger}
\address[S.~Metzger]{Friedrich--Alexander Universität Erlangen--Nürnberg,~Cauerstraße 11,~91058~Erlangen,~Germany}
\email{stefan.metzger@fau.de}


\keywords{stochastic Allen--Cahn equation, multiplicative noise, finite elements, convergence, scalar auxiliary variable}
\subjclass[2010]{60H35, 65M60, 60H15, 65M12}



%
\selectlanguage{english}

\allowdisplaybreaks

\begin{abstract}
We discuss an extension of the scalar auxiliary variable approach, which was originally introduced by Shen et al.~([Shen, Xu, Yang, J.~Comput.~Phys., 2018]) for the discretization of deterministic gradient flows.
By introducing an additional scalar auxiliary variable, this approach allows to derive a linear scheme, while still maintaining unconditional stability.
Our extension augments the approximation of the evolution of this scalar auxiliary variable with higher order terms, which enables its application to stochastic partial differential equations.
Using the stochastic Allen--Cahn equation as a prototype for nonlinear stochastic partial differential equations with multiplicative noise, we propose an unconditionally energy stable, linear, fully discrete finite element scheme based on our augmented scalar auxiliary variable method.
Recovering a discrete version of the energy estimate and establishing Nikolskii estimates with respect to time, we are able to prove convergence of discrete solutions towards pathwise unique martingale solutions by applying Jakubowski's generalization of Skorokhod's theorem.
A generalization of the Gyöngy--Krylov characterization of convergence in probability to quasi-Polish spaces finally provides convergence of fully discrete solutions towards strong solutions of the stochastic Allen--Cahn equation.
Finally, we present numerical simulations underlining the practicality of the scheme and the importance of the introduced augmentation terms.
\end{abstract}
\maketitle

\section{Introduction}
We propose a new approach for the numerical approximation of solutions to nonlinear stochastic partial differential equations.
During the last decades, major progress in the numerical treatment of stochastic partial differential equations was made.
Concerning schemes for the numerical approximation of linear stochastic partial differential equations we refer the reader to \cite{GyongyNualart1997,Yan2005} and the references therein.
Results on various different nonlinear stochastic partial differential equations (SPDEs) were discussed e.g.~in \cite{Printems2001, DeBouardDebussche2004, GyongyMillet2009, Debussche2011, CarelliProhl2012, Banas2013,  Brzezniak2013, MajeeProhl2018, CuiHongLiuZhou2019, GrillmeierGruen2019, Grillmeier2020, HutzenthalerJentzen2020, AntonopoulouBanasNurnbergProhl2021, CuiHongSun2022_arxiv} and the references therein.
While most discrete schemes are tailored to It\^o noise, a specific approximation of the Stratonovich integral was used in \cite{Banas2013} to transfer the sphere constraint imposed on solutions to the stochastic Landau--Lifshitz--Gilbert equation to the discrete setting.
A recent attempt on the unification of the techniques from stochastic and numerical analysis was published in \cite{Ondrejat2022}.
There, a general convergence theory for nonlinear stochastic partial differential equations based on the Lax--Richtmeyer theorem was presented.
This general approach consists of the following three steps:
In the first step, uniform bounds on the numerical solutions need to be established. 
Based on these bounds, compactness results can then be obtained via Prokhorov's theorem.
Finally, these compactness results can be used to identify (weakly) converging subsequences and to pass to the limit.\\
In particular the first step of this approach is highly problem dependent and requires mimicking the structural properties of the original PDE in the discrete setting.
Although this is typically well understood for deterministic PDEs, an application to stochastic PDEs is not straightforward.
While the typical stability results for discrete schemes of deterministic PDEs require a testing procedure relying on mostly implicit test functions, the tools available for the treatment of the stochastic terms rely on the independence of the stochastic increments -- i.e.~they require completely explicit test functions.
To satisfy both of these requirements, suitable absorption arguments have to be used requiring in particular a careful treatment of the nonlinearities.
In this publication, we present a new approach for the discretization of stochastic PDEs.
For the simplicity of representation, we shall consider the stochastic Allen--Cahn equation as an example of a nonlinear stochastic PDE with multiplicative noise:\\
Let $\trkla{\Omega,\mathcal{A},\mathcal{F},\Prob}$ be a filtered probability space with a filtration $\mathcal{F}=\trkla{\mathcal{F}_t}_{t\in\tekla{0,T}}$ and $\Om\subset\mathds{R}^d$ ($d\in\tgkla{2,3}$) be a bounded Lipschitz domain.
On this domain, we consider the process $\phi\,:\,\Omega\times\tekla{0,T}\times\overline{\Om}\rightarrow\mathds{R}$ solving
\begin{subequations}\label{eq:model}
\begin{align}
\mathrm{d}\phi+\trkla{-\varepsilon\Delta\phi+\tfrac1\varepsilon F^\prime\trkla{\phi}}\dt&=\Phi\trkla{\phi}\mathrm{d}W\,,\label{eq:model:phi}\\
\nabla\phi\cdot\bs{n}&=0\label{eq:model:bc}\,,
\end{align}
\end{subequations}
with given parameter $\varepsilon>0$.
Alternatively, one can formulate \eqref{eq:model:phi} on the $d$-dimensional torus and hence replace \eqref{eq:model:bc} by periodic boundary conditions. 
In this case the results presented in this paper remain valid.
The deterministic version of \eqref{eq:model}, the Allen--Cahn equation, describes the dynamics of a diffuse interface whose width is governed by the parameter $\varepsilon$. 
It can be derived as the $L^2$-gradient flow of the Helmholtz free energy functional
\begin{align}\label{eq:Helmholtz}
\mathcal{E}\trkla{\phi}=\tfrac12\varepsilon\iO \tabs{\nabla\phi}^2\dx+\varepsilon^{-1}\iO F\trkla{\phi}\dx
\end{align}
and satisfies the energy equality
\begin{align}
\mathcal{E}\trkla{\phi}\big\vert_{t=T}+\int_0^T\iO \abs{-\varepsilon\Delta\phi+\tfrac1\varepsilon F^\prime\trkla{\phi}}^2\dx\dt=\mathcal{E}\trkla{\phi}\big\vert_{t=0}\,.
\end{align}
Here, $F$ denotes a double-well potential with minima in (or close to) $\phi=\pm1$ describing the mixing energy.
Typical choices for $F$ that constrain $\phi$ to the physically relevant interval $\tekla{-1,+1}$ are the logarithmic double-well potential
\begin{align}
F_{\log}\trkla{\phi}:=\tfrac\vartheta2\trkla{1+\phi}\log\trkla{1+\phi}+\tfrac\vartheta2\trkla{1-\phi}\log\trkla{1-\phi}-\tfrac{\vartheta_c}2\phi^2
\end{align}
with $0<\vartheta<\vartheta_c$ and the double obstacle potential
\begin{align}
F_{\operatorname{obst}}\trkla{\phi}:=\left\{\begin{array}{cc}
\vartheta\trkla{1-\phi^2}&\phi\in\tekla{-1,+1}\\
\infty&\text{else}
\end{array}\right.&&\text{with~}\vartheta>0\,.
\end{align}
To avoid the technical issues arising from the singularities, $F_{\log}$ is often approximated by the polynomial double-well potential $F_{\operatorname{pol}}\trkla{\phi}:=\tfrac14\trkla{\phi^2-1}^2$.
In this publication, we will neglect singular potentials like the logarithmic double-well potential or the double obstacle potential and focus on variations of the polynomial double-well potential.
The stochastic part of \eqref{eq:model} describing the influence of thermal fluctuations on the evolution of the interface is driven by a $\mathcal{Q}$-Wiener process $W= \trkla{W_t}_{t\in\tekla{0,T}}$ defined on $\trkla{\Omega,\mathcal{A},\mathcal{F},\Prob}$.\\
The stochastic Allen--Cahn equation and its numerical treatment have already been extensively studied.
Although it was shown that purely explicit discretizations of SPDEs with non-globally Lipschitz continuous coefficients are prone to divergence (cf.~\cite{HutzenthalerJentzenKloeden2011}), convergence and error estimates can still be established for tamed or truncated schemes (see e.g.~\cite{HutzenthalerJentzen2015, BeckerJentzen2019, JentzenPusnik2019, Wang2020, CaiGanWang2021, BeckerGessJentzenKloeden2023}).
Other results provide error estimates for splitting schemes or semi-implicit Euler--Maruyama schemes that treat the Laplacian implicitly, but discretize the nonlinear parts still explicitly (cf.~\cite{BrehierCuiHong2018, BrehierGoudenege2019, BrehierGoudenege2020, QiAzaiezHuangXu2022}).
For the case of additive noise, error estimates for schemes that employ an implicit time-discretization of the nonlinearity were studied e.g.~in \cite{KovacsLarssonLindgren2015, KovacsLarssonLindgren2018, LiuQiao2019, QiWang2019}.
The case of multiplicative \Ito~noise was discussed e.g.~in \cite{ MajeeProhl2018, LiuQiao2021, BreitProhl2024}.
Here, \cite{LiuQiao2021} analyzed drift-implicit Euler schemes as well as higher-order Milstein schemes.
Gradient-type multiplicative Stratonovich noise was analyzed in \cite{FengLiZhang2017}.\\
In this publication, we want to adapt the scalar auxiliary variable (SAV) method to derive discrete schemes that are linear in the unknown quantities, but still unconditionally stable with respect to a modified energy.
Such a method was recently employed in \cite{QiZhangXu2023} for the numerical simulation of the stochastic Allen--Cahn equation with random diffusion coefficient field and multiplicative noise.
An analytical convergence proof is yet still missing.\\
The advantages of a convergent SAV method are two-fold.
First, the SAV approach allows for a significant reduction in computation time.
Secondly, it also allows to hide the nonlinearity in the newly introduced auxiliary variable, hence, opening a pathway to generalized results that do not depend on the exact form of the nonlinearity anymore.\\

The SAV approach was originally introduced by
\ifnum\value{IMASTYLE}>1 {\cite{ShenXuYang2018}}\else {Shen et al.~in \cite{ShenXuYang2018}\fi~for deterministic PDEs describing gradient flows.
It has been applied to various problems (see e.g.~\cite{ShenXuYang19, ShenYang20} and the references therein) and many different variations of this approach have been proposed and tested (see e.g.~\cite{HouAzaiezXu19, LiuLi20, HuangShenYang2020, YangDong2020, ZhangShen2022}).
From the analytical point of view, convergence results providing error estimates under the assumption that a sufficiently regular solution exists have been extensively studied (see e.g.~\cite{ShenXu18,LinCaoZhangSun20,YangZhang20,LiShen2020}).
Yet, convergence results without such regularity assumptions on the solution are scarce:
In \cite{ShenXu18}, the convergence behavior of discrete solutions to $L^2$- and $H^{-1}$-gradient flows was investigated.
Assuming that the initial conditions have $H^4$-regularity, (weak) convergence of subsequences of discrete solutions was established.
For an investigation of the asymptotic behavior of SAV schemes, we refer the reader to \cite{BouchritiPierreAlaa20}.
Recently, the author applied the SAV method to Cahn--Hilliard equations with reaction rate dependent dynamic boundary conditions and established the convergence of subsequences of discrete solutions towards weak solutions of the original model (see \cite{KnopfLamLiuMetzger2021} for the derivation of the model and \cite{Metzger2023} for the analysis  of the SAV scheme).
Analytical results for an SAV discretization of the Cahn--Hilliard equation with deterministic source terms can be found in \cite{LamWang2023}.\\
The goal of this publication is the adaption of the SAV method to SPDEs.
The intricacy in this endeavor stems from the severe limitation of the time-regularity of the solution caused by the additional stochastic terms.
In the deterministic case, it is typically rather easy to show that solutions are Hölder continuous in time with an exponent greater than $1/2$, which is sufficient to establish the convergence of the auxiliary variable.
For most stochastic PDEs, however, this is more intricate. 
Although for some SPDEs like the stochastic wave equation the standard SAV method as it was proposed in \cite{ShenXuYang2018} can be applied successfully and convergence can be established (cf.~\cite{CuiHongSun2022_arxiv}), for most stochastic PDEs a successful application of the SAV method is not straightforward.
In particular, we can often only anticipate that the solution is Hölder continuous with an exponent smaller than $1/2$ which is not sufficient to show convergence for the standard SAV schemes.
Hence, we propose an augmented version of the SAV method and establish  convergence of the scheme even for SPDEs with less regular solutions.
To ensure the convergence, we need to augment the evolution equation for the auxiliary variable by a particular approximation of higher order terms to obtain a linear and convergent scheme (cf.~Remark \ref{rem:taylor} below).\\

The outline of the paper is as follows:
In Section \ref{sec:notation}, we introduce the finite and infinite dimensional function spaces and collect our assumptions on the data.
Our augmented SAV scheme is presented in Section \ref{sec:discscheme}.
The main results (existence of discrete solutions, convergence towards unique martingale solutions, and convergence towards unique strong solutions) are collected in Section \ref{sec:mainresults}. 
Their proofs can be found in the subsequent sections.
The existence of discrete solutions is established in Section \ref{sec:discexistence}.
Sections \ref{sec:regularity}-\ref{sec:uniqueness} contain the proof of Theorem \ref{thm:maintheorem}, i.e.~convergence towards pathwise unique martingale solutions.
In particular, we will use Section \ref{sec:regularity} to show that the proposed scheme is unconditionally stable with respect to a modified energy and to establish improved uniform regularity results for the discrete solutions.
These results will be used in Section \ref{sec:compactness} to identify weakly and strongly converging subsequences by applying Jakubowski's theorem.
Section \ref{sec:limit} is devoted to the convergence of the stochastic source term. 
Following the arguments in e.g.~\cite{Ondrejat2022}, we show that the discrete approximations of the stochastic source term converge towards an \Ito-integral driven by a $\mathcal{Q}$-Wiener process.
We conclude the proof of Theorem \ref{thm:maintheorem} in Section \ref{sec:uniqueness} by establishing the pathwise uniqueness of the martingale solutions.
The proof of Theorem \ref{thm:strongconvergence} can be found in Section \ref{sec:strong}, where we use a generalized version of the Gyöngy--Krylov characterization of convergence in probability to show that our discrete solutions converge for a given Wiener process towards strong solutions.
We conclude this paper by presenting numerical simulations in Section \ref{sec:numerics}. 
By comparing the augmented SAV scheme to the discretization analyzed in \cite{MajeeProhl2018} we validate the practicality of the proposed scheme.
By comparing the results with a straightforward application of the SAV method, we underline the importance of the augmentation terms in the stochastic setting.

\section{Notation and assumptions}\label{sec:notation}
The spatial domain $\Om\subset\mathds{R}^d$ with $d\in\tgkla{2,3}$ is assumed to be bounded and convex.
To avoid additional technicalities, we further assume that $\Om$ is polygonal (or polyhedral, respectively).
We denote the space of $k$-times weakly differentiable functions with weak derivatives in $L^p\trkla{\Om}$ by $W^{k,p}\trkla{\Om}$.
For $p=2$, we denote the Hilbert spaces $W^{k,2}\trkla{\Om}$ by $H^k\trkla{\Om}$. 
For a Banach space $X$ and a set $I$, the symbol $L^p\trkla{I;X}$ ($p\in\tekla{1,\infty}$) stands for the Bochner space of strongly measurable $L^p$-integrable functions on $I$ with values in $X$.
If $X$ is only separable (and not reflexive), we follow the notation used in \cite[Chapter 0.3]{FeireislNovotny17} and denote the dual space of $L^{p/\trkla{p-1}}\trkla{I;X}$ by $L^p_{\operatorname{weak-}\trkla{*}}\trkla{I;X^\prime}$.
By $C^{k,\alpha}\trkla{I;X}$ with $k\in\mathds{N}_0$ and $\alpha\in(0,1]$, we denote the space of $k$-times continuously differentiable functions from $I$ to $X$ whose $k$-th derivatives are Hölder continuous with Hölder exponent $\alpha$.
If $I=X=\mathds{R}$, we will write $C^{k,\alpha}\trkla{\mathds{R}}$.
We shall also introduce the Nikolskii spaces $N^{\alpha,\beta}$ ($\alpha\in\trkla{0,1}$, $\beta\in\tekla{1,\infty}$) which are defined for a time interval $\trkla{0,T}$ and a Banach space $X$ via
\begin{align}
N^{\alpha,\beta}\trkla{0,T;X}:=\gkla{f\in L^\beta\trkla{0,T;X}\,:\, \sup_{k>0} k^{-\alpha}\norm{f\trkla{\cdot+k}-f\trkla{\cdot}}_{L^\beta\trkla{-k,T-k;X}}<\infty}\,.
\end{align}
Here, the standard convention $f\equiv0$ outside of $\trkla{0,T}$ for $f\in L^\beta\trkla{0,T;X}$ is used.\\

Concerning the discretization with respect to time, we assume that
\begin{itemize}
\labitem{\textbf{(T)}}{item:timedisc} the time interval $I:=[0,T]$ is subdivided into $N$ equidistant intervals $I_n$ given by nodes $\trkla{t^n}_{n=0,\ldots,N}$ with $t^0=0$, $t^N=T$, and $t^{n+1}-t^n=\tau=\tfrac{T}{N}$. Without loss of generality, we assume $\tau<1$.
\end{itemize}
The spatial discretization is based on partitions $\Th$ of $\Om$ depending on a discretization parameter $h>0$ satisfying the following assumption:
\begin{itemize}
\labitem{\textbf{(S)}}{item:spatialdisc} $\tgkla{\Th}_{h>0}$ is a quasi-uniform family of partitions of $\Om$ (in the sense of \cite{BrennerScott}) into disjoint, open simplices $K$, satisfying
\begin{align*}
\overline{\Om}\equiv\bigcup_{K\in\Th}\overline{K}&&&\text{with~}\max_{K\in\Th}\diam\trkla{K}\leq h\,.
\end{align*}
\end{itemize}
For our approximation, we consider the space of continuous, piecewise linear finite element functions on $\Th$. 
This space shall be denoted by $\Uh$ and is spanned by functions $\tgkla{\chi_{h,k}}_{k=1,\ldots,\dim\Uh}$ forming a dual basis to the vertices $\tgkla{\bs{x}_{h,k}}_{k=1,\ldots,\dim\Uh}$ of $\Th$.
We consider the nodal interpolation operators $\Ihop\,:\,C^0\trkla{\overline{\Om}}\rightarrow\Uh$ defined by
\begin{align}
\Ih{a}:=\sum_{k=1}^{\dim\Uh} a\trkla{\bs{x}_{h,k}}\chi_{h,k}\,.
\end{align}
For future reference, we state the following norm equivalence for $p\in[1,\infty)$ and $\zeta\h\in\Uh$:
\begin{align}\label{eq:normequivalence}
c\norm{\zeta\h}_{L^p\trkla{\Om}}\leq \rkla{\iO\Ih{\tabs{\zeta\h}^p}\dx}^{1/p}\leq C\norm{\zeta\h}_{L^p\trkla{\Om}}
\end{align}
with constants $c,\,C>0$ independent of $h$. 
A proof of this norm equivalence can be found e.g.~in \cite[Lemma 3.2.11]{SieberDiss2021}\footnote{In \cite{SieberDiss2021}, the additional global assumption that the quasi-uniform family of partitions consists of non-obtuse simplices was used. Yet, this stronger assumption is not required in the proof of Lemma 3.2.11 which solely relies on standard inverse estimates.}.
To simplify the notation, we shall define the discrete norms
\begin{align}
\norm{\zeta\h}\h:=\sqrt{\iO\Ih{\tabs{\zeta\h}^2}\dx}&&\text{and}&&\norm{\zeta\h}_{H^1\h\trkla{\Om}}:=\sqrt{\norm{\zeta\h}\h^2+\norm{\nabla\zeta\h}^2_{L^2\trkla{\Om}}}
\end{align}
on $\Uh$.
By \eqref{eq:normequivalence}, these norms are equivalent to the usual $L^2\trkla{\Om}$-norm and $H^1\trkla{\Om}$-norm, respectively.
For future reference, we also recall the following error estimates which were proven e.g.~in Lemma 2.1 in \cite{Metzger2020}:
\begin{lemma}\label{lem:interpolation}
Let $\mathcal{T}\h$ satisfy \ref{item:spatialdisc} and let $p\in[1,\infty)$ and $1\leq q\leq\infty$.
Then, for $q^*=\tfrac{q}{q-1}$, if $q<\infty$, and $q^*=1$, if $q=\infty$, the estimates
\begin{subequations}
\begin{align*}
\norm{\trkla{1-\Ihop}\tgkla{f\h g\h}}_{L^p\trkla{\Om}}&\leq Ch^2\norm{\nabla f\h}_{L^{pq}\trkla{\Om}}\norm{\nabla g\h}_{L^{pq^*}\trkla{\Om}}\,,\\
\norm{\trkla{1-\Ihop}\tgkla{f\h g\h}}_{W^{1,p}\trkla{\Om}}&\leq Ch\norm{\nabla f\h}_{L^{pq}\trkla{\Om}}\norm{\nabla g\h}_{L^{pq^*}\trkla{\Om}}
\end{align*}
\end{subequations}
hold true for all $f\h,\,g\h\in\Uh$.
\end{lemma}

To approximate the Laplacian, we use the discrete Laplacian $\Delta\h\,:\,\Uh\rightarrow \Uh$ defined via
\begin{align}\label{eq:def:discLaplacian}
\iO\Ih{\Delta\h\zeta\h\psi\h}\dx=-\iO\nabla\zeta\h\cdot\nabla\psi\h\dx&&\text{for~}\zeta\h,\psi\h\in\Uh\,.
\end{align}

For the double-well potential $F$, we shall make the following assumptions:
\begin{itemize}
\labitem{\textbf{(P)}}{item:potentialF} $F\in C^2\trkla{\mathds{R}}$ is bounded from below by a positive constant $\gamma>0$ and satisfies the growth estimates
\begin{align*}
c_\gamma\trkla{1+\tabs{\zeta}^4}&\leq F\trkla{\zeta}\leq C\trkla{1+\tabs{\zeta}^4}\,,&\tabs{F^\prime\trkla{\zeta}}&\leq C\trkla{1+\tabs{\zeta}^3}\,,&\tabs{F^{\prime\prime}\trkla{\zeta}}\leq C\trkla{1+\tabs{\zeta}^2}
\end{align*}
for all $\zeta\in\mathds{R}$ and a positive constant $c_\gamma$.\\
Furthermore, $F$ can be decomposed into $F_1\in C^{2,\nu}\trkla{\mathds{R}}$ and $F_2\in C^3\trkla{\mathds{R}}$ satisfying $\tabs{F_2^{\prime\prime\prime}\trkla{\zeta}}\leq C\trkla{1+\tabs{\zeta}^2}$.
\end{itemize}
\begin{remark}
Having a positive lower bound for the double-well potential $F$ is purely a technical assumption which is required to state the SAV scheme. 
As the stochastic Allen--Cahn equation \eqref{eq:model} only depends on the derivative $F^\prime$, we can always add an arbitrary constant to $F$ without changing \eqref{eq:model}.
Hence, any lower bound for the double-well potential is sufficient and the necessary shift can be interpreted as a numerical parameter.\\
The assumptions stated in \ref{item:potentialF} are satisfied for instance by a shifted polynomial double-well potential $\widetilde{F}_{\operatorname{pol}}\trkla{\phi}:=\tfrac14\trkla{\phi^2-1}^2+\gamma$ for any $\gamma>0$.
We are, however, not limited to this specific choice.
Hence, our assumptions on the potential are more general than the typical ones used e.g.~in \cite{MajeeProhl2018}.\\
In comparison with the assumptions used in \cite{Metzger2023} to treat the deterministic Cahn--Hilliard equation with dynamic boundary conditions, the assumptions stated in \ref{item:potentialF} are stricter.
This is owed to the fact that the stochastic noise term in \eqref{eq:model} reduces the solution's regularity with respect to time.
Therefore, we need to use a higher order approximation for the evolution of the scalar auxiliary variable which involves higher derivatives (cf.~\eqref{eq:discscheme:r}) to ensure convergence.
Furthermore, the polynomial lower bound on $F$ is necessary, to allow for cancellation effects needed to control the source terms (see e.g.~\eqref{eq:ehtrick} below or \cite[Remark 5]{LamWang2023}).
\end{remark}

For simplicity, we consider deterministic initial data and assume that
\begin{itemize}
\labitem{\textbf{(I)}}{item:initial} $\phi_0\in H^1\trkla{\Om}$ and there is a family of mappings $\mathcal{P}\h\,:\,H^1\trkla{\Om}\rightarrow \Uh$, $h\in(0,1]$, such that the estimate $\norm{\mathcal{P}\h\phi_0}_{H^1\trkla{\Om}}\leq C\trkla{\phi_0}$ with $C$ independent of $h$, and $\mathcal{P}\h\phi_0\rightarrow\phi_0$ in $L^2\trkla{\Om}$ for $h\searrow0$.
\end{itemize}
An example for such a mapping is the $H^1$-stable $L^2$-projection $\projop\,:\, L^2\trkla{\Om}\rightarrow \Uh$ (cf.~\cite{BramblePasciakSteinbach2002}).
If the initial data is more regular, i.e.~$\phi_0\in H^2\trkla{\Om}$, one could also use the nodal interpolation operator $\Ihop$.\\
The stochastic source term in \eqref{eq:model} is governed by a $\mathcal{Q}$-Wiener process and an operator $\Phi$.
We shall assume that
\begin{itemize}
\labitem{\textbf{(W1)}}{item:stochastic} the trace class operator $\mathcal{Q}\,:\,L^2\trkla{\Om}\rightarrow L^2\trkla{\Om}$ satisfies
\begin{align}
\mathcal{Q}g:=\sum_{k\in\mathds{Z}}\lambda_k^2\trkla{g,\g{k}}_{L^2\trkla{\Om}}\g{k}\,,
\end{align}
where $\trkla{.,.}_{L^2\trkla{\Om}}$ denotes the $L^2$-inner product, $\trkla{\lambda_k}_{k\in\mathds{Z}}$ are given real numbers, and $\trkla{\g{k}}_{k\in\mathds{Z}}$ is an orthonormal basis of $L^2\trkla{\Om}$. 
\end{itemize}
Hence, $\mathcal{Q}^{1/2}$ given by $\mathcal{Q}^{1/2}g:=\sum_{k\in\mathds{Z}}\lambda_k\trkla{g,\g{k}}_{L^2\trkla{\Om}}\g{k}$ is a Hilbert--Schmidt operator from $L^2\trkla{\Om}$ to $L^2\trkla{\Om}$.
We shall denote its image of $L^2\trkla{\Om}$ by
\begin{align}
\mathcal{Q}^{1/2}L^2\trkla{\Om}:=\gkla{\mathcal{Q}^{1/2}g\,:\,g\in L^2\trkla{\Om}}\,.
\end{align}

Hence, the $\mathcal{Q}$-Wiener process $\trkla{W_t}_{t\in\tekla{0,T}}$ has the representation
\begin{align}\label{eq:representationWiener}
W_t:=\sum_{k\in\mathds{Z}}\lambda_k\g{k}\beta_k\trkla{t}
\end{align}
with $\trkla{\beta_k}_{k\in\mathds{Z}}$ being mutually independent Brownian motions.
Furthermore, we shall assume that the $\mathcal{Q}$-Wiener process is colored in the sense that
\begin{itemize}
\labitem{\textbf{(W2)}}{item:Wcolor} there exists a positive constant $\widetilde{C}$ such that
\begin{align}
\sum_{k\in\mathds{Z}}\lambda_k^2\norm{\g{k}}_{W^{1,\infty}\trkla{\Om}}^2\leq\widetilde{C}\,.\label{eq:Wcolor}
\end{align}
\end{itemize}
The operator $\Phi$ mapping a stochastic process $\phi$ into the space of Hilbert--Schmidt operators from $\mathcal{Q}^{1/2}L^2\trkla{\Om}$ to $L^2\trkla{\Om}$ is defined via
\begin{align}\label{eq:Phi}
\Phi\trkla{\phi}g:=\sigma\trkla{\phi}\sum_{k\in\mathds{Z}}\trkla{g,\g{k}}_{L^2\trkla{\Om}}\g{k}
\end{align}
for all $g\in\mathcal{Q}^{1/2}L^2\trkla{\Om}$.
With respect to the coefficient function $\sigma\,:\,\mathds{R}\rightarrow\mathds{R}$ our assumptions are similar to the ones used in \cite{MajeeProhl2018}.
In particular, we shall assume that
\begin{itemize}
\labitem{\textbf{(C)}}{item:sigma} $\sigma\in L^\infty\trkla{\mathds{R}}\cap C^{0,1}\trkla{\mathds{R}}$.
\end{itemize}
This assumption guarantees that $\Phi$ as it is defined in \eqref{eq:Phi} is indeed a mapping into the space of Hilbert--Schmidt operators from $\mathcal{Q}^{1/2}L^2\trkla{\Om}$ to $L^2\trkla{\Om}$, since
\begin{align}
\sum_{k\in\mathds{Z}}\norm{\Phi\trkla{\phi}\trkla{\lambda_k\g{k}}}_{L^2\trkla{\Om}}^2=\sum_{k\in\mathds{Z}}\norm{\sigma\trkla{\phi}\lambda_k\g{k}}_{L^2\trkla{\Om}}^2\leq \norm{\sigma}_{L^\infty\trkla{\mathds{R}}}^2 \sum_{k\in\mathds{Z}}\norm{\mathcal{Q}^{1/2}\g{k}}_{L^2\trkla{\Om}}^2\leq C\,.
\end{align}
Assumption \ref{item:sigma} also implies that $\norm{\nabla\Ih{\sigma\trkla{\zeta\h}}}_{L^p\trkla{\Om}}\leq \overline{C}\norm{\nabla\zeta\h}_{L^p\trkla{\Om}}$ for any $\zeta\h\in \Uh$ and $p\in\tekla{1,\infty}$ with a constant $\overline{C}$ depending on the Lipschitz constant of $\sigma$.
Hence, the source term on the right-hand side of \eqref{eq:model:phi} can be written as
\begin{align}
\Phi\trkla{\phi}\mathrm{d}W:=\sum_{k\in\mathds{Z}}\lambda_k\g{k}\sigma\trkla{\phi}\db{k}\,.
\end{align}\\
In our discrete scheme, we shall approximate $W$ by a sequence of discrete stochastic increments denoted $\tgkla{\sinc{n}}_{n=1,\ldots,N}$ which are supposed to have the following properties:
\begin{itemize}
\labitem{\textbf{(D0)}}{item:filtration} Let $\mathcal{F}^\tau=\trkla{\mathcal{F}_t^\tau}_{t\in[0,T]}$ be defined by $\mathcal{F}_t^\tau=\mathcal{F}_{t\no}$ for $t\in[t\no,t\nn)$.
\labitem{\textbf{(D1)}}{item:increments} $\sinc{n}$ is $\mathcal{F}_{t\nn}^\tau$-measurable and independent of $\mathcal{F}_{t^m}^\tau$ for all $0\leq m\leq n-1$.
\labitem{\textbf{(D2)}}{item:randomvariables} There exist mutually independent symmetric random variables $\xi_k^{n,\tau}$ such that
\begin{align*}
\sinc{n}=\sqrt{\tau}\sum_{k\in\mathds{Z}}\lambda_k\g{k}\xi_k^{n,\tau}\,,
\end{align*}
where $\expected{\xi_k^{n,\tau}}=0$, $\expected{\tabs{\xi_k^{n,\tau}}^2}=1$, and $\expected{\tabs{\xi_k^{n,\tau}}^p}\leq C_p$ for all integer $p\geq2$ with a constant $C_p$ depending on the exponent $p$ but not on $h$, $\tau$, $k$, or $n$.
\end{itemize}
Here, $\trkla{\g{k}}_{k\in\mathds{Z}}$ are an orthonormal basis of $L^2\trkla{\Om}$ and $\trkla{\lambda_k}_{k\in\mathds{Z}}$ are given real numbers such that
\begin{itemize}
\labitem{\textbf{(D3)}}{item:color} the noise $\sinc{n}$ is colored in the sense that \eqref{eq:Wcolor} in assumption \ref{item:Wcolor} is satisfied for a positive constant $\widetilde{C}$ that is independent of $h$ and $\tau$.
\end{itemize}

\begin{remark}\label{rem:eigenfunctions}
Assumption \ref{item:color} is similar to the assumptions used in \cite{LiuQiao2021, BreitProhl2024} and implies that also spatial derivatives of the $\mathcal{Q}$-Wiener process can be controlled.
This is necessary to establish suitable a priori estimates:
To overcome the loss of temporal regularity of the solution caused by the stochastic term on the right-hand side of \eqref{eq:model:phi}, one typically multiplies the (discrete) equation pathwise by the variational derivative of the energy (see e.g.~\cite{Ondrejat2022}).
In case of the stochastic Allen--Cahn equation the Helmholtz free energy includes $\tfrac12\iO \tabs{\nabla\phi}^2\dx$ (cf.~\eqref{eq:Helmholtz}).
Hence, we need to multiply our discrete equations by an approximation of $-\Delta\phi$ which at this point has to be understood as an element of $H^1\trkla{\Om}^\prime$.
Consequently, the right-hand side of \eqref{eq:model:phi} needs to be in $H^1\trkla{\Om}$.
These considerations are made rigorous in Lemma \ref{lem:energy} below.\\ 
In Section \ref{sec:numerics}, we will consider the two-dimensional torus and define the functions $\g{k}\equiv\g{k\trkla{l,m}}$ as products of eigenfunctions of the one-dimensional Laplacian with periodic boundary conditions, i.e.
\begin{subequations}
\begin{align}
g_l^x:=&\sqrt{\frac{2}{L_x}}\left\{\begin{array}{cl}
\cos\rkla{\frac{2\pi lx}{L_x}}&\text{for~}l\geq1\,,\\
\frac{1}{\sqrt{2}}&\text{for~}l=0\,,\\
\sin\rkla{\frac{2\pi lx}{L_x}}&\text{for~}l\leq-1\,,
\end{array}\right.\\
g_m^y:=&\sqrt{\frac{2}{L_y}}\left\{\begin{array}{cl}
\cos\rkla{\frac{2\pi my}{L_y}}&\text{for~}m\geq1\,,\\
\frac{1}{\sqrt{2}}&\text{for~}m=0\,,\\
\sin\rkla{\frac{2\pi my}{L_y}}&\text{for~}m\leq-1\,,
\end{array}\right.
\end{align}
\end{subequations}
where $L_x$ and $L_y$ denote the length of $\Om$ in $x$- and $y$-direction.
In this case, we have $\norm{\g{k\trkla{l,m}}}_{W^{1,\infty}\trkla{\Om}}\sim\trkla{l+m}$.
Hence, choosing $\lambda_{k\trkla{l,m}}\sim \trkla{lm}^{-2}$ is sufficient to guarantee \ref{item:color}.
\end{remark}

Summing $\sinc{n}$ from $n=1$ to $m$ provides the discrete approximation $\bs{\xi}^{m,\tau}$ of the $\mathcal{Q}$-Wiener process $W$.
We want to emphasize that the assumptions in \ref{item:randomvariables} are more general than considering a $\mathcal{Q}$-Wiener process satisfying \ref{item:stochastic}.
If we evaluate the $\mathcal{Q}$-Wiener process at given times $t\no$ and $t\nn$, we obtain a piecewise constant approximation where the mutually independent random variables $\xi_k^{n,\tau}$ satisfy \ref{item:randomvariables} and are in addition $\mathcal{N}\trkla{0,1}$-distributed.
The more general assumption \ref{item:randomvariables} on the other hand allows for bounded increments satisfying e.g.~$\tabs{\xi_k^{n,\tau}}\leq \nu$ for a suitable constant $\nu\geq0$.
Such an approximation was used by \cite{GrillmeierGruen2019} to establish the non-negativity of discrete solutions to the stochastic porous-medium equations.\\
We will approximate $\Phi$ by $\Phi\h$ which is defined via
\begin{align}\label{eq:defPhih}
\Phi\h\trkla{\zeta\h}f:=\sum_{k\in\Zh}\Ih{\sigma\trkla{\zeta\h}\trkla{f,\g{k}}_{L^2\trkla{\Om}}\g{k}}
\end{align}
for all $f\in\mathcal{Q}^{1/2}L^2\trkla{\Om}$ and $\zeta\h\in\Uh$.
Here, $\Zh$ is an $h$-dependent finite subset of $\mathds{Z}$ satisfying $\mathds{Z}_{h_1}\subset\mathds{Z}_{h_2}$ for $h_1\geq h_2$ and $\bigcup_{h>0}\Zh=\mathds{Z}$.
This results in only a finite sum of modes entering the discrete scheme for each given $h$.
Hence, we introduce a finite dimensional approximation of the $\mathcal{Q}$-Wiener process defined by
\begin{align}\label{eq:deffinteprocess}
\bs{\xi}\h^{m,\tau}:=\sum_{n=1}^m\sqrt{\tau}\sum_{k\in\mathds{Z}\h}\lambda_k\g{k}\xi_k^{n,\tau}\,.
\end{align}
Definition \eqref{eq:defPhih} also entails that $\Phi\h\trkla{\zeta\h}f$ is an element of $\Uh$.
If we want to emphasize this fact, we will write $\Ih{\Phi\h\trkla{\zeta\h}f}$ although the application of $\Ihop$ is superfluous.\\
In order to control higher moments of the (discrete) stochastic integrals, we will need the following estimates which are proven in Appendix \ref{sec:prelim}:
\begin{lemma}\label{lem:BDG}
Let $\trkla{\sinc{n}}_{n=1,\ldots,N}$ be a sequence of discrete stochastic increments satisfying \ref{item:filtration}-\ref{item:color}.
Furthermore, let $\trkla{\Phi\h\no}_{n=1,\ldots,N}$ be a sequence of $\trkla{\mathcal{F}_{t\no}^\tau}_{n=1,\ldots,N}$-measurable mappings from $\Omega$ to the set of Hilbert--Schmidt operators mapping $\mathcal{Q}^{1/2}L^2\trkla{\Om}$ to a separable Hilbert space $H$, i.e.~$\Phi\h\no\,:\,\Omega\rightarrow L_2\trkla{\mathcal{Q}^{1/2}L^2\trkla{\Om};H}$, satisfying $\Phi\h\no \g{k}=0$ for all $k\notin\Zh$ (cf.~\eqref{eq:defPhih}).
Then, for $p\in[1,\infty)$ the estimates
\begin{subequations}
\begin{align}\label{eq:bdg1}
\expected{\norm{\Phi\h\no\sinc{n}}_H^p}\leq&\, C_p\tau^{p/2}\expected{\rkla{\sum_{k\in\Zh}\norm{\lambda_k\Phi\h\no\g{k}}_H^2}^{p/2}}\,,\\
\label{eq:bdg2}\expected{\max_{1\leq l\leq m }\norm{\sum_{n=1}^l \Phi\h\no\sinc{n}}_H^p}\leq &\,C_p\rkla{m\tau}^{\tfrac{p-2}2}\sum_{n=1}^m\tau\expected{\rkla{\sum_{k\in\Zh}\norm{\lambda_k\Phi\h\no\g{k}}_H^2}^{p/2}}
\end{align}
\end{subequations}
hold true with a constant $C_p>0$ which depends on $p$, but is independent of $h$ and $\tau$.
\end{lemma}

In order to pass to the limit $\trkla{h,\tau}\rightarrow\trkla{0,0}$ with families of fully discrete random variables, we define time-interpolants of a time-discrete function $a\nn$, $n=0,\ldots,N$, and introduce some time-index-free notation as follows:
\begin{subequations}\label{eq:deftimeinterpol}
\begin{align}
a\tl\trkla{.,t}&:=\tfrac{t-t\no}{\tau}a\nn\trkla{.}+\tfrac{t\nn-t}{\tau}a\no\trkla{.}&&t\in\tekla{t\no,t\nn},\,n\geq1\,,\\
a\tm\trkla{.,t}&:=a\no\trkla{.}&&t\in[t\no,t\nn),\,n\geq1\,,\\
a\tp\trkla{.,t}&:=a\nn\trkla{.}&&t\in(t\no,t\nn],\,n\geq1\,.
\end{align}
\end{subequations}
For completeness, we define $a\tp\trkla{.,0}:=a^0\trkla{.}$.
Obviously, we have
\begin{align}
&&a\tl\trkla{.,t}-a\tm\trkla{.,t}&=\tfrac{t-t\no}{\tau}\trkla{a\nn-a\no}&&\forall t\in[t\no,t\nn)\,.
\end{align}
If a statement is valid for all three time-interpolants defined in \eqref{eq:deftimeinterpol}, we shall use the abbreviation $a\tpm$.

\section{The discrete scheme}\label{sec:discscheme}
For the ease of representation, we set $\varepsilon=1$ and define the abbreviation
\begin{align}
E\h\trkla{\zeta\h}:=\iO\Ih{F\trkla{\zeta\h}}\dx
\end{align}
for $\zeta\h\in\Uh$.
To derive a linear discrete scheme for \eqref{eq:model}, we introduce a stochastic scalar auxiliary variable $r\,:\,\Omega\times\tekla{0,T}\rightarrow\mathds{R}$ with $r\trkla{\omega,t}=\sqrt{\iO F\trkla{\phi\trkla{\omega,t,x}}\dx}$ as the square root of the non-quadratic parts of the energy and approximate it using a sequence of random variables $\trkla{r\h\nn}_{n=0,\ldots,N}$ which are supposed to approximate $\trkla{\sqrt{E\h\trkla{\phi\h\nn}}}_{n=0,\ldots,N}$.
Defining the discrete initial data via $\phi\h^0:=\mathcal{P}\h\phi_0$ (cf.~Assumption \ref{item:initial}) and $r\h^0:=\sqrt{E\h\trkla{\phi\h^0}}$, we state the following discrete scheme:\\
For given $\Uh\times\mathds{R}$-valued random variables $\trkla{\phi\h\no,r\h\no}$, find an $\Uh\times\mathds{R}$-valued random variable $\trkla{\phi\h\nn,r\h\nn}$ such that pathwise
\begin{subequations}\label{eq:discscheme}
\begin{multline}\label{eq:discscheme:phi}
\iO\Ih{\trkla{\phi\h\nn-\phi\h\no}\psi\h}\dx+\tau\iO\nabla\phi\h\nn\cdot\nabla\psi\h\dx+\tau\frac{r\h\nn}{\sqrt{E\h\trkla{\phi\h\no}}}\iO\Ih{F^\prime\trkla{\phi\h\no}\psi\h}\dx\\
-\tfrac14\tau\frac{r\h\nn}{\tekla{E\h\trkla{\phi\h\no}}^{3/2}}\iO\Ih{F^\prime\trkla{\phi\h\no}\Phi\h\trkla{\phi\h\no}\sinc{n}}\dx\iO\Ih{F^\prime\trkla{\phi\h\no}\psi\h}\dx\\
+\tfrac12\tau\frac{r\h\nn}{\sqrt{E\h\trkla{\phi\h\no}}}\iO\Ih{F^{\prime\prime}\trkla{\phi\h\no}\Phi\h\trkla{\phi\h\no}\sinc{n}\psi\h}\dx = \iO\Ih{\Phi\h\trkla{\phi\h\no}\sinc{n}\psi\h}\dx\,,
\end{multline}
for all $\psi\h\in\Uh$ and 
\begin{multline}\label{eq:discscheme:r}
r\h\nn-r\h\no=\frac{1}{2\sqrt{E\h\trkla{\phi\h\no}}}\iO\Ih{F^\prime\trkla{\phi\h\no}\trkla{\phi\h\nn-\phi\h\no}}\dx\\
-\frac{1}{8\tekla{E\h\trkla{\phi\h\no}}^{3/2}}\iO\Ih{F^\prime\trkla{\phi\h\no}\Phi\h\trkla{\phi\h\no}\sinc{n}}\dx\iO\Ih{F^\prime\trkla{\phi\h\no}\trkla{\phi\h\nn-\phi\h\no}}\dx\\
+\frac{1}{4\sqrt{E\h\trkla{\phi\h\no}}}\iO\Ih{F^{\prime\prime}\trkla{\phi\h\no}\trkla{\phi\h\nn-\phi\h\no}\Phi\h\trkla{\phi\h\no}\sinc{n}}\dx\,.
\end{multline}
\end{subequations}

To simplify the notation, we introduce the chemical potential
\begin{align}\label{eq:defmu}
\begin{split}
\mu\h\nn:=&-\Delta\h\phi\h\nn+\frac{r\h\nn}{\sqrt{E\h\trkla{\phi\h\no}}}\Ih{F^\prime\trkla{\phi\h\no}}\\
&-\frac{r\h\nn}{4\tekla{E\h\trkla{\phi\h\no}}^{3/2}}\iO\Ih{F^\prime\trkla{\phi\h\no}\Phi\h\trkla{\phi\h\no}\sinc{n}}\dx\,\Ih{F^\prime\trkla{\phi\h\no}}\\
&+\frac{r\h\nn}{2\sqrt{E\h\trkla{\phi\h\no}}}\Ih{F^{\prime\prime}\trkla{\phi\h\no}\Phi\h\trkla{\phi\h\no}\sinc{n}}\,\\
=:&\,-\Delta\h\phi\h\nn+\frac{r\h\nn}{\sqrt{E\h\trkla{\phi\h\no}}}\Ih{F^\prime\trkla{\phi\h\no}} +\Xi\h\nn\,
\end{split}
\end{align}
with the discrete Laplacian $\Delta\h$ being defined in \eqref{eq:def:discLaplacian} and
\begin{align}\label{eq:def:errormu}
\begin{split}
\Xi\h\nn:=&-\frac{r\h\nn}{4\tekla{E\h\trkla{\phi\h\no}}^{3/2}}\iO\Ih{F^\prime\trkla{\phi\h\no}\Phi\h\trkla{\phi\h\no}\sinc{n}}\dx\,\Ih{F^\prime\trkla{\phi\h\no}}\\
&+\frac{r\h\nn}{2\sqrt{E\h\trkla{\phi\h\no}}}\Ih{F^{\prime\prime}\trkla{\phi\h\no}\Phi\h\trkla{\phi\h\no}\sinc{n}}\,
\end{split}
\end{align}
denoting additional terms added to the standard SAV method to ensure convergence of the scheme (cf.~Remark \ref{rem:taylor} below and Lemma \ref{lem:SAVerror} for rigorous computations).

With this definition \eqref{eq:discscheme:phi} simplifies to
\begin{align}\label{eq:modeldisc:phimu}
\iO\Ih{\trkla{\phi\h\nn-\phi\h\no}\psi\h}\dx+\tau\iO\Ih{\mu\h\nn\psi\h}\dx=\iO\Ih{\Phi\h\trkla{\phi\h\no}\sinc{n}\psi\h}\dx\,.
\end{align}
\begin{remark}\label{rem:taylor}
In the original derivation of the SAV method (cf.~\cite{ShenXuYang2018}), the evolution of the scalar auxiliary variable was approximated by
\begin{align}\label{eq:SAVshort}
r\h\nn-r\h\no=\frac{1}{2\sqrt{E\h\trkla{\phi\h\no}}}\iO\Ih{F^\prime\trkla{\phi\h\no}\trkla{\phi\h\nn-\phi\h\no}}\dx\,,
\end{align}
i.e.~without the last two terms on the right-hand side of \eqref{eq:discscheme:r}.
As the scalar auxiliary variable was originally introduced as the square root of the nonlinear part of the energy, i.e.~ $r=\sqrt{\iO F\trkla{\phi}\dx}$, this approximation was motivated by the formal computations 
\begin{align}
\frac{\mathrm{d}}{\mathrm{d}t} r =\frac{\mathrm{d}}{\mathrm{d}t} \sqrt{\iO F\trkla{\phi}\dx} = \frac{1}{2\sqrt{\iO F\trkla{\phi}\dx}}\iO F^\prime\trkla{\phi}\para{t}\phi\dx\,.
\end{align}
In order to prove convergence, one could interpret the right-hand side of \eqref{eq:SAVshort} as an approximation of $\sqrt{E\h\trkla{\phi\h\nn}}-\sqrt{E\h\trkla{\phi\h\no}}$.
Assuming sufficient regularity, a Taylor expansion suggests that the error of this approximation depends on $\trkla{\phi\h\nn-\phi\h\no}^2$.
Hence, roughly speaking, if the time-interpolation of the discrete solution $\trkla{\phi\h\nn}_n$ is Hölder continuous with an exponent greater than $1/2$ (or uniformly bounded in a corresponding Nikolskii space), the discretization error scaling with $\sum_{n=1}^N\trkla{\phi\h\nn-\phi\h\no}^2$ will vanish.\\
In the stochastic setting, however, this argument does not work anymore, as the regularity restrictions caused by the random walk approximating the Wiener process will only allow for exponents that do not exceed $1/2$.
In order to still obtain a convergent scheme, we augmented \eqref{eq:SAVshort} using higher order terms from the Taylor expansion of $\sqrt{E\h\trkla{\phi\h\nn}}$.
Consequently, the discretization error will scale with approximately $\sum_{n=1}^N\tabs{\phi\h\nn-\phi\h\no}^3$ allowing for convergence even for less regular solutions.
These additional terms, however include second powers of $\phi\h\nn-\phi\h\no$ and, hence, destroy the linearity of the discrete scheme.
To recover the linearity, we replaced one factor $\trkla{\phi\h\nn-\phi\h\no}$ by $\Phi\h\trkla{\phi\h\no}\sinc{n}$ and ended up with \eqref{eq:discscheme:r}.
As we will show in Lemma \ref{lem:SAVerror} below, this additional linearization does not destroy the convergence argument.
Hence, we still have $\abs{r\h\nn-\sqrt{E\h\trkla{\phi\h\nn}}}\rightarrow0$.\\
In order to preserve stability of the scheme, we need to also include the additional terms $\Xi\h\nn$ in \eqref{eq:discscheme:phi}.
In Corollary \ref{cor:muerror}, we will show that these additional terms also vanish when passing to the limit.
Yet, these additional terms are not only important for the analytical treatment of the scheme, but also affect the quality of the numerical simulations significantly.
As we will show in Section \ref{sec:numerics}, neglecting these higher order approximations will lead to wrong results even for very small time-increments.
\end{remark}

\begin{remark}
The assumptions on the initial data stated in \ref{item:initial} are sufficient to obtain convergence of the discrete initial data $\trkla{\phi\h^0,r\h^0}$.
In particular, one can show that
\begin{align*}
\phi\h^0&\rightarrow\phi_0&\text{in~}L^q\trkla{\Om}\text{~for~}q\in[1,\tfrac{2d}{d-2})\,,\\
r\h^0&\rightarrow \sqrt{\iO F\trkla{\phi_0}\dx}\,.
\end{align*}
While the convergence of $\phi\h^0$ is a direct result of \ref{item:initial}, the standard Sobolev embeddings, and Vitali's convergence theorem, the convergence of $r\h^0$ follows from
\begin{multline*}
\abs{r\h^0-\sqrt{\iO F\trkla{\phi_0}\dx}}\leq \abs{r\h^0-\sqrt{\iO F\trkla{\phi\h^0}\dx}} +\abs{\sqrt{\iO F\trkla{\phi\h^0}\dx}-\sqrt{\iO F\trkla{\phi_0}\dx}}\\
\leq \frac{1}{r\h^0+\sqrt{\iO F\trkla{\phi\h^0}\dx}}\abs{\iO\trkla{1-\Ihop}\gkla{F\trkla{\phi\h^0}}}\\
+\frac{1}{\sqrt{\iO F\trkla{\phi\h^0}\dx}+\sqrt{\iO F\trkla{\phi_0}\dx}}\abs{\iO F\trkla{\phi\h^0}-F\trkla{\phi_0}\dx}\,.
\end{multline*}
Due to the lower bound on $F$, both coefficients are bounded by a constant depending on $\gamma^{-1}$.
Due to the growth condition of $F$ and the convergence of $\phi\h^0$, the second term on the right-hand side vanishes.
As shown in (4.8) in \cite{Metzger2023}, the first term is bounded by $C h\rkla{\tabs{\Om}+\norm{\phi\h^0}_{H^1\trkla{\Om}}^6}$ and will therefore also vanish.
\end{remark}
\section{Main results}\label{sec:mainresults}
In this section, we shall state the main results which will be proven in the subsequent sections.
We start by establishing the existence of discrete solutions:
\begin{lemma}\label{lem:existence}
Let assumptions \ref{item:spatialdisc}, \ref{item:timedisc}, \ref{item:initial}, and \ref{item:filtration} hold true.
Then, there exists a sequence $\trkla{\phi\h\nn,r\h\nn}_{n\geq1}$ of $\Uh\times\mathds{R}$-valued random variables, which solves \eqref{eq:discscheme} for each $\omega\in\Omega$, such that the map $\trkla{\phi\h\nn,r\h\nn}\,:\,\Omega\rightarrow\Uh\times\mathds{R}$ is $\mathcal{F}_{t\nn}$-measurable.
\end{lemma}
We shall then prove the following theorem stating the convergence towards pathwise unique martingale solutions:
\begin{theorem}\label{thm:maintheorem}
Let the assumptions \ref{item:spatialdisc}, \ref{item:timedisc}, \ref{item:potentialF}, \ref{item:initial}, \ref{item:sigma}, \ref{item:filtration}, \ref{item:increments}, \ref{item:randomvariables}, and \ref{item:color}  hold true.
Then, there exists a filtered probability space $\trkla{\widetilde{\Omega},\widetilde{\mathcal{A}},\widetilde{\mathcal{F}},\widetilde{\Prob}}$ and a sequence of random variables $\trkla{\widetilde{\phi}\h\tl}_{h,\tau}$ on $\widetilde{\Omega}$ whose laws coincide with the laws of linear time-interpolants $\trkla{{\phi}\h\tl}$ of the discrete solutions to \eqref{eq:discscheme} established in Lemma \ref{lem:existence}.
Furthermore, there exists a $\widetilde{\mathcal{F}}$-measurable $\mathcal{Q}$-Wiener process $\widetilde{W}$ on $\widetilde{\Omega}$ and a progressively $\widetilde{\mathcal{F}}$-measurable process
\begin{align*}
\widetilde{\phi}\in L^{2\p}_{\operatorname{weak-}\trkla{*}}&\trkla{\widetilde{\Omega};L^\infty\trkla{0,T;H^1\trkla{\Om}}}\cap L^{2\p}\trkla{\widetilde{\Omega};C^{0,\trkla{\p-1}/\trkla{2\p}}\trkla{\tekla{0,T};L^2\trkla{\Om}}}\\
&\cap L^{2\p}\trkla{\widetilde{\Omega};L^2\trkla{0,T;H^2\trkla{\Om}}}
\end{align*}
for any $\p\in\trkla{1,\infty}$ such that $\widetilde{\Prob}$-almost surely
\begin{align*}
\lim_{\trkla{h,\tau}\rightarrow\trkla{0,0}}\widetilde{\phi}\h\tl=\widetilde{\phi}&&\text{in~}C\trkla{\tekla{0,T};L^s\trkla{\Om}}
\end{align*}
for $s\in[1,\tfrac{2d}{d-2})$.
This process is pathwise unique and satisfies
\begin{align}\label{eq:solution}
\widetilde{\phi}\trkla{t}-\phi_0+\int_0^t \rkla{-\Delta\widetilde{\phi}+F^\prime\trkla{\widetilde{\phi}}}\ds=\int_0^t\Phi\trkla{\widetilde{\phi}}\mathrm{d}\widetilde{W}
\end{align}
$\widetilde{\Prob}$-a.s. in $L^2\trkla{\Om}$ for all $t\in\tekla{0,T}$.
\end{theorem}

\begin{proof}[Sketch of the proof]
The proof of Theorem \ref{thm:maintheorem} can be found in Sections \ref{sec:regularity}-\ref{sec:uniqueness}.
For the reader's convenience we provide a short overview:\\
In Section \ref{sec:regularity}, we start by establishing $\trkla{h,\tau}$-independent regularity results for the family of random variables satisfying \eqref{eq:discscheme}.
In particular, we establish a modified version of the energy estimate in Lemma \ref{lem:energy} which provides control over the full $H^1\trkla{\Om}$-norm.
In Corollary \ref{cor:h2} we improve these results by establishing a discrete $L^q\trkla{\Omega;L^2\trkla{0,T;L^2\trkla{\Om}}}$-bound ($q\in[1,\infty)$) on the discrete Laplacian of $\phi\h\nn$.
To obtain improved regularity with respect to time, we shall show that the phase-field parameter is uniformly bounded in a suitable Nikolskii space (cf.~Lemma \ref{lem:nikolskii}).
Section \ref{sec:regularity} is concluded by an estimate on the approximation error of the scalar auxiliary variable.\\
In Section \ref{sec:compactness}, we apply Jakubowski's theorem and deduce the existence of a sequence of random variables defined on a different probability space $\trkla{\widetilde{\Omega},\widetilde{\mathcal{A}},\widetilde{\Prob}}$ that converges $\widetilde{\Prob}$-almost surely and whose laws coincide with the laws of the time-interpolants of the fully discrete solutions (cf.~Theorem \ref{thm:jakubowski}).
By combining this convergence result with the already established uniform bounds, we can deduce additional convergence results which are collected in Lemma \ref{lem:convergence}.\\
In Section \ref{sec:limit}, we pass to the limit $\trkla{h,\tau}\rightarrow\trkla{0,0}$ and establish convergence towards martingale solutions.
The main difficulty of this step is showing that the right-hand side of \eqref{eq:discscheme:phi} converges towards an \Ito-integral driven by a $\mathcal{Q}$-Wiener process.
For this reason, we show in Lemma \ref{lem:martingale} that the limit process is a martingale with respect to the filtration $\trkla{\widetilde{\mathcal{F}}_t}_{t\in\tekla{0,T}}$ and derive expressions for its quadratic and cross variation processes in Lemma \ref{eq:variations}.
This allows us to follow the approach introduced in \cite{BrzezniakOndrejat, Ondrejat2010, HofmanovaSeidler} (see also \cite{HofmanovaRoegerRenesse,BreitFeireislHofmanova, Ondrejat2022}) and to show that this martingale can be written as an \Ito-integral with the Wiener process $\widetilde{W}$ (cf.~Lemma \ref{lem:itointegral}).\\
The pathwise uniqueness of the obtained martingale solutions is established in Section \ref{sec:uniqueness} using a local monotonicity argument.
The pathwise uniqueness also allows us to extend the convergence results stated in the prior sections for subsequences to the complete sequence.
\end{proof}

\begin{remark}
As the martingale solutions obtained in Theorem \ref{thm:maintheorem} are pathwise unique, the Yamada--Watanabe theorem provides the existence of strong solutions (cf.~\cite{RoecknerSchmulandZhang2008} or Theorem E.0.8 in \cite{LiuRoeckner}).
\end{remark}
If we approximate a given $\mathcal{Q}$-Wiener process $W$ satisfying \ref{item:stochastic} and \ref{item:Wcolor} via
\begin{align}\label{eq:approxW}
\bs{\xi}\h^{m,\tau}=\sum_{k\in\mathds{Z}\h}\trkla{W\trkla{t^m},\g{k}}_{L^2\trkla{\Om}}\g{k}\,,
\end{align}
$\bs{\xi}_h^{m,\tau}$ still satisfies \ref{item:filtration}-\ref{item:color} with Gaussian random variables $\xi_k^{n,\tau}$. 
For this choice, we can use the pathwise uniqueness established in Theorem \ref{thm:maintheorem} to establish convergence towards strong solutions:
\begin{theorem}\label{thm:strongconvergence}
Let $W$ be a Wiener process defined on $\trkla{\Omega,\mathcal{A},\mathcal{F},\Prob}$ satisfying \ref{item:stochastic} and \ref{item:Wcolor} with finite dimensional approximations given by \eqref{eq:approxW} and let the assumptions \ref{item:spatialdisc}, \ref{item:timedisc}, \ref{item:potentialF}, \ref{item:initial}, \ref{item:sigma}, \ref{item:filtration}, \ref{item:increments}, and \ref{item:randomvariables} hold true.
Then, there exists a unique, progressively $\mathcal{F}$-measurable process
\begin{align*}
\phi\in L^{2\p}_{\operatorname{weak-}\trkla{*}}&\trkla{\Omega;L^\infty\trkla{0,T;H^1\trkla{\Om}}}\cap L^{2\p}\trkla{\Omega;C^{0,\trkla{\p-1}/\trkla{2\p}}\trkla{\tekla{0,T};L^2\trkla{\Om}}}\\
&\cap L^{2\p}\trkla{\Omega;L^2\trkla{0,T;H^2\trkla{\Om}}}
\end{align*}
for any $\p\in\trkla{1,\infty}$, which is the limit of the linear time-interpolants $\trkla{\phi\h\tl}_{h,\tau}$ of the discrete solutions to \eqref{eq:discscheme} established in Lemma \ref{lem:existence}.
In particular, we have $\Prob$-almost surely
\begin{align*}
\lim_{\trkla{h,\tau}\rightarrow\trkla{0,0}}\phi\h\tl=\phi&&\text{in~}C\trkla{\tekla{0,T};L^s\trkla{\Om}}
\end{align*}
for $s\in[1,\tfrac{2d}{d-2})$. 
This process satisfies
\begin{align*}
\phi\trkla{t}-\phi_0+\int_0^t\rkla{-\Delta\phi}+F^\prime\trkla{\phi}\ds=\int_0^t\Phi\trkla{\phi}\mathrm{d}W
\end{align*}
$\Prob$-almost surely in $L^2\trkla{\Om}$ for all $t\in\tekla{0,T}$.
\end{theorem}
The proof of this theorem relies on a generalization of the Gyöngy--Krylov characterization of convergence in probability which was established in \cite{BreitFeireislHofmanova} and can be found in Section \ref{sec:strong}.

\section{Existence of discrete solutions}\label{sec:discexistence}
In this section, we shall establish the existence of discrete solutions to \eqref{eq:discscheme}, i.e.~we prove Lemma \ref{lem:existence}.\\
We start by showing that for each fixed $\omega\in\Omega$, \eqref{eq:discscheme} has a unique solution:
Since $\sinc{n}$ is given for fixed $\omega$, \eqref{eq:discscheme} is linear with respect to the unknown quantities $\phi\h\nn$, $r\h\nn$.
As for finite dimensional linear problems, uniqueness of possible solutions guarantees the existence of solutions for arbitrary right-hand sides, we assume that for given $\phi\h\no$, $r\h\no$, and $\sinc{n}$ there exist two solutions.
We denote their difference by $\hat{\phi}$ and $\hat{r}$.
Obviously, these differences satisfy
\begin{multline}\label{eq:existence1}
\iO\Ih{\hat{\phi}\psi\h}\dx +\tau\iO\nabla\hat{\phi}\cdot\nabla\psi\h\dx+\tau\frac{\hat{r}}{\sqrt{E\h\trkla{\phi\h\no}}}\iO\Ih{F^\prime\trkla{\phi\h\no}\psi\h}\dx\\
-\tfrac14\tau\frac{\hat{r}}{\tekla{E\h\trkla{\phi\h\no}}^{3/2}}\iO\Ih{F^\prime\trkla{\phi\h\no}\Phi\h\trkla{\phi\h\no}\sinc{n}}\dx\iO\Ih{F^\prime\trkla{\phi\h\no}\psi\h}\dx\\
+\tfrac12\tau\frac{\hat{r}}{\sqrt{E\h\trkla{\phi\h\no}}}\iO\Ih{F^{\prime\prime}\trkla{\phi\h\no}\Phi\h\trkla{\phi\h\no}\sinc{n}\psi\h}\dx=0\,,
\end{multline}
for all $\psi\h\in\Uh$ and
\begin{align}\label{eq:existence2}
\begin{split}
\hat{r}=&\,\frac{1}{2\sqrt{E\h\trkla{\phi\h\no}}}\iO\Ih{F^\prime\trkla{\phi\h\no}\hat{\phi}}\dx\\
&-\frac{1}{8\tekla{E\h\trkla{\phi\h\no}}^{3/2}}\iO\Ih{F^\prime\trkla{\phi\h\no}\Phi\h\trkla{\phi\h\no}\sinc{n}}\dx\iO\Ih{F^\prime\trkla{\phi\h\no}\hat{\phi}}\dx\\
&+\frac{1}{4\sqrt{E\h\trkla{\phi\h\no}}}\iO\Ih{F^{\prime\prime}\trkla{\phi\h\no}\hat{\phi}\Phi\h\trkla{\phi\h\no}\sinc{n}}\dx\,.
\end{split}
\end{align}
Choosing $\psi\h\equiv\hat{\phi}\trkla{\omega}$ in \eqref{eq:existence1} and applying \eqref{eq:existence2}, we obtain
\begin{align}
\iO\Ih{\tabs{\hat{\phi}}^2}\dx+\tau\iO\tabs{\nabla\hat{\phi}}^2\dx+2\tau\tabs{\hat{r}}^2=0\,,
\end{align}
i.e.~uniqueness and therefore existence of solutions.
The uniqueness of the solution for each $\omega\in\Omega$ entails its measurability (see e.g.~Theorem 6.7 in \cite{Grillmeier2020}).
Hence, Lemma \ref{lem:existence} is proven.

\section{Regularity results}\label{sec:regularity}
In this section, we shall establish uniform regularity results for the discrete solutions.
Our first result is based on the unconditional stability of the scheme with respect to the modified SAV energy.
In the absence of force terms, one can (pathwise) choose $\psi\h=\mu\h\nn$ in \eqref{eq:discscheme:phi} (or \eqref{eq:modeldisc:phimu}, respectively).
Substituting \eqref{eq:discscheme:r} provides
\begin{multline}\label{eq:formalenergy}
\tfrac12\norm{\nabla\phi\h\nn}_{L^2\trkla{\Om}}^2 + \tabs{r\h\nn}^2 +\tfrac12\norm{\nabla\phi\h\nn-\nabla\phi\h\no}_{L^2\trkla{\Om}}^2+\tabs{r\h\nn-r\h\no}^2 +\tau\norm{\mu\h\nn}\h^2\\
\leq \tfrac12 \norm{\nabla\phi\h\no}_{L^2\trkla{\Om}}^2+\tabs{r\h\no}^2\,,
\end{multline}
i.e.~the SAV energy $\tfrac12\norm{\nabla\phi\h\nn}_{L^2\trkla{\Om}}^2+\tabs{r\h\nn}^2$ is non-increasing over time.
As $\tabs{r\h\nn}^2\neq \iO\Ih{F\trkla{\phi\h\nn}}\dx$, this result does not guarantee that the original energy $\iO\tfrac12\tabs{\nabla\phi\h\nn}^2 +\Ih{F\trkla{\phi\h\nn}}\dx$ is also non-increasing.
This in particular entails that the SAV energy does not provide any control over the deviation of the phase-field parameter from $\pm1$.
Hence, despite \eqref{eq:formalenergy} being valid, oscillations may occur if the time-step size is too large.
From the analytical point of view, however, \eqref{eq:formalenergy} still allows to control the solution by constants depending only on the given data.
A further downside of the modified energy is the missing bound on the $L^2\trkla{\Om}$-norm of $\phi\h\nn$, as \eqref{eq:formalenergy} only contains the $H^1\trkla{\Om}$-seminorm.
To overcome this obstacle, we choose $\psi\h=\phi\h\nn$ in \eqref{eq:modeldisc:phimu} and obtain via Young's inequality
\begin{align}\label{eq:formall2}
\tfrac12\norm{\phi\h\nn}\h^2+\tfrac12\norm{\phi\h\nn-\phi\h\no}\h^2\leq \tfrac12\norm{\phi\h\no}\h^2 +\tfrac34\tau\norm{\mu\h\nn}\h^2+\tfrac13\tau\norm{\phi\h\nn}\h^2\,.
\end{align}
Hence, combining \eqref{eq:formalenergy} and \eqref{eq:formall2} and applying a discrete version of Gronwall's inequality allows us to estimate the $H^1\trkla{\Om}$-norm of $\phi\h\nn$ by a constant depending only on the given data.
In the presence of (stochastic) source terms, these estimates are more intricate as we also need to estimate the source terms using Gronwall's inequality.
The arguments showing that the proposed scheme is stable in the sense that the $H^1\trkla{\Om}$-norm of $\phi\h\nn$ and the modulus of $r\h\nn$ are bounded by a constant depending only on the given data are made rigorous in the following lemma:

\begin{lemma}\label{lem:energy}
Let the assumptions \ref{item:timedisc}, \ref{item:spatialdisc}, \ref{item:potentialF}, \ref{item:initial}, \ref{item:sigma}, and \ref{item:filtration}--\ref{item:color} hold true.
Then, for every $1\leq\p<\infty$, there exists a constant $C\equiv C\trkla{\p,T}>0$ independent of $h$ and $\tau$ such that
\begin{multline*}
\expected{\max_{0\leq n\leq N}\norm{\phi\h^n}_{H^1\trkla{\Om}}^{2\p}}+\expected{\max_{0\leq n\leq N}\tabs{r\h^n}^{2\p}}\\
+\expected{\rkla{\sum_{n=1}^N\norm{\phi\h\nn-\phi\h\no}_{H^1\trkla{\Om}}^2}^\p}+ \expected{\rkla{\sum_{n=1}^N\tabs{r\h\nn-r\h\no}^2}^\p}\\
+\expected{\rkla{\sum_{n=1}^N\tau\norm{\mu\h\nn}_{L^2\trkla{\Om}}^2}^\p}\leq C\,.
\end{multline*}
\end{lemma}
\begin{proof}
In a first step, we shall prove
\begin{align}\label{eq:energyaux}
\max_{0\leq n\leq N}\expected{\norm{\phi\h\nn}_{H^1\trkla{\Om}}^{2\p}}+\max_{0\leq n\leq N}\expected{\tabs{r\h\nn}^{2\p}}\leq C\,.
\end{align}
For fixed $\omega\in\Omega$, we test \eqref{eq:discscheme:phi} (or \eqref{eq:modeldisc:phimu}, respectively) by $\mu\h\nn$ and multiply \eqref{eq:discscheme:r} by $r\h\nn$ to obtain
\begin{align}\label{eq:energy1}
\!\!\!\!\!\!\!\!\begin{split}
\tfrac12&\norm{\nabla\phi\h\nn}_{L^2\trkla{\Om}}^2+\tfrac12\norm{\nabla\phi\h\nn-\nabla\phi\h\no}_{L^2\trkla{\Om}}^2-\tfrac12\norm{\nabla\phi\h\no}_{L^2\trkla{\Om}}^2\\
&+\tabs{r\h\nn}^2+\tabs{r\h\nn-r\h\no}^2-\tabs{r\h\no}^2+\tau\norm{\mu\h\nn}\h^2\\
=&-\iO\Ih{\Phi\h\trkla{\phi\h\no}\sinc{n}\Delta\h\phi\h\nn}\dx+\frac{r\h\nn}{\sqrt{E\h\trkla{\phi\h\no}}}\iO\Ih{\Phi\h\trkla{\phi\h\no}\sinc{n} F^\prime\trkla{\phi\h\no}}\dx\\
&-\frac{r\h\nn}{4\tekla{E\h\trkla{\phi\h\no}}^{3/2}}\rkla{\iO\Ih{F^\prime\trkla{\phi\h\no}\Phi\h\trkla{\phi\h\no}\sinc{n}}\dx}^2\\
&+\frac{r\h\nn}{2\sqrt{E\h\trkla{\phi\h\no}}}\iO\Ih{F^{\prime\prime}\trkla{\phi\h\no}\trkla{\Phi\h\trkla{\phi\h\no}\sinc{n}}^2}\dx\\
=:&\,S_1+S_2+S_3+S_4\,.
\end{split}
\end{align}
As the scalar auxiliary variable $r\h\nn$ is merely an approximation of a positive term, there is no non-negativity result available.
Hence, $S_3$ can not simply be neglected.
Next, we apply Young's inequality to separate the implicit terms and the stochastic increments $\sinc{n}$. 
In particular, we deduce for $S_1,\ldots,S_4$:
{\allowdisplaybreaks
\begin{align}
\begin{split}
S_1=&-\iO\Ih{\Phi\h\trkla{\phi\h\no}\sinc{n}\trkla{\Delta\h\phi\h\nn-\Delta\h\phi\h\no}}\dx\\
&-\iO\Ih{\Phi\h\trkla{\phi\h\no}\sinc{n}\Delta\h\phi\h\no}\dx\\
\leq&\,\tfrac14\norm{\nabla\phi\h\nn-\nabla\phi\h\no}_{L^2\trkla{\Om}}^2 +C\norm{\nabla\Ih{\Phi\h\trkla{\phi\h\no}\sinc{n}}}_{L^2\trkla{\Om}}^2\\
&-\iO\Ih{\Phi\h\trkla{\phi\h\no}\sinc{n}\Delta\h\phi\h\no}\dx\,,
\end{split}\\
\begin{split}
S_2=&\,\frac{r\h\nn-r\h\no}{\sqrt{E\h\trkla{\phi\h\no}}}\iO\Ih{\Phi\h\trkla{\phi\h\no}\sinc{n} F^\prime\trkla{\phi\h\no}}\dx \\
&+\frac{r\h\no}{\sqrt{E\h\trkla{\phi\h\no}}}\iO\Ih{\Phi\h\trkla{\phi\h\no}\sinc{n}F^\prime\trkla{\phi\h\no}}\dx\\
\leq&\,\tfrac14\tabs{r\h\nn-r\h\no}^2+C\frac{1}{E\h\trkla{\phi\h\no}}\abs{\iO\Ih{\Phi\h\trkla{\phi\h\no}\sinc{n} F^\prime\trkla{\phi\h\no}}\dx}^2\\
&+\frac{r\h\no}{\sqrt{E\h\trkla{\phi\h\no}}}\iO\Ih{\Phi\h\trkla{\phi\h\no}\sinc{n}F^\prime\trkla{\phi\h\no}}\dx\,,
\end{split}\\
\begin{split}
S_3\leq&\,\tfrac14\tabs{r\h\nn-r\h\no}^2+C\frac{1}{\tekla{E\h\trkla{\phi\h\no}}^{3}}\rkla{\iO\Ih{F^\prime\trkla{\phi\h\no}\Phi\h\trkla{\phi\h\no}\sinc{n}}\dx}^4\\
&-\frac{r\h\no}{4\tekla{E\h\trkla{\phi\h\no}}^{3/2}}\rkla{\iO\Ih{F^\prime\trkla{\phi\h\no}\Phi\h\trkla{\phi\h\no}\sinc{n}}\dx}^2\,,
\end{split}\\
\begin{split}
S_4\leq&\,\tfrac14\tabs{r\h\nn-r\h\no}^2 +C\frac{1}{E\h\trkla{\phi\h\no}}\rkla{\iO\Ih{F^{\prime\prime}\trkla{\phi\h\no}\trkla{\Phi\h\trkla{\phi\h\no}\sinc{n}}^2}\dx}^2\\
&+\frac{r\h\no}{2\sqrt{E\h\trkla{\phi\h\no}}}\iO\Ih{F^{\prime\prime}\trkla{\phi\h\no}\trkla{\Phi\h\trkla{\phi\h\no}\sinc{n}}^2}\dx\,.
\end{split}
\end{align}}
In order to obtain the full $H^1\trkla{\Om}$-norm on the left-hand side of \eqref{eq:energy1}, we choose $\psi\h=\phi\h\nn$ in \eqref{eq:discscheme:phi} which provides after applying Young's inequality
\begin{align}\label{eq:energy2}
\begin{split}
\tfrac12\norm{\phi\h\nn}\h^2&+\tfrac12\norm{\phi\h\nn-\phi\h\no}\h^2-\tfrac12\norm{\phi\h\no}\h^2\\
\leq&\,\tau\tfrac34\norm{\mu\h\nn}\h^2+\tau\tfrac13\norm{\phi\h\nn}\h^2+\tfrac14\norm{\phi\h\nn-\phi\h\no}\h^2+\norm{\Ih{\Phi\h\trkla{\phi\h\no}\sinc{n}}}\h^2\\
&+\iO\Ih{\Phi\h\trkla{\phi\h\no}\sinc{n}\phi\h\no}\dx\,.
\end{split}
\end{align}
Combining \eqref{eq:energy1} and \eqref{eq:energy2} and summing from $n=1$ to $m$, we obtain
\begin{align}
\tfrac12&\norm{\phi\h^m}_{H^1\h\trkla{\Om}}^2+\tabs{r\h^m}^2 +\tfrac14\sum_{n=1}^m\norm{\phi\h\nn-\phi\h\no}_{H^1\h\trkla{\Om}}^2 +\tfrac14\sum_{n=1}^m\tabs{r\h\nn-r\h\no}^2+\tfrac14\tau\sum_{n=1}^m\norm{\mu\h\nn}\h^2\nonumber\\
\leq&\,\tfrac12\norm{\phi\h^0}_{H^1\h\trkla{\Om}}^2+\tabs{r\h^0}^2 +\tfrac13\sum_{n=1}^m\tau\norm{\phi\h\nn}\h^2+C\sum_{n=1}^m\norm{\nabla\Ih{\Phi\h\trkla{\phi\h\no}\sinc{n}}}_{L^2\trkla{\Om}}^2\nonumber\\
&-\sum_{n=1}^m\iO\Ih{\Phi\h\trkla{\phi\h\no}\sinc{n}\Delta\h\phi\h\no}\dx\nonumber \\
\begin{split}
&+ C\sum_{n=1}^m\frac{1}{E\h\trkla{\phi\h\no}}\abs{\iO\Ih{\Phi\h\trkla{\phi\h\no}\sinc{n} F^\prime\trkla{\phi\h\no}}\dx}^2\\
&+\sum_{n=1}^m\frac{r\h\no}{\sqrt{E\h\trkla{\phi\h\no}}}\iO\Ih{\Phi\h\trkla{\phi\h\no}\sinc{n}F^\prime\trkla{\phi\h\no}}\dx\\
&+C\sum_{n=1}^m\frac{1}{\tekla{E\h\trkla{\phi\h\no}}^{3}}\abs{\iO\Ih{F^\prime\trkla{\phi\h\no}\Phi\h\trkla{\phi\h\no}\sinc{n}}\dx}^4\\
&-\sum_{n=1}^m\frac{r\h\no}{4\tekla{E\h\trkla{\phi\h\no}}^{3/2}}\abs{\iO\Ih{F^\prime\trkla{\phi\h\no}\Phi\h\trkla{\phi\h\no}\sinc{n}}\dx}^2\\
&+C\sum_{n=1}^m\frac{1}{E\h\trkla{\phi\h\no}}\abs{\iO\Ih{F^{\prime\prime}\trkla{\phi\h\no}\trkla{\Phi\h\trkla{\phi\h\no}\sinc{n}}^2}\dx}^2\\
&+\sum_{n=1}^m\frac{r\h\no}{2\sqrt{E\h\trkla{\phi\h\no}}}\iO\Ih{F^{\prime\prime}\trkla{\phi\h\no}\trkla{\Phi\h\trkla{\phi\h\no}\sinc{n}}^2}\dx\,.
\end{split}
\end{align}
Absorbing $\tau\tfrac13\norm{\phi\h^m}\h^2$ on the left-hand side, taking the $\p$-th power, and using the norm equivalence \eqref{eq:normequivalence}, we obtain for any $m\in\tgkla{0,\ldots,N}$
\begin{align}\label{eq:energy3}
\begin{split}
&\norm{\phi\h^m}_{H^1\trkla{\Om}}^{2\p}+\tabs{r\h^m}^{2\p}+\rkla{\sum_{n=1}^m\norm{\phi\h\nn-\phi\h\no}_{H^1\trkla{\Om}}^2}^\p+\rkla{\sum_{n=1}^m\tabs{r\h\nn-r\h\no}^2}^\p\\
&\qquad+\rkla{\tau\sum_{n=1}^m\norm{\mu\h\nn}_{L^2\trkla{\Om}}^2}^\p\\
\end{split}\nonumber\\
\begin{split}
&\quad\leq C\norm{\phi\h^0}_{H^1\trkla{\Om}}^{2\p}+C\tabs{r\h^0}^{2\p}+C\sum_{n=1}^m\tau\norm{\phi\h\no}_{L^2\trkla{\Om}}^{2\p}\\
&\qquad+C\rkla{\sum_{n=1}^m\norm{\nabla\Ih{\Phi\h\trkla{\phi\h\no}\sinc{n}}}_{L^2\trkla{\Om}}^2}^\p \\
&\qquad+C\rkla{\max_{1\leq l\leq m}\abs{\sum_{n=1}^l\iO\Ih{\Phi\h\trkla{\phi\h\no}\sinc{n}\Delta\h\phi\h\no}\dx}}^\p\\
&\qquad +C\rkla{\sum_{n=1}^m\frac{1}{E\h\trkla{\phi\h\no}}\abs{\iO\Ih{\Phi\h\trkla{\phi\h\no}\sinc{n} F^\prime\trkla{\phi\h\no}}\dx}^2}^\p
\end{split}\nonumber\\
\begin{split}
&\qquad +C\rkla{\max_{1\leq l\leq m}\abs{\sum_{n=1}^l\frac{r\h\no}{\sqrt{E\h\trkla{\phi\h\no}}}\iO\Ih{\Phi\trkla{\phi\h\no}\sinc{n} F^\prime\trkla{\phi\h\no}}\dx}}^\p\\
&\qquad+C\rkla{\sum_{n=1}^m\frac{1}{\tekla{E\h\trkla{\phi\h\no}}^3}\abs{\iO\Ih{F^\prime\trkla{\phi\h\no}\Phi\h\trkla{\phi\h\no}\sinc{n}}\dx}^4}^\p\\
&\qquad+C\rkla{\max_{1\leq l\leq m}\abs{\sum_{n=1}^l\frac{r\h\no}{\tekla{E\h\trkla{\phi\h\no}}^{3/2}}\abs{\iO\Ih{F^\prime\trkla{\phi\h\no}\Phi\h\trkla{\phi\h\no}\sinc{n}}\dx}^2}}^\p\\
&\qquad+C\rkla{\sum_{n=1}^m\frac{1}{E\h\trkla{\phi\h\no}}\abs{\iO\Ih{F^{\prime\prime}\trkla{\phi\h\no}\trkla{\Phi\h\trkla{\phi\h\no}\sinc{n}}^2}\dx}^2}^\p\\
&\qquad+C\rkla{\max_{1\leq l\leq m}\abs{\sum_{n=1}^l\frac{r\h\no}{\sqrt{E\h\trkla{\phi\h\no}}}\iO\Ih{F^{\prime\prime}\trkla{\phi\h\no}\trkla{\Phi\h\trkla{\phi\h\no}\sinc{n}}^2}\dx}}^\p\\
&\quad=: C\norm{\phi\h^0}_{H^1\trkla{\Om}}^{2\p}+C\tabs{r\h^0}^{2\p} +C\sum_{n=1}^m\tau\norm{\phi\h\no}_{L^2\trkla{\Om}}^{2\p}+R_{1,m}+R_{2,m}+R_{3,m}\\
&\qquad+R_{4,m}+R_{5,m}+R_{6,m}+R_{7,m}+R_{8,m}\,,
\end{split}
\end{align}
where the constant $C$ depends on $\p$ but not on $h$ or $\tau$.
We shall now derive estimates for the expected values of the stochastic terms $R_{1,m},\ldots,R_{8,m}$.
To obtain an estimate for $R_{1,m}$, we use Hölder's inequality, Lemma \ref{lem:BDG}, \eqref{eq:defPhih}, \ref{item:color}, and \ref{item:sigma} to compute
\begin{align}
\begin{split}
\expected{R_{1,m}}\leq& C m^{\p-1}\sum_{n=1}^m\expected{\norm{\nabla\Ih{\Phi\h\trkla{\phi\h\no}\sinc{n}}}_{L^2\trkla{\Om}}^{2\p}}\\
\leq &\,Cm^{\p-1}\sum_{n=1}^m\tau^\p\expected{\rkla{\sum_{k\in\Zh}\norm{\lambda_k\nabla\Ih{\Phi\h\trkla{\phi\h\no}\g{k}}}_{L^2\trkla{\Om}}^2}^{\p}}\\
\leq&\, C\sum_{n=1}^m\tau\expected{\norm{\Ih{\sigma\trkla{\phi\h\no}}}_{H^1\trkla{\Om}}^{2\p}}\leq C\sum_{n=1}^m\tau\expected{\rkla{1+\norm{\nabla\phi\h\no}_{L^2\trkla{\Om}}^{2\p}}}\,.
\end{split}
\end{align}
As $\sum_{n=1}^l\iO\Ih{\Phi\h\trkla{\phi\h\no}\sinc{n}\Delta\h\phi\h\no}\dx$ is an $\mathds{R}$-valued (discrete) martingale, we apply Lemma \ref{lem:BDG}, \eqref{eq:defPhih}, \ref{item:color}, and \ref{item:sigma} which yields
\begin{align}
\begin{split}
\expected{R_{2,m}}=&\,C\expected{\rkla{\max_{1\leq l\leq m}\abs{\sum_{n=1}^l\iO\nabla\Ih{\Phi\h\trkla{\phi\h\no}\sinc{n}}\cdot\nabla\phi\h\no\dx}}^\p}\\
\leq &\,C \sum_{n=1}^m\tau\expected{\norm{\Ih{\sigma\trkla{\phi\h\no}}}_{H^1\trkla{\Om}}^\p\norm{\nabla\phi\h\no}_{L^2\trkla{\Om}}^\p}\\
\leq&\, C+C\expected{\sum_{n=1}^m\tau\norm{\nabla\phi\h\no}_{L^2\trkla{\Om}}^{2\p}}\,.
\end{split}
\end{align}
Using similar arguments, we obtain for $R_{3,m}$
\begin{align}
\begin{split}
\expected{R_{3,m}}\leq&\,C\sum_{n=1}^m\tau\expected{\frac{1}{\tekla{E\h\trkla{\phi\h\no}}^\p}\norm{\Ih{\sigma\trkla{\phi\h\no} F^\prime\trkla{\phi\h\no}}}_{L^1\trkla{\Om}}^{2\p}}\,.
\end{split}
\end{align}
As $\sigma$ is bounded (cf.~\ref{item:sigma}), we obtain from \ref{item:potentialF} using \eqref{eq:normequivalence} and Hölder's inequality
\begin{align}\label{eq:ehtrick}
\begin{split}
&\rkla{\iO\abs{\Ih{\sigma\trkla{\phi\h\no}F^\prime\trkla{\phi\h\no}}}\dx}^2\leq C\rkla{\iO\Ih{\abs{F^\prime\trkla{\phi\h\no}}}\dx}^2\\
&\qquad\leq C\rkla{\iO\Ih{1+\tabs{\phi\h\no}^3}\dx}^2\leq C+C\norm{\phi\h\no}_{L^2\trkla{\Om}}^2\norm{\tabs{\phi\h\no}^2}_{L^2\trkla{\Om}}^2\\
&\qquad\leq C+C\norm{\phi\h\no}_{L^2\trkla{\Om}}^2\iO\Ih{F\trkla{\phi\h\no}}\dx\,.
\end{split}
\end{align}
Therefore, we deduce
\begin{align}
\expected{R_{3,m}}\leq C+C\sum_{n=1}^m\tau\expected{\norm{\phi\h\no}_{L^2\trkla{\Om}}^{2\p}}
\end{align}
due to the lower bound on $F$ expressed in \ref{item:potentialF}.
For $R_{4,m}$, we obtain
\begin{align}
\begin{split}
\expected{R_{4,m}}\leq &\,C\sum_{n=1}^m\tau\expected{\rkla{\frac{\tabs{r\h\no}}{\sqrt{E\h\trkla{\phi\h\no}}}\norm{\Ih{\sigma\trkla{\phi\h\no}F^\prime\trkla{\phi\h\no}}}_{L^1\trkla{\Om}}}^\p}\\
\leq&\,C+C\sum_{n=1}^m\tau\expected{\tabs{r\h\no}^{2\p}}+C\sum_{n=1}^m\tau\expected{\norm{\phi\h\no}_{L^2\trkla{\Om}}^{2\p}}
\end{split}
\end{align}
due to \eqref{eq:ehtrick} and Young's inequality.
Using similar arguments, we obtain for the remaining terms
\begin{align}
\begin{split}
\expected{R_{5,m}}\leq&\,Cm^{\p-1}\sum_{n=1}^m\expected{\abs{\frac{1}{\tekla{E\h\trkla{\phi\h\no}}^{3/4}}\iO\Ih{F^\prime\trkla{\phi\h\no}\Phi\h\trkla{\phi\h\no}\sinc{n}}\dx}^{4\p}}\\
\leq&\, Cm^{\p-1}\!\sum_{n=1}^m\expected{\!\rkla{\!\tau\sum_{k\in\mathds{Z}}\!\rkla{\frac{1}{\tekla{E\h\trkla{\phi\h\no}}^{3/4}}\!\iO\!\Ih{F^\prime\trkla{\phi\h\no}\sigma\trkla{\phi\h\no}\lambda_k\g{k}}\dx}^2}^{2\p}}\\
\leq&\,C\tau^\p \sum_{n=1}^m\tau\expected{\rkla{\frac{1}{\tekla{E\h\trkla{\phi\h\no}}^{3/4}}\iO\abs{\Ih{F^\prime\trkla{\phi\h\no}}}\dx}^{4\p}}\\
\leq &\,C\tau^\p\,,
\end{split}\\
\begin{split}
\expected{R_{6,m}}\leq&\,C \sum_{n=1}^m\tau\expected{\rkla{\frac{\tabs{r\h\no}^{1/2}}{\tekla{E\h\trkla{\phi\h\no}}^{3/4}}\norm{\Ih{F^\prime\trkla{\phi\h\no}\sigma\trkla{\phi\h\no}}}_{L^1\trkla{\Om}}}^{2\p}}\\
\leq&\, C\sum_{n=1}^m\tau\expected{\tabs{r\h\no}^\p}\leq C\sum_{n=1}^m\tau\expected{\tabs{r\h\no}^{2\p}}+C\,,
\end{split}\\
\begin{split}
\expected{R_{7,m}}\leq&\,C m^{\p-1}\sum_{n=1}^m\expected{\norm{\Ih{\frac{1}{\tekla{E\h\trkla{\phi\h\no}}^{1/4}}\abs{F^{\prime\prime}\trkla{\phi\h\no}}^{1/2}\Phi\h\trkla{\phi\h\no}\sinc{n}}}\h^{4\p}}\\
\leq&\, Cm^{\p-1}\sum_{n=1}^m\expected{\tau^{2\p}\frac{1}{\tekla{E\h\trkla{\phi\h\no}}^{\p}}\norm{\Ih{\abs{F^{\prime\prime}\trkla{\phi\h\no}}^{1/2}\sigma\trkla{\phi\h\no}}}\h^{4\p}}\\
\leq&\,C\tau^\p\sum_{n=1}^m\tau\expected{\rkla{\frac{\iO\Ih{\abs{F^{\prime\prime}\trkla{\phi\h\no}}}\dx}{\sqrt{E\h\trkla{\phi\h\no}}}}^{2\p}}\leq C\tau^\p\,,
\end{split}\\
\begin{split}
\expected{R_{8,m}}\leq&\,C\sum_{n=1}^m\tau\expected{\norm{\Ih{\abs{F^{\prime\prime}\trkla{\phi\h\no}}^{1/2}}}\h^{2\p}\frac{\tabs{r\h\no}^\p}{\tekla{E\h\trkla{\phi\h\no}}^{\p/2}}}\\
\leq&\,C\sum_{n=1}^m\tau\expected{1+\tabs{r\h\no}^\p}\leq C\sum_{n=1}^m\tau\expected{\tabs{r\h\no}^{2\p}}+C\,.
\end{split}
\end{align}
Collecting the results and recalling assumption \ref{item:initial}, we obtain
\begin{align}
\expected{\norm{\phi\h^m}_{H^1\trkla{\Om}}^{2\p}}+\expected{\tabs{r\h^m}^{2\p}}\leq C\trkla{\phi_0}+C\sum_{n=1}^m\tau\expected{\norm{\phi\h\no}_{H^1\trkla{\Om}}^{2\p}} +C\sum_{n=1}^m\tau\expected{\tabs{r\h\no}^{2\p}}
\end{align}
for all $m\in\tgkla{1,\ldots,N}$.
Now, a discrete version of Gronwall's inequality provides \eqref{eq:energyaux}.\\
In a second step, we shall now establish the original claim of the lemma.
Starting from \eqref{eq:energy3}, we take the maximum over $m\in\tgkla{1,\ldots,N}$ before taking the expected value.
As the right-hand side is maximized for $m=N$, we can reuse the above calculations to obtain estimates for $R_{1,N},\ldots,R_{8,N}$ and obtain
\begin{align}
\begin{split}
&\expected{\max_{0\leq m\leq N}\norm{\phi\h^m}_{H^1\trkla{\Om}}^{2\p}}+\expected{\max_{0\leq m\leq N}\tabs{r\h^m}^{2\p}}+\expected{\rkla{\sum_{n=1}^N\norm{\phi\h\nn-\phi\h\no}_{H^1\trkla{\Om}}^2}^\p}\\
&\qquad+\expected{\rkla{\sum_{n=1}^N\abs{r\h\nn-r\h\no}^2}^\p}+\expected{\rkla{\sum_{n=1}^N\tau\norm{\mu\h\nn}_{L^2\trkla{\Om}}^2}^\p}\\
&\quad\leq C\trkla{\phi_0} + C\sum_{n=1}^N\tau\expected{\norm{\phi\h\no}_{H^1\trkla{\Om}}^{2\p}}+C\sum_{n=1}^N\tau\expected{\tabs{r\h\no}^{2\p}}\leq C\,,
\end{split}
\end{align}
due to \eqref{eq:energyaux}.
\end{proof}

As shown in the following corollary, the discrete version of the $L^{2\p}\trkla{\Omega;L^2\trkla{0,T;L^2\trkla{\Om}}}$-bound on the chemical potential provides also a discrete $L^{2q}\trkla{\Omega;L^2\trkla{0,T;L^2\trkla{\Om}}}$-bound for the discrete Laplacian of $\phi$.
\begin{corollary}\label{cor:h2}
Let the assumptions \ref{item:timedisc}, \ref{item:spatialdisc}, \ref{item:potentialF}, \ref{item:initial}, \ref{item:sigma}, and \ref{item:filtration}--\ref{item:color} hold true.
Then, for all $q\in[1,\infty)$ there exists a constant $C>0$ which is independent of $h$ and $\tau$ such that
\begin{align*}
\expected{\rkla{\sum_{n=1}^N\tau\norm{\Delta\h\phi\h\nn}_{L^2\trkla{\Om}}^2}^q}\leq C\,.
\end{align*}
\end{corollary}
\begin{proof}
Starting from the definition of the chemical potential in \eqref{eq:defmu}, we obtain by applying Young's inequality and an absorption argument
\begin{align}
\begin{split}
\norm{\Delta\h\phi\h\nn}\h^2\leq&\,C\norm{\mu\h\nn}_{L^2\trkla{\Om}}^2 +C\frac{\tabs{r\h\nn}^2}{E\h\trkla{\phi\h\no}}\norm{\Ih{F^\prime\trkla{\phi\h\no}}}\h^2\\
&+C\frac{\tabs{r\h\nn}^2}{\tekla{E\h\trkla{\phi\h\no}}^3}\abs{\iO\Ih{F^\prime\trkla{\phi\h\no}\Phi\h\trkla{\phi\h\no}\sinc{n}}\dx}^2\norm{\Ih{F^\prime\trkla{\phi\h\no}}}\h^2\\
&+C\frac{\tabs{r\h\nn}^2}{E\h\trkla{\phi\h\no}}\norm{\Ih{F^{\prime\prime}\trkla{\phi\h\no}\Phi\h\trkla{\phi\h\no}\sinc{n}}}\h^2\\
=:&\,R_{1,n}+R_{2,n}+R_{3,n}+R_{4,n}\,.
\end{split}
\end{align}
Multiplying by $\tau$, summing from $n=1$ to $N$, and taking the $q$-th power and the expected value yields
\begin{multline}
\expected{\rkla{\sum_{n=1}^N\tau\norm{\Delta\h\phi}\h^2}^q}\leq C\expected{\rkla{\sum_{n=1}^N\tau\norm{\mu\h\nn}_{L^2\trkla{\Om}}^2}^q}+C\expected{\rkla{\sum_{n=1}^N\tau R_{2,n}}^q}\\
+C\expected{\rkla{\sum_{n=1}^N\tau R_{3,n}}^q}+C\expected{\rkla{\sum_{n=1}^N\tau R_{4,n}}^q}\,.
\end{multline}
While the first term on the right-hand side can be controlled using the result of Lemma \ref{lem:energy}, we use growth conditions from \ref{item:potentialF}, the lower bound of $E\h\trkla{\phi\h\no}$, and Young's inequality to derive
\begin{align}
\begin{split}
\expected{\rkla{\sum_{n=1}^N\tau R_{2,n}}^q}\leq&\, C\expected{\max_{1\leq n\leq N}\tabs{r\h\nn}^{2q}\rkla{1+\max_{1\leq n\leq N}\norm{\phi\h\no}_{H^1\trkla{\Om}}^{6q}}}\leq C\,,
\end{split}
\end{align}
due to Lemma \ref{lem:energy}.
To estimate the remaining terms, we need to apply Lemma \ref{lem:BDG}.
Using Hölder's inequality and recalling that $E\h\trkla{.}$ is bounded from below by a constant independent of $h$ and $\tau$, we obtain
\begin{align}
\begin{split}
\expected{\rkla{\sum_{n=1}^N\tau R_{3,n}}^q}\leq&\,C\expected{\rkla{\rkla{\max_{1\leq n\leq N}\tabs{r\h\nn}^2}\rkla{1+\max_{1\leq n\leq N}\norm{\phi\h\no}_{H^1\trkla{\Om}}^6}}^{2q}}^{1/2}\\
&\times\expected{\rkla{\sum_{n=1}^N\tau\abs{\iO\Ih{F^\prime\trkla{\phi\h\no}\Phi\h\trkla{\phi\h\no}\sinc{n}}\dx}^2}^{2q}}^{1/2}\,.
\end{split}
\end{align}
We estimate the second factor via
\begin{align}
&\expected{\rkla{\sum_{n=1}^N\tau\abs{\iO\Ih{F^\prime\trkla{\phi\h\no}\Phi\h\trkla{\phi\h\no}\sinc{n}}\dx}^2}^{2q}}\nonumber\\
&\qquad\leq C\sum_{n=1}^N\tau\expected{\abs{\iO\Ih{F^\prime\trkla{\phi\h\no}\Phi\h\trkla{\phi\h\no}\sinc{n}}\dx}^{4q}}\nonumber\\
&\qquad\leq C\sum_{n=1}^N\tau\expected{\rkla{\tau\sum_{k\in\mathds{Z}}\abs{\iO\Ih{F^\prime\trkla{\phi\h\no}\sigma\trkla{\phi\h\no}\lambda_k\g{k}}\dx}^2}^{2q}}\nonumber\\
&\qquad\leq C\expected{\rkla{1+\max_{1\leq n\leq N}\norm{\phi\h\no}_{H^1\trkla{\Om}}^{12q}}\tau^{2q}}\,.
\end{align}
Therefore, we have 
\begin{align}\label{eq:muerror:1}
\expected{\rkla{\sum_{n=1}^N\tau R_{3,n}}^q}\leq C\tau^q\,
\end{align}
due to the regularity results established in Lemma \ref{lem:energy}.
To obtain an estimate for the last term, we again apply Hölder's inequality and compute
\begin{align}
\begin{split}
\expected{\rkla{\sum_{n=1}^N\tau R_{4,n}}^q}\leq&\,C\expected{\max_{1\leq n\leq N}\tabs{r\h\nn}^{4q}}^{1/2}\expected{\sum_{n=1}^N\tau\norm{\Ih{F^{\prime\prime}\trkla{\phi\h\no}\Phi\h\trkla{\phi\h\no}\sinc{n}}}\h^{4q}}^{1/2}\,.
\end{split}
\end{align}
Again, the second factor can be controlled by applying Lemma \ref{lem:BDG}:
\begin{multline}
\expected{\sum_{n=1}^N\tau\norm{\Ih{F^{\prime\prime}\trkla{\phi\h\no}\Phi\h\trkla{\phi\h\no}\sinc{n}}}\h^{4q}}\\
\leq C\sum_{n=1}^N\tau \expected{\rkla{\tau\sum_{k\in\mathds{Z}}\norm{\Ih{F^{\prime\prime}\trkla{\phi\h\no}\sigma\trkla{\phi\h\no}\lambda_k\g{k}}}\h^2}^{2q}}\\
\leq C\sum_{n=1}^N\tau\expected{\tau^{2q}\norm{\Ih{F^{\prime\prime}\trkla{\phi\h\no}}}\h^{4q}}\leq C\tau^{2q}\,.
\end{multline}
Therefore,
\begin{align}\label{eq:muerror:2}
\expected{\rkla{\sum_{n=1}^N\tau R_{4,n}}^q}\leq C\tau^q\,
\end{align}
which provides the result.
\end{proof}

The proof of Corollary \ref{cor:h2} also shows that the higher order terms used to approximate $F^\prime\trkla{\phi}$ will vanish when passing to the limit $\tau\searrow0$.
In particular, \eqref{eq:muerror:1} and \eqref{eq:muerror:2} provide the following result:\\
\begin{corollary}\label{cor:muerror}
Let the assumptions \ref{item:timedisc}, \ref{item:spatialdisc}, \ref{item:potentialF}, \ref{item:initial}, \ref{item:sigma}, and \ref{item:filtration}--\ref{item:color} hold true.
Then, for all $q\in[1,\infty)$ there exists a constant $C>0$ which is independent of $h$ and $\tau$ such that
\begin{align*}
\expected{\rkla{\sum_{n=1}^N\tau\norm{\Xi\h\nn}_{L^2\trkla{\Om}}^2}^q}\leq C\tau^q\,.
\end{align*}
\end{corollary}
To obtain compactness with respect to time, we will show that our solutions are contained in a suitable Nikolskii space:
\begin{lemma}\label{lem:nikolskii}
Let the assumptions \ref{item:timedisc}, \ref{item:spatialdisc}, \ref{item:potentialF}, \ref{item:initial}, \ref{item:sigma}, and \ref{item:filtration}--\ref{item:color} hold true.
Then, for all $\alpha\geq1$ there exists a constant $C>0$ which is independent of $h$ and $\tau$ such that
\begin{align*}
\expected{\sum_{m=0}^{N-l}\tau\norm{\phi\h^{m+l}-\phi\h^m}_{L^2\trkla{\Om}}^{2\alpha}}\leq C\trkla{l\tau}^\alpha
\end{align*}
for all $l=0,\ldots,N$.
\end{lemma}
\begin{proof}
Summing \eqref{eq:discscheme:phi} from $n=m+1$ to $m+l\leq N$, choosing $\psi\h=\trkla{\phi\h^{m+l}-\phi\h^m}$, taking the $\alpha$-th power on both sides, multiplying by $\tau$, summing the result from $m=0$ to $N-l$, and computing the expected value, we obtain
\begin{align}
\begin{split}
\expected{\sum_{m=0}^{N-l}\tau\norm{\phi\h^{m+l}-\phi\h^m}\h^{2\alpha}}\leq&\,C\expected{\sum_{m=0}^{N-l}\tau\abs{\sum_{n=m+1}^{m+l}\tau\iO\Ih{\mu\h\nn\trkla{\phi\h^{m+l}-\phi\h^m}}\dx}^\alpha}\\
&+C\expected{\sum_{m=0}^{N-l}\tau\abs{\iO\sum_{n=m+1}^{m+l}\Ih{\Phi\h\trkla{\phi\h\no}\sinc{n}\trkla{\phi\h^{m+l}-\phi\h^m}}\dx}^\alpha}\\
=:&\,A+B\,.
\end{split}
\end{align}
Applying Hölder's inequality and Young's inequality to $A$, we obtain
\begin{align}
\begin{split}
A\leq&\,\tfrac14\expected{\sum_{m=0}^{N-l}\tau\norm{\phi\h^{m+l}-\phi\h^m}\h^{2\alpha}}+C\expected{\sum_{m=0}^{N-l}\tau\rkla{\sum_{n=m+1}^{m+l}\tau\norm{\mu\h\nn}\h}^{2\alpha}}\,.
\end{split}
\end{align}
Applying Hölder's inequality to the second term on the right-hand side, we deduce
\begin{align}
\begin{split}
\expected{\sum_{m=0}^{N-l}\tau\rkla{\sum_{n=m+1}^{m+l}\tau\norm{\mu\h\nn}\h}^{2\alpha}}\leq&\, C\trkla{l\tau}^\alpha\expected{\sum_{m=0}^{N-l}\tau\rkla{\sum_{n=m+1}^{m+l}\tau\norm{\mu\h\nn}\h^2}^\alpha}\\
\leq&\, C\trkla{l\tau}^{\alpha}T\expected{\rkla{\sum_{n=1}^N\tau\norm{\mu\h\nn}\h^2}^\alpha}\leq C\trkla{l\tau}^\alpha\,,
\end{split}
\end{align}
due to Lemma \ref{lem:energy}.
To obtain an estimate for $B$, we apply Hölder's inequality, Young's inequality, and Lemma \ref{lem:BDG} which provides
\begin{align}
\begin{split}
B\leq&\,C\expected{\sum_{m=0}^{N-l}\tau\rkla{\norm{\sum_{n=m+1}^{m+l}\Ih{\Phi\h\trkla{\phi\h\no}\sinc{n}}}\h\norm{\phi\h^{m+l}-\phi\h^m}\h}^\alpha}\\
\leq&\,\tfrac14\expected{\sum_{m=0}^{N-l}\tau\norm{\phi\h^{m+l}-\phi\h^m}\h^{2\alpha}}+C\sum_{m=0}^{N-l}\tau\expected{\norm{\sum_{n=m+1}^{m+l}\Ih{\Phi\h\trkla{\phi\h\no}\sinc{n}}}\h^{2\alpha}}\\
\leq&\,\tfrac14\expected{\sum_{m=0}^{N-l}\tau\norm{\phi\h^{m+l}-\phi\h^m}\h^{2\alpha}}+C\sum_{m=0}^{N-l}\tau\trkla{l\tau}^{\alpha-1}\sum_{n=m+1}^{m+l}\tau\expected{\norm{\Ih{\sigma\trkla{\phi\h\no}}}_{L^2\trkla{\Om}}^{2\alpha}}\,.
\end{split}
\end{align}
As $\sigma$ is uniformly bounded (cf.~\ref{item:sigma}), combining the above results concludes the proof.
\end{proof}

Based on the above result, we shall now show that the scalar auxiliary variable $r\h\nn$ is indeed a suitable approximation of $\sqrt{E\h\trkla{\phi\h\nn}}=\sqrt{\iO\Ih{F\trkla{\phi\h\nn}}\dx}$.
In particular, we will establish the following convergence result:
\begin{lemma}\label{lem:SAVerror}
Let the assumptions \ref{item:timedisc}, \ref{item:spatialdisc}, \ref{item:potentialF}, \ref{item:initial}, \ref{item:sigma}, and \ref{item:filtration}--\ref{item:color} hold true.
Then, for all $p\in[1,\infty)$, there exists a constant $C>0$ which is independent of $h$ and $\tau$ such that the estimate
\begin{align*}
\expected{\max_{0\leq m\leq N}\abs{r\h^m-\sqrt{E\h\trkla{\phi\h^m}}}^p}\leq C\tau^{p\min\tgkla{\nu/2,\,\trkla{1/4-\delta/2}}}
\end{align*}
holds true for any $0<\delta<1/2$, where $\nu$ is the Hölder exponent from \ref{item:potentialF}.
\end{lemma}
\begin{proof}
We start with a Taylor expansion of $\sqrt{E\h\trkla{\phi\h\nn}}$ to quantify the difference between the increments $r\h\nn-r\h\no$ and $\sqrt{E\h\trkla{\phi\h\nn}}-\sqrt{E\h\trkla{\phi\h\no}}$:
\begin{align}
\begin{split}
\sqrt{E\h\trkla{\phi\h\nn}}&-\sqrt{E\h\trkla{\phi\h\no}}\\
=&\,\frac{1}{2\sqrt{E\h\trkla{\phi\h\no}}}\trkla{E\h\trkla{\phi\h\nn}-E\h\trkla{\phi\h\no}}-\frac{1}{8\tekla{E\h\trkla{\phi\h\no}}^{3/2}}\rkla{E\h\trkla{\phi\h\nn}-E\h\trkla{\phi\h\no}}^2\\
&+\frac{1}{16\tekla{E\h\trkla{\varphi_1}}^{5/2}}\rkla{E\h\trkla{\phi\h\nn}-E\h\trkla{\phi\h\no}}^3\\
=&\,\frac{1}{2\sqrt{E\h\trkla{\phi\h\no}}}\rkla{\iO\Ih{F^\prime\trkla{\phi\h\no}\trkla{\phi\h\nn-\phi\h\no}+\tfrac12F^{\prime\prime}\trkla{\phi\h\no}\trkla{\phi\h\nn-\phi\h\no}^2}\dx}\\
&+\frac{1}{2\sqrt{E\h\trkla{\phi\h\no}}}\iO\Ih{\tfrac12\rkla{F^{\prime\prime}\trkla{\varphi_2}-F^{\prime\prime}\trkla{\phi\h\no}}\trkla{\phi\h\nn-\phi\h\no}^2}\dx\\
&-\frac{1}{8\tekla{E\h\trkla{\phi\h\no}}^{3/2}}\rkla{\iO\Ih{F^\prime\trkla{\phi\h\no}\trkla{\phi\h\nn-\phi\h\no}}\dx}^2\\
&-\frac{1}{8\tekla{E\h\trkla{\phi\h\no}}^{3/2}}\rkla{\iO\Ih{F^\prime\trkla{\phi\h\no}\trkla{\phi\h\nn-\phi\h\no}}\dx}\\
&\qquad\qquad\times\rkla{\iO\Ih{F^{\prime\prime}\trkla{\varphi_3}\trkla{\phi\h\nn-\phi\h\no}^2}\dx}\\
&-\frac{1}{8\tekla{E\h\trkla{\phi\h\no}}^{3/2}}\rkla{\tfrac12\iO\Ih{F^{\prime\prime}\trkla{\varphi_3}\trkla{\phi\h\nn-\phi\h\no}^2}\dx}^2\\
&+\frac{1}{16\tekla{E\h\trkla{\varphi_1}}^{5/2}}\rkla{\iO\Ih{F^\prime\trkla{\varphi_4}\trkla{\phi\h\nn-\phi\h\no}}\dx}^3\,
\end{split}\label{eq:taylor1}
\end{align}
with $\varphi_1,\,\varphi_2,\,\varphi_3,\,\varphi_4\in\operatorname{conv}\tgkla{\phi\h\nn,\phi\h\no}$.
Due to \eqref{eq:modeldisc:phimu}, we have $\trkla{\phi\h\nn-\phi\h\no}=-\tau\mu\h\nn+\Ih{\Phi\h\trkla{\phi\h\no}\sinc{n}}$ which allows us to rewrite \eqref{eq:taylor1} as
\begin{align}\label{eq:taylor2}
\begin{split}
\sqrt{E\h\trkla{\phi\h\nn}}&\,-\sqrt{E\h\trkla{\phi\h\no}}=r\h\nn-r\h\no\\
&-\frac{1}{4\sqrt{E\h\trkla{\phi\h\no}}}\iO\Ih{F^{\prime\prime}\trkla{\phi\h\no}\trkla{\phi\h\nn-\phi\h\no}\tau\mu\h\nn}\dx\\
&+\frac{1}{4\sqrt{E\h\trkla{\phi\h\no}}}\iO\Ih{\trkla{F^{\prime\prime}\trkla{\varphi_2}-F^{\prime\prime}\trkla{\phi\h\no}}\trkla{\phi\h\nn-\phi\h\no}^2}\dx\\
&+\frac{1}{8\tekla{E\h\trkla{\phi\h\no}}^{3/2}}\rkla{\iO\Ih{F^\prime\trkla{\phi\h\no}\trkla{\phi\h\nn-\phi\h\no}}\dx}\rkla{\iO\Ih{F^\prime\trkla{\phi\h\no}\tau\mu\h\nn}\dx}\\
&-\frac{1}{8\tekla{E\h\trkla{\phi\h\no}}^{3/2}}\rkla{\iO\Ih{F^\prime\trkla{\phi\h\no}\trkla{\phi\h\nn-\phi\h\no}}\dx}\\
&\qquad\qquad\times\rkla{\iO\Ih{F^{\prime\prime}\trkla{\varphi_3}\trkla{\phi\h\nn-\phi\h\no}^2}\dx}\\
&-\frac{1}{8\tekla{E\h\trkla{\phi\h\no}}^{3/2}}\rkla{\tfrac12\iO\Ih{F^{\prime\prime}\trkla{\varphi_3}\trkla{\phi\h\nn-\phi\h\no}^2}\dx}^2\\
&+\frac{1}{16\tekla{E\h\trkla{\varphi_1}}^{5/2}}\rkla{\iO\Ih{F^\prime\trkla{\varphi_4}\trkla{\phi\h\nn-\phi\h\no}}\dx}^3\\
=:&\,r\h\nn-r\h\no +R_{1,n}+R_{2,n}+R_{3,n}+R_{4,n}+R_{5,n}+R_{6,n}\,.
\end{split}
\end{align}
Summing \eqref{eq:taylor2} from $n=1$ to $m\leq N$, noting that by definition $r\h^0=\sqrt{E\h\trkla{\phi\h^0}}$, taking the $p$-th power and the supremum over all $m\in\tgkla{0,\ldots,N}$, and computing the expected value, we obtain
\begin{align}
\begin{split}
\expected{\sup_{0\leq m\leq N}\abs{r\h^m-\sqrt{E\h\trkla{\phi\h^m}}}^p}\leq&\, C\expected{\abs{\sum_{n=1}^N R_{1,n}}^p}+C\expected{\abs{\sum_{n=1}^N R_{2,n}}^p}+C\expected{\abs{\sum_{n=1}^N R_{3,n}}^p}\\
+&C\expected{\abs{\sum_{n=1}^N R_{4,n}}^p}+C\expected{\abs{\sum_{n=1}^N R_{5,n}}^p}+C\expected{\abs{\sum_{n=1}^N R_{6,n}}^p}\,.
\end{split}
\end{align}
To estimate $R_{1,n}$, we combine Hölder's inequality, \ref{item:potentialF}, Young's inequality, and the results from Lemma \ref{lem:energy} to compute
\begin{align}
\begin{split}
&\expected{\abs{\sum_{n=1}^N R_{1,n}}^p}\leq C\expected{\abs{\sum_{n=1}^N\tau\norm{\mu\h\nn}_{L^2\trkla{\Om}} \norm{F^{\prime\prime}\trkla{\phi\h\no}}_{L^3\trkla{\Om}}\norm{\phi\h\nn-\phi\h\no}_{L^6\trkla{\Om}} }^p}\\
&\leq C\expected{\abs{\rkla{1+\max_{1\leq n\leq N}\norm{\phi\h\nn}_{H^1\trkla{\Om}}^2}\tau^{1/2}\rkla{\sum_{n=1}^N\tau\norm{\mu\h\nn}_{L^2\trkla{\Om}}^2+\sum_{n=1}^N\norm{\phi\h\nn-\phi\h\no}_{H^1\trkla{\Om}}^2}}^p}\\
&\leq C\tau^{p/2}\,.
\end{split}
\end{align}
Concerning $R_{2,n}$ we first consider the $C^{2,\nu}$-part $F_1$ (cf.~\ref{item:potentialF}) and obtain the pathwise estimate
\begin{multline}\label{eq:R2F1}
\abs{\iO\Ih{\rkla{F_1^{\prime\prime}\trkla{\varphi_2}-F_1^{\prime\prime}\trkla{\phi\h\no}}\trkla{\phi\h\nn-\phi\h\no}^2}\dx}\leq C\iO\Ih{\tabs{\phi\h\nn-\phi\h\no}^{2+\nu}}\dx\\
\leq C\norm{\phi\h\nn-\phi\h\no}_{H^1\trkla{\Om}}^{3\nu/2}\norm{\phi\h\nn-\phi\h\no}_{L^2\trkla{\Om}}^{\trkla{4-\nu}/2}\\
\leq C\tau^{\nu/2}\norm{\phi\h\nn-\phi\h\no}_{H^1\trkla{\Om}}^2+C\tau^{-3\nu^2/\trkla{8-6\nu}}\norm{\phi\h\nn-\phi\h\no}_{L^2\trkla{\Om}}^{\trkla{8-2\nu}/\trkla{4-3\nu}}\,,
\end{multline}
due to the Gagliardo--Nirenberg inequality and Young's inequality.
The remaining part can be estimated using the growth condition on $F_2^{\prime\prime\prime}$ stated in \ref{item:potentialF} and the Gagliardo--Nirenberg inequality.
For any $1/2>\delta>0$, we obtain
\begin{align}\label{eq:optimal}
\begin{split}
\!\!\!\!\!&\abs{\iO\Ih{\trkla{F_2^{\prime\prime}\trkla{\varphi_2}-F_2^{\prime\prime}\trkla{\phi\h\no}}\trkla{\phi\h\nn-\phi\h\no}^2}\dx}\leq C\norm{\Ih{F_2^{\prime\prime\prime}\trkla{\hat{\varphi}}}}_{L^3\trkla{\Om}}\norm{\phi\h\nn-\phi\h\no}_{L^{9/2}\trkla{\Om}}^3\\
&\qquad\leq C\rkla{1+\norm{\phi\h\nn}_{H^1\trkla{\Om}}^2+\norm{\phi\h\no}_{H^1\trkla{\Om}}^2}\norm{\phi\h\nn-\phi\h\no}_{H^1\trkla{\Om}}^{5/2}\norm{\phi\h\nn-\phi\h\no}_{L^2\trkla{\Om}}^{1/2}\\
&\qquad\leq C\rkla{1+\norm{\phi\h\nn}_{H^1\trkla{\Om}}^{5/2+\delta}+\norm{\phi\h\no}_{H^1\trkla{\Om}}^{5/2+\delta}}\norm{\phi\h\nn-\phi\h\no}_{H^1\trkla{\Om}}^{2-\delta}\norm{\phi\h\nn-\phi\h\no}_{L^2\trkla{\Om}}^{1/2}\\
&\qquad\leq C\rkla{1+\norm{\phi\h\nn}_{H^1\trkla{\Om}}^{5/2+\delta}+\norm{\phi\h\no}_{H^1\trkla{\Om}}^{5/2+\delta}}\\
&\qquad\qquad\times\rkla{\tau^{\trkla{1/4-\delta/2}}\norm{\phi\h\nn-\phi\h\no}_{H^1\trkla{\Om}}^2+\tau^{-\tfrac{\trkla{2-\delta}\trkla{1/2-\delta}}{2\delta}}\norm{\phi\h\nn-\phi\h\no}_{L^2\trkla{\Om}}^{1/\delta}}
\end{split}
\end{align}
with a suitable $\hat{\varphi}\in\operatorname{conv}\tgkla{\phi\h\nn,\phi\h\no}$.
Combining \eqref{eq:R2F1} and \eqref{eq:optimal} with the uniform lower bound on $E\h\trkla{\phi\h\no}$, Hölder's inequality, and the results from Lemma \ref{lem:energy} and Lemma \ref{lem:nikolskii}, we compute
\begin{align}
\begin{split}
&\expected{\abs{\sum_{n=1}^N R_{2,n}}^p}\leq\, C\expected{\rkla{\tau^{\nu/2}\sum_{n=1}^N\norm{\phi\h\nn-\phi\h\no}_{H^1\trkla{\Om}}^2}^p}\\
&\quad+C\expected{\rkla{\tau^{-2\nu^2/\trkla{8-6\nu}}\sum_{n=1}^N\norm{\phi\h\nn-\phi\h\no}_{L^2\trkla{\Om}}^{\trkla{8-2\nu}/\trkla{4-3\nu}}}^p}\\
&\quad+C\rkla{\expected{1+\max_{0\leq n\leq N}\norm{\phi\h\nn}_{H^1\trkla{\Om}}^{\trkla{5/2+\delta}p}}\expected{\rkla{\tau^{\trkla{1/4-\delta/2}}\sum_{n=1}^N\norm{\phi\h\nn-\phi\h\no}_{H^1\trkla{\Om}}^2}^{2p}}}^{1/2}\\
&\quad+C\rkla{\expected{1+\max_{0\leq n\leq N}\norm{\phi\h\nn}_{H^1\trkla{\Om}}^{\trkla{5/2+\delta}p}}\expected{\rkla{\tau^{-\tfrac{\trkla{2-\delta}\trkla{1/2-\delta}}{2\delta}}\sum_{n=1}^N\norm{\phi\h\nn-\phi\h\no}_{L^2\trkla{\Om}}^{1/\delta}}^{2p}}}^{1/2}\\
&\leq \,C\tau^{p\nu/2}+C\tau^{\trkla{1/4-\delta/2}p}\,.
\end{split}
\end{align}
For $R_{3,n}$, we obtain using Hölder's inequality, \ref{item:potentialF}, Young's inequality, and the results from Lemma \ref{lem:energy} that
\begin{multline}
\expected{\abs{\sum_{n=1}^N R_{3,n}}^p}\leq\, C\expected{\rkla{\sum_{n=1}^N\tau\norm{\Ih{F^\prime\trkla{\phi\h\no}}}_{L^2\trkla{\Om}}^2\norm{\phi\h\nn-\phi\h\no}_{L^2\trkla{\Om}}\norm{\mu\h\nn}_{L^2\trkla{\Om}}}^p}\\
\leq\,C\expected{\rkla{1+\max_{1\leq n\leq N}\norm{\phi\h\no}_{H^1\trkla{\Om}}^6}^p\tau^{p/2}\rkla{\sum_{n=1}^N\tau\norm{\mu\h\nn}_{L^2\trkla{\Om}}^2+\sum_{n=1}^N\norm{\phi\h\nn-\phi\h\no}_{L^2\trkla{\Om}}^2}^p}\\
\leq C\tau^{p/2}\,.
\end{multline}
The remaining terms $R_{4,n},R_{5,n},R_{6,n}$ can be estimated in a similar manner as $R_{2,n}$ by applying the Gagliardo--Nirenberg inequality and the growth conditions stated in \ref{item:potentialF}.
This provides
\begin{align}
\begin{split}
\expected{\abs{\sum_{n=1}^N R_{4,n}}^p}\leq&\, C\expected{\rkla{\rkla{1+\max_{0\leq n\leq N}\norm{\phi\h\nn}_{H^1\trkla{\Omega}}^5}\sum_{n=1}^N\tau^{1/2}\norm{\phi\h\nn-\phi\h\no}_{H^1\trkla{\Om}}^2}^p}\\
&\, +C\expected{\rkla{\rkla{1+\max_{0\leq n\leq N}\norm{\phi\h\nn}_{H^1\trkla{\Omega}}^5}\sum_{n=1}^N\tau^{-1/2}\norm{\phi\h\nn-\phi\h\no}_{L^2\trkla{\Om}}^{4}}^p}\\
\leq&\,C\tau^{p/2}\,,
\end{split}\\
\begin{split}
\expected{\abs{\sum_{n=1}^N R_{5,n}}^p}\leq&\,C\expected{\rkla{\sum_{n=1}^N\norm{\Ih{F^{\prime\prime}\trkla{\varphi_3}}}_{L^3\trkla{\Om}}^2\norm{\phi\h\nn-\phi\h\no}_{H^1\trkla{\Om}}^{3/2}\norm{\phi\h\nn-\phi\h\no}_{L^2\trkla{\Om}}^{5/2}}^p}\\
\leq&\,C\expected{\rkla{1+\max_{0\leq n\leq N}\norm{\phi\h\nn}_{H^1\trkla{\Om}}^{5p}}\rkla{\sum_{n=1}^N\norm{\phi\h\nn-\phi\h\no}_{H^1\trkla{\Om}}\norm{\phi\h\nn-\phi\h\no}_{L^2\trkla{\Om}}^2}^p}\\
\leq&\,C\tau^{p/2}\,,
\end{split}\\
\begin{split}
\expected{\abs{\sum_{n=1}^N R_{6,n}}^p}\leq&\,C\expected{\rkla{\max_{0\leq n\leq N}\rkla{1+\norm{\phi\h\nn}_{H^1\trkla{\Om}}^9}\sum_{n=1}^N\norm{\phi\h\nn-\phi\h\no}_{L^2\trkla{\Om}}^3}^p}\\
\leq&\, C\tau^{p/2}\,.
\end{split}
\end{align}
Combining the above estimates provides the result.
\end{proof}
\section{Compactness properties of discrete solutions}\label{sec:compactness}
The goal of this section is to prove tightness of the laws of the discrete solutions and to deduce the existence of weakly and strongly converging subsequences.

Using the time-index-free notation defined in \eqref{eq:deftimeinterpol}, we can restate the regularity results obtained in the last section as follows:
\begin{subequations}\label{eq:regularity}
\begin{align}\label{eq:regularity:energy}
\begin{split}
\norm{\phi\h\tpm}_{L^{2\p}\trkla{\Omega;L^\infty\trkla{0,T;H^1\trkla{\Om}}}} +\norm{r\h\tpm}_{L^{2\p}\trkla{\Omega;L^\infty\trkla{0,T}}}+\norm{\mu\h\tp}_{L^{2\p}\trkla{\Omega;L^2\trkla{0,T;L^2\trkla{\Om}}}}&\\
+\tau^{-1/2}\norm{\phi\h\tp-\phi\h\tm}_{L^{2\p}\trkla{\Omega;L^2\trkla{0,T;H^1\trkla{\Om}}}}+\tau^{-1/2}\norm{r\h\tp-r\h\tm}_{L^{2\p}\trkla{\Omega;L^2\trkla{0,T}}}&\leq C\,,
\end{split}
\end{align}
\begin{align}\label{eq:regularity:discLap}
\norm{\Delta\h\phi\h\tp}_{L^{2q}\trkla{L^2\trkla{0,T;L^2\trkla{\Om}}}}\leq C\,,
\end{align}
\begin{align}\label{eq:regularity:hoelder}
\norm{\phi\h\tl}_{L^{2\alpha}\trkla{\Omega; N^{1/2,2\alpha}\trkla{0,T;L^2\trkla{\Om}}}}+\norm{\phi\h\tl}_{L^{2\alpha}\trkla{\Omega;C^{0,\trkla{\alpha-1}/\trkla{2\alpha}}\trkla{\tekla{0,T};L^2\trkla{\Om}}}}\leq C\,,
\end{align}
\begin{align}\label{eq:regularity:muerror}
\norm{\Xi\h\tp}_{L^q\trkla{\Omega;L^2\trkla{0,T;L^2\trkla{\Om}}}}\leq C\tau^{1/2}\,,
\end{align}
\end{subequations}
for $\p,q\in[1,\infty)$ and $\alpha\in\trkla{1,\infty}$.
While \eqref{eq:regularity:energy} and \eqref{eq:regularity:discLap} are direct consequences of Lemma \ref{lem:energy} and Corollary \ref{cor:h2}, \eqref{eq:regularity:hoelder} follows from Lemma 3.2 in \cite{Banas2013}, which guarantees that Lemma \ref{lem:nikolskii} is sufficient to show that $\phi\h\tl$ is in the corresponding Nikolskii space, and the embedding result in \cite{Simon1990}.
The final estimate \eqref{eq:regularity:muerror} follows from Corollary \ref{cor:muerror}.\\
We shall now identify almost surely converging subsequences by applying Jakubowski's generalization of Skorokhod's theorem (cf.~\cite{Jakubowski1998}).
In particular, we are interested in the convergence properties of $\trkla{\phi\h\tl,r\h\tl,\Delta\h\phi\h\tp,\mu\h\tp}$ and the linear interpolation $\bs{\xi}\h\tl$ of $\tgkla{\bs{\xi}\h^{m,\tau}}_m$.
We want to remark that these time-continuous processes $\bs{\xi}\h^\tau$ are not martingales.
Yet, we shall later show that they converge towards martingales. 
We start by establishing uniform estimates for $\bs{\xi}\h\tl$.
\begin{lemma}\label{lem:regWiener}
Let the assumptions \ref{item:filtration}-\ref{item:color} hold true.
Then, the piecewise linear processes $\bs{\xi}\h\tl$ satisfy
\begin{align}
\norm{\bs{\xi}\h\tl}_{L^{2p}\trkla{\Omega;C^{0,\trkla{p-1}/\trkla{2p}}\trkla{\tekla{0,T};H^1\trkla{\Om}}}}\leq C
\end{align}
for arbitrary $p\in\trkla{1,\infty}$ with a constant $C>0$ that depends on $p$ but not on $h$ or $\tau$.
\end{lemma}
\begin{proof}
Recalling the definition of $\bs{\xi}\h^{m,\tau}$ in \eqref{eq:deffinteprocess}, we obtain from Lemma \ref{lem:BDG}
\begin{multline}
\expected{\sum_{n=0}^{N-l}\tau\norm{\bs{\xi}\h^{m+l,\tau}-\bs{\xi}\h^{m,\tau}}_{H^1\trkla{\Om}}^{2p}}=\expected{\sum_{m=0}^{N-l}\tau\norm{\sum_{n=m+1}^{m+l}\sqrt{\tau}\sum_{k\in\mathds{Z}\h}\lambda_k\g{k}\xi\h^{n,\tau}}_{H^1\trkla{\Om}}^{2p}}\\
\leq C\sum_{m=0}^{N-l}\tau\trkla{l\tau}^{p-1}\sum_{n=m+1}^{m+l}\tau\expected{\rkla{\sum_{k\in\mathds{Z}\h}\norm{\lambda_k\g{k}}_{H^1\trkla{\Om}}^2}^p}\leq C\trkla{l\tau}^p\,.
\end{multline}
According to Lemma 3.2 in \cite{Banas2013}, this is sufficient to establish the result.
\end{proof}

The bound established in Lemma \ref{lem:regWiener} in particular provides the tightness of the laws of $\trkla{\bs{\xi}\h^\tau}_{h,\tau}$ in $C\trkla{\tekla{0,T};L^2\trkla{\Om}}$.
In the next lemma, we show that the laws of $\trkla{\phi\h\tl}_{h,\tau}$ are tight on $C\trkla{\tekla{0,T};L^s\trkla{\Om}}$ with $s\in[1,\tfrac{2d}{d-2})$:
\begin{lemma}\label{lem:tightness}
Let $\trkla{\phi\h\tl}_{h,\tau}$ be a family of continuous, piecewise linear processes that satisfy the bounds stated in \eqref{eq:regularity}. 
Then the family of laws $\nu_{\phi\h\tl}$ is tight on $C\trkla{\tekla{0,T};L^s\trkla{\Om}}$ with $s\in[1,\tfrac{2d}{d-2})$.
\end{lemma}
\begin{proof}
Due to the well-known compactness theorem by Simon (cf.~\cite{Simon1987}), the closed ball $\overline{B}_R$ in $L^\infty\trkla{0,T;H^1\trkla{\Om}}\cap C^{0,\trkla{\alpha-1}/\trkla{2\alpha}}\trkla{\tekla{0,T};L^2\trkla{\Om}}$ is a compact subset of $C\trkla{\tekla{0,T};L^s\trkla{\Om}}$ for $s\in[1,\tfrac{2d}{d-2})$.
Furthermore, we have for any $R>0$
\begin{align}
\begin{split}
\nu_{\phi\h\tl}\trkla{\mathcal{X}_\phi\setminus\overline{B}_R}=&\,\prob{\norm{\phi\h\tl}_{L^\infty\trkla{0,T;H^1\trkla{\Om}}}^{2\alpha}+\norm{\phi\h\tl}_{C^{0,\trkla{\alpha-1}/\trkla{2\alpha}}\trkla{\tekla{0,T};L^2\trkla{\Om}}}^{2\alpha}>R^{2\alpha}}\\
\leq&\,R^{-2\alpha}\expected{\norm{\phi\h\tl}_{L^\infty\trkla{0,T;H^1\trkla{\Om}}}^{2\alpha}+\norm{\phi\h\tl}_{C^{0,\trkla{\alpha-1}/\trkla{2\alpha}}\trkla{\tekla{0,T};L^2\trkla{\Om}}}^{2\alpha}}\,,
\end{split}
\end{align}
which shows the tightness of $\trkla{\nu_{\phi\h\tl}}_{h,\tau}$.
\end{proof}

As closed balls in $L^2\trkla{0,T}$ and $L^2\trkla{0,T;L^2\trkla{\Om}}$ are compact in the weak topology, the laws of $r\h\tl$, $\Delta\h\phi\h\tp$, and $\mu\h\tp$ are tight in the spaces $L^2\trkla{0,T}_{\weaktop}$, $L^2\trkla{0,T;L^2\trkla{\Om}}_{\weaktop}$, and $L^2\trkla{0,T;L^2\trkla{\Om}}_{\weaktop}$, due to Markov's inequality and the bounds collected in \eqref{eq:regularity}.
Hence, the joint laws of $\trkla{\phi\h\tl}_{h,\tau}$, $\trkla{r\h\tl}_{h,\tau}$, $\trkla{\Delta\h\phi\h\tp}_{h,\tau}$, $\trkla{\mu\h\tp}_{h,\tau}$, and $\trkla{\bs{\xi}\h\tl}_{h,\tau}$ are tight on the path space
\begin{align}
\begin{split}
\mathcal{X}:=& C\trkla{\tekla{0,T};L^s\trkla{\Om}}\times L^2\trkla{0,T}_{\weaktop}\times L^2\trkla{0,T;L^2\trkla{\Om}}_{\weaktop}\\
&\quad\times L^2\trkla{0,T;L^2\trkla{\Om}}_{\weaktop}\times C\trkla{\tekla{0,T};L^2\trkla{\Om}}\,,
\end{split}
\end{align}
with $s\in[1,\tfrac{2d}{d-2})$.
Hence, we obtain the following convergence results:
\begin{theorem}\label{thm:jakubowski}
Let $\trkla{\phi\h\tl,r\h\tl,\Delta\h\phi\h\tp,\mu\h\tp}_{h,\tau}$ satisfy the estimates \eqref{eq:regularity} and let the family $\trkla{\bs{\xi}\h\tl}_{h,\tau}$ satisfy the bounds stated in Lemma \ref{lem:regWiener}.
Then, there exists a subsequence
\begin{align*}
\trkla{\phi_j,r_j,\Delta\hj\phi_j^+,\mu_j^+,\bs{\xi}_j}_j:=\trkla{\phi\hj^{\tau_j},r\hj^{\tau_j},\Delta\hj\phi\hj^{\tau_j,+},\mu\hj^{\tau_j,+},\bs{\xi}\hj^{\tau_j}}_j\,,
\end{align*}
a stochastic basis $\trkla{\widetilde{\Omega},\widetilde{\mathcal{A}},\widetilde{\Prob}}$, a sequence of random variables
\begin{align*}
\trkla{\widetilde{\phi}_j,\widetilde{r}_j,\Delta\hj\widetilde{\phi}_j^+,\widetilde{\mu}_j^+,\widetilde{\bs{\xi}}_j}\,:\,\widetilde{\Omega}\rightarrow\mathcal{X}\,,
\end{align*}
and random variables
\begin{align*}
\trkla{\widetilde{\phi},\widetilde{r},\widetilde{L},\widetilde{\mu},\widetilde{W}}\,:\,\widetilde{\Omega}\rightarrow\mathcal{X}
\end{align*}
such that the following holds:
\begin{itemize}
\item The law of $\trkla{\widetilde{\phi}_j,\widetilde{r}_j, \Delta\hj\widetilde{\phi}_j^+,\widetilde{\mu}_j^+,\widetilde{\bs{\xi}}_j}$ on $\mathcal{X}$ under $\widetilde{\Prob}$ coincides for any $j\in\mathds{N}$ with the law of $\trkla{\phi_j,r_j,\Delta\hj\phi_j^+,\mu_j^+,\bs{\xi}_j}$ under $\Prob$.
\item The sequence $\trkla{\widetilde{\phi}_j,\widetilde{r}_j,\Delta\hj\widetilde{\phi}_j^+,\widetilde{\mu}_j^+,\widetilde{\bs{\xi}}_j}$ converges $\widetilde{\Prob}$-almost surely towards $\trkla{\widetilde{\phi},\widetilde{r},\widetilde{L},\widetilde{\mu},\widetilde{W}}$ in the topology of $\mathcal{X}$.
\end{itemize}
\end{theorem}
\begin{proof}
As the combined laws on $\mathcal{X}$ are tight, we can deduce the existence of a stochastic basis and a sequence of random variables $\trkla{\widetilde{\phi}_j,\widetilde{r}_j,\widetilde{L}_j^+,\widetilde{\mu}_j^+,\widetilde{\bs{\xi}}_j}$ on $\widetilde{\Omega}$ with the desired convergence properties by applying Jakubowski's theorem (cf.~\cite{Jakubowski1998}).
Hence, it remains to identify $\widetilde{L}_j^+$ with $\Delta\hj\widetilde{\phi}_j^+$.
As the laws of $\widetilde{\phi}_j$ and $\phi_j$ coincide, $\widetilde{\phi}_j$ is $\widetilde{\Prob}$-almost surely a piecewise linear time-interpolation of a $\Uh$-valued finite element function.
Hence, pointwise in time evaluations are possible and the continuous dependence of $\Delta\hj\phi_j^+$ on $\phi_j$ completes the argument.
\end{proof}

\begin{remark}
The restriction to subsequences in Theorem \ref{thm:jakubowski} is necessary, as Jakubowski's theorem relies on a generalization of Prokhorov's theorem, i.e.~the uniform bounds on $\trkla{\phi\h\tl,r\h\tl,\Delta\h\phi\h\tp,\mu\h\tp}_{h,\tau}$ are used to deduce convergence in distribution for a subsequence. For $\trkla{\bs{\xi}\h\tl}_{h,\tau}$ convergence in distribution can already be deduced from Donsker's invariance theorem.
\end{remark}

As the random variables $\widetilde{\phi}_j$ and $\widetilde{r}_j$ introduced in Theorem \ref{thm:jakubowski} are $\widetilde{\Prob}$-almost surely continuous and piecewise linear with respect to time, we can evaluate these functions at the nodes of the time grid and recover the remaining piecewise constant time-interpolants.
Analogously to \eqref{eq:def:errormu}, we introduce the $\widetilde{\Prob}$-almost surely piecewise constant in time finite element functions $\widetilde{\Xi}_j^+$.
Due to the identity of laws, these sequence satisfies the bounds stated in \eqref{eq:regularity:muerror} with respect to the new probability space $\trkla{\widetilde{\Omega},\widetilde{\mathcal{A}},\widetilde{\Prob}}$.
Hence, we can establish the following additional convergence properties:

\begin{lemma}\label{lem:convergence}
Let $\trkla{\widetilde{\phi}_j\pml}_{j\in\mathds{N}}$, $\trkla{\widetilde{r}_j\pml}_{j\in\mathds{N}}$, $\trkla{\widetilde{\mu}_j^+}_{j\in\mathds{N}}$, and $\trkla{\widetilde{\Xi}_j^+}_{j\in\mathds{N}}$ be the sequences of stochastic processes defined on the basis of the sequence from Theorem \ref{thm:jakubowski}, and let the assumptions  \ref{item:timedisc}, \ref{item:spatialdisc}, \ref{item:potentialF}, \ref{item:initial}, \ref{item:sigma}, and \ref{item:filtration}--\ref{item:color} hold true.
Then, there exists a process
\begin{align*}
\begin{split}
\widetilde{\phi}\in L^{2\p}&\trkla{\widetilde{\Omega};L^\infty\trkla{0,T;H^1\trkla{\Om}}}\cap L^{2\p}\trkla{\widetilde{\Omega};C^{0,\trkla{\p-1}/\trkla{2\p}}\trkla{\tekla{0,T};L^2\trkla{\Om}}}\\
&\cap L^{2\p}\trkla{\widetilde{\Omega};L^2\trkla{0,T;H^2\trkla{\Om}}}
\end{split}
\end{align*}
for any $\p\in\trkla{1,\infty}$ such that
\begin{subequations}
\begin{align}
\widetilde{\phi}_j&\rightarrow\widetilde{\phi}&&\text{in~} L^{p}\trkla{\widetilde{\Omega};C\trkla{\tekla{0,T};L^q\trkla{\Om}}}\,,\label{eq:conv:strong:cont}\\
\widetilde{\phi}_j\pml&\rightarrow\widetilde{\phi}&&\text{in~}L^{p}\trkla{\widetilde{\Omega};L^{p}\trkla{0,T};L^q\trkla{\Om}}\,,\label{eq:conv:strong:lp}\\
\widetilde{\phi}_j^+&\rightarrow\widetilde{\phi}&&\text{in~}L^p\trkla{\widetilde{\Omega};L^2\trkla{0,T;W^{1,q}\trkla{\Om}}}\,,\label{eq:conv:strong:w1p}\\
\widetilde{\phi}_j\pml&\weakstar\widetilde{\phi}&&\text{in~} L^{p}_{\operatorname{weak-}\trkla{*}}\trkla{\widetilde{\Omega};L^\infty\trkla{0,T;H^1\trkla{\Om}}}\,,\label{eq:conv:weakstar}\\
\Delta\hj\widetilde{\phi}_j^+&\weak\Delta\widetilde{\phi}&&\text{in~}L^{p}\trkla{\widetilde{\Omega};L^2\trkla{0,T;L^2\trkla{\Om}}}\,\label{eq:conv:weakLap},\\
\iO\Ihj{F\trkla{\widetilde{\phi}_j\pml}}\dx&\rightarrow \iO F\trkla{\widetilde{\phi}}\dx&&\text{in~}L^p\trkla{\widetilde{\Omega};L^p\trkla{0,T}}\,,\label{eq:conv:potential}\\
\widetilde{r}_j\pml&\rightarrow\sqrt{\iO F\trkla{\widetilde{\phi}}\dx}&&\text{in~}L^p\trkla{\widetilde{\Omega};L^p\trkla{0,T}}\label{eq:conv:strongr}\,,\\
\widetilde{r}_j\pml&\weakstar \sqrt{\iO F\trkla{\widetilde{\phi}}\dx}&&\text{in~}L^p_{\operatorname{weak-}\trkla{*}}\trkla{\widetilde{\Omega};L^\infty\trkla{0,T}}\label{eq:conv:weakstarr}\,,\\
\widetilde{\mu}_j^+&\weak-\Delta\widetilde{\phi}+F^\prime\trkla{\widetilde{\phi}}&&\text{in~}L^p\trkla{\widetilde{\Omega};L^2\trkla{0,T;L^2\trkla{\Om}}}\,,\label{eq:conv:weakmu}\\
\widetilde{\Xi}_j^+&\rightarrow0&&\text{in~}L^p\trkla{\widetilde{\Omega};L^2\trkla{0,T;L^2\trkla{\Om}}}\,\label{eq:conv:muerrorstrong}
\end{align}
\end{subequations}
as $j\rightarrow\infty$ for $p\in[1,\infty)$ and $q\in[1,\tfrac{2d}{d-2})$ after restriction to an appropriate subsequence.
\end{lemma}

\begin{proof}
Due to Theorem \ref{thm:jakubowski}, we have $\widetilde{\phi}_j\rightarrow\widetilde{\phi}$ in $C\trkla{\tekla{0,T};L^q\trkla{\Om}}$ $\widetilde{\Prob}$-almost surely.
By Vitali's convergence theorem and the bounds in \eqref{eq:regularity}, we obtain the strong convergence expressed in \eqref{eq:conv:strong:cont}.
Restricting ourselves to a suitable subsequence, we note that $\widetilde{\phi}_j\rightarrow\widetilde{\phi}$ pointwise almost everywhere in $\widetilde{\Omega}\times\tekla{0,T}\times\Om$.
As estimate \eqref{eq:regularity:energy} yields $\norm{\widetilde{\phi}_j\pml-\widetilde{\phi}_j}_{L^{2\p}\trkla{\widetilde{\Omega};L^2\trkla{0,T;H^1\trkla{\Om}}}}\rightarrow0$ for $j\rightarrow\infty$, we deduce the existence of a subsequence $\trkla{\widetilde{\phi}_j\pml}_{j\in\mathds{N}}$ converging pointwise almost everywhere towards $\widetilde{\phi}$.
Again, a combination of Vitali's convergence theorem and the uniform bounds in \eqref{eq:regularity:energy} provide \eqref{eq:conv:strong:lp}.\\
In order to establish the strong convergence of $\nabla\widetilde{\phi}_j^+$, we use the arguments from Lemma 5.1 in \cite{Metzger2020}.
We first consider the case $d=3$:
Applying Hölder's inequality, we obtain for any $\tilde{\varepsilon}>0$ and $\varrho=\tfrac{\tilde{\varepsilon}}{12-2\tilde{\varepsilon}}$
\begin{multline}\label{eq:tmp:convw1p:1}
\expectedt{\norm{\widetilde{\phi}_j^+-\widetilde{\phi}}_{L^2\trkla{0,T;W^{1,p-\tilde{\varepsilon}}\trkla{\Om}}}^p}\leq C\expectedt{\norm{\widetilde{\phi}_j^+-\widetilde{\phi}}_{L^2\trkla{0,T;W^{1,6}\trkla{\Om}}}^{p\trkla{2-2\varrho}}\norm{\widetilde{\phi}_j^+-\widetilde{\phi}}_{L^2\trkla{0,T;H^1\trkla{\Om}}}^{2\varrho p}}\\
\leq C\expectedt{\norm{\widetilde{\phi}_j^+-\widetilde{\phi}}_{L^2\trkla{0,T;W^{1,6}\trkla{\Om}}}^{2p\trkla{2-2\varrho}}}^{1/2}\expectedt{\norm{\widetilde{\phi}_j^+-\widetilde{\phi}}_{L^2\trkla{0,T;H^1\trkla{\Om}}}^{4\varrho p}}^{1/2}\,.
\end{multline} 
Recalling a discrete version of the Gagliardo--Nirenberg inequality (see e.g.~\cite{Metzger2020}) which states
\begin{align}
\norm{\zeta\h}_{W^{1,p}\trkla{\Om}}\leq C\norm{\Delta\h\zeta\h}_{L^2\trkla{\Om}}^\alpha\norm{\zeta\h}_{H^1\trkla{\Om}}^{1-\alpha}+C\norm{\zeta\h}_{H^1\trkla{\Om}}
\end{align}
for all $\zeta\h\in\Uh$ and $\alpha=\tfrac{3p-6}{2p}$ with $p\in\tekla{2,6}$, we note that the first factor on the right-hand side of \eqref{eq:tmp:convw1p:1} is uniformly bounded.
It therefore suffices to show strong convergence in $L^p\trkla{\widetilde{\Omega};L^2\trkla{0,T;H^1\trkla{\Om}}}$ for $p<\infty$.
The case $d=2$ can be treated using similar arguments with different exponents.
Therefore, it remains to establish the strong convergence in $L^p\trkla{\widetilde{\Omega};L^2\trkla{0,T;H^1\trkla{\Om}}}$.
Denoting the Ritz projection onto $\Uhj$ by $\Ritzop$, we obtain
\begin{align}
\lim_{j\rightarrow\infty}\norm{\widetilde{\phi}_j^+-\widetilde{\phi}}_{L^p\trkla{\widetilde{\Omega};L^2\trkla{0,T;H^1\trkla{\Om}}}}=\lim_{j\rightarrow\infty} \norm{\widetilde{\phi}_j^+-\Ritz{\widetilde{\phi}}}_{L^p\trkla{\widetilde{\Omega};L^2\trkla{0,T;H^1\trkla{\Om}}}}\,,
\end{align}
due to the strong convergence of $\Ritz{\widetilde{\phi}}$ towards $\widetilde{\phi}$ (cf.~Theorem A.2 in \cite{Girault1986}).
Following the arguments in \cite{Metzger2020}, we have
\begin{multline}
\norm{\nabla\trkla{\widetilde{\phi}_j^+-\Ritz{\widetilde{\phi}}}}_{L^2\trkla{0,T;L^2\trkla{\Om}}}^2\\
\leq \rkla{\norm{\Delta\hj\widetilde{\phi}_j}_{L^2\trkla{0,T;L^2\trkla{\Om}}}+\norm{\Delta\widetilde{\phi}}_{L^2\trkla{0,T;L^2\trkla{\Om}}}}\norm{\widetilde{\phi}_j^+-\Ritz{\widetilde{\phi}}}_{L^2\trkla{0,T;L^2\trkla{\Om}}}\,.
\end{multline}
Therefore, \eqref{eq:conv:strong:w1p} follows from \eqref{eq:conv:strong:lp}.\\
The weak* convergence result expressed in \eqref{eq:conv:weakstar} follows directly from the bounds collected in \eqref{eq:regularity}.
The weak convergence of the discrete Laplacian towards a limit process $\widetilde{L}$ also follows from the bounds stated in \eqref{eq:regularity:discLap} and the almost sure convergence deduced in Theorem \ref{thm:jakubowski}.
Hence, it remains to identify the limit $\widetilde{L}$ with $\Delta\widetilde{\phi}$.
For this reason, we chose a test function $\psi\in L^{p/\trkla{p-1}}\trkla{\widetilde{\Omega};L^2\trkla{0,T;C^\infty_0\trkla{\Om}}}$ and denote its Scott--Zhang interpolations by $\psi\hj$.
These interpolations exhibit some useful properties (see e.g.~\cite[Section 1.6.2]{Ern2004} and \cite{ScottZhang90}):
On the one hand, they preserve the homogeneous boundary conditions.
On the other hand, the Scott--Zhang interpolation is $H^1\trkla{\Om}$-stable and also converges strongly in $H^1\trkla{\Om}$ on quasi-uniform partitions.
Hence, we have
\begin{multline}
\expectedt{\int_0^T\iO \widetilde{L}\psi\dx\dt}\leftarrow \expectedt{\int_0^T\iO \Delta\hj\widetilde{\phi}_j^+\psi\hj\dx\dt}\\
=\expectedt{\int_0^T\iO\Ihj{\Delta\hj\widetilde{\phi}_j^+\psi\hj}\dx\dt}-\expectedt{\int_0^T\iO\trkla{1-\Ihjop}\gkla{\Delta\hj\widetilde{\phi}_j^+\psi\h}\dx\dt}\\
=-\expectedt{\int_0^T\iO\nabla\widetilde{\phi}_j^+\cdot\nabla\psi\h\dx\dt}-\expectedt{\int_0^T\iO\trkla{1-\Ihjop}\gkla{\Delta\hj\widetilde{\phi}_j^+\psi\h}\dx\dt}\\
\rightarrow -\expectedt{\int_0^T\iO\nabla\widetilde{\phi}\cdot\nabla\psi\dx\dt}=\expectedt{\int_0^T\iO\Delta\widetilde{\phi}\psi\dx\dt}\,,
\end{multline}
as the error terms vanish due to the stability of the Scott--Zhang interpolation and Lemma \ref{lem:interpolation}.\\
To establish the strong convergence \eqref{eq:conv:potential}, we follow the lines of \cite{Metzger2023}, Corollary 4.1, and first show that the nodal interpolation operators $\Ihjop$ are negligible.
For this reason, we use estimate (4.8) from \cite{Metzger2023} which reads
\begin{align}
\esssup_{t\in\trkla{0,T}}\iO\abs{\trkla{1-\Ihjop}\gkla{F\trkla{\widetilde{\phi}_j\pml}}}\dx\leq Ch_j\tabs{\Om}+C h_j\esssup_{t\in\trkla{0,T}}\norm{\widetilde{\phi}_j\pml}_{H^1\trkla{\Om}}^6\,.
\end{align}
Therefore, we have
\begin{align}
\norm{\iO\abs{\trkla{1-\Ihjop}\gkla{F\trkla{\widetilde{\phi}_j\pml}}}\dx}_{L^p\trkla{\widetilde{\Omega};L^p\trkla{0,T}}}\rightarrow 0
\end{align}
for all $p\in[1,\infty)$.
As the strong convergence of $\widetilde{\phi}_j\pml$ implies the existence of a subsequence converging pointwise almost everywhere towards $\widetilde{\phi}$, we may use the continuity of $F$, the growth condition in \ref{item:potentialF}, and Vitali's convergence theorem to obtain \eqref{eq:conv:potential}.\\
The strong convergence expressed in \eqref{eq:conv:strongr} follows from \eqref{eq:conv:potential} and Lemma \ref{lem:SAVerror}.
The weak*-convergence in \eqref{eq:conv:weakstarr} can then be deduced from the bounds in \eqref{eq:regularity:energy} and \eqref{eq:conv:strongr}.\\
Recalling Corollary \ref{cor:muerror} and \eqref{eq:def:errormu}, we immediately obtain \eqref{eq:conv:muerrorstrong}.
Hence, it remains to show that the expression $\frac{\widetilde{r}_j^+}{\sqrt{E\hj\trkla{\widetilde{\phi}_j^-}}}\Ihj{F^\prime\trkla{\widetilde{\phi}_j^-}}$ converges weakly towards $F^\prime\trkla{\widetilde{\phi}}$ to establish \eqref{eq:conv:weakmu}.
The strong convergence of $\Ihj{F^\prime\trkla{\widetilde{\phi}_j^-}}$ towards $F^\prime\trkla{\widetilde{\phi}}$ follows by similar arguments as \eqref{eq:conv:potential}.
Now, \eqref{eq:conv:weakmu} follows from \eqref{eq:conv:weakstarr} and \eqref{eq:conv:potential}.
\end{proof}
\begin{corollary}\label{cor:identification}
Let the assumptions of Theorem \ref{thm:jakubowski} and Lemma \ref{lem:convergence} hold true.
Then, the random variables $\widetilde{r}$, $\widetilde{L}$, and $\widetilde{\mu}$ introduced in Theorem \ref{thm:jakubowski} can be identified with $\sqrt{\iO F\trkla{\widetilde{\phi}}\dx}$, $\Delta\widetilde{\phi}$, and $-\Delta\widetilde{\phi}+F^\prime\trkla{\widetilde{\phi}}$.
\end{corollary}

\section{Passage to the limit}\label{sec:limit}
In this section, we use a more precise notation for our time grid and write $t_j^n$ ($n=0,\ldots,N_j$) for the nodes of the equidistant grid obtained using the step size $\tau_j$.
We shall denote the increments $\widetilde{\bs{\xi}}_j\trkla{t_j\nn}-\widetilde{\bs{\xi}}_j\trkla{t_j\no}$ by $\sinctilde{n}$.
With the stochastic processes defined in the last section, we can now rewrite \eqref{eq:modeldisc:phimu} as
\begin{multline}\label{eq:timecont}
\iO\Ihj{\rkla{\widetilde{\phi}_j\trkla{t}-\widetilde{\phi}_j^-}\psi\hj}\dx+\trkla{t-t\no_j}\iO\Ihj{\widetilde{\mu}_j^+\psi\hj}\dx\\
 = \frac{t-t_j\no}{\tau_j}\iO\Ihj{\Phi\hj\trkla{\widetilde{\phi}^-_j}\sinctilde{n}\psi\hj}\dx
\end{multline}
for $t\in[t_j\no,t_j\nn)$ and $\psi\hj\in U\hj$.
Due to the structure of \eqref{eq:timecont}, an equivalent formulation is
\begin{align}
\widetilde{\phi}_j\trkla{t}-\widetilde{\phi}_j^- +\trkla{t-t_j\no}\widetilde{\mu}_j^+=\frac{t-t_j\no}{\tau_j}\Phi\hj\trkla{\widetilde{\phi}_j^-}\sinctilde{n}
\end{align}
with both sides being elements of $\Uhj$.
While the convergence of the left-hand side can be shown using the results from Lemma \ref{lem:convergence}, identifying the limit of the right-hand side as a suitable It\^o-integral is more intricate.
We start by introducing the filtration $\trkla{\widetilde{\mathcal{F}}_t}_{t\in\tekla{0,T}}$ as the augmentation of the filtration generated by $\widetilde{\phi}$ and $\widetilde{W}$ and show that $\widetilde{W}$ is a $\mathcal{Q}$-Wiener process with respect to $\trkla{\widetilde{\mathcal{F}}_t}_{t\in\tekla{0,T}}$.
\begin{lemma}\label{lem:QWiener}
The process $\widetilde{W}$ obtained in Theorem \ref{thm:jakubowski} is a $\mathcal{Q}$-Wiener process adapted to the augmentation of the filtration $\trkla{\widetilde{\mathcal{F}}_t}_{t\in\tekla{0,T}}$ and can be written as
\begin{align*}
\widetilde{W}=\sum_{k\in\mathds{Z}}\lambda_k\g{k}\widetilde{\beta}_k\,.
\end{align*}
Here, $\trkla{\widetilde{\beta}_k}_{k\in\mathds{Z}}$ is a family of independently and identically distributed Brownian motions with respect to the augmentation of $\trkla{\widetilde{\mathcal{F}}_t}_{t\in\tekla{0,T}}$.
\end{lemma}
\begin{proof}
Due to Corollary \ref{cor:identification}, we have that $\widetilde{\mu}$ depends continuously on $\widetilde{\phi}$. Hence, we can add $\widetilde{\mu}$ to the generator without changing the filtration.
As the laws of $\bs{\xi}\hj^{\tau_j}$ and $\widetilde{\bs{\xi}}_j$ coincide, we have that $\widetilde{\bs{\xi}}_j$ is $\widetilde{\Prob}$-almost surely piecewise linear in time and satisfies
\begin{align}
\widetilde{\bs{\xi}}_j\trkla{t_j^m}=\sum_{n=1}^m\sqrt{\tau_j}\sum_{k\in\mathds{Z}}\lambda_k\g{k}\widetilde{\xi}_k^{n,\tau_j}\,,
\end{align}
with mutually independent random variables $\widetilde{\xi}_k^{n,\tau_j}$ satisfying \ref{item:randomvariables}.
Although $\widetilde{\bs{\xi}}_j$ are not martingales, we can use the $\widetilde{\Prob}$-almost sure convergence of $\widetilde{\phi}_j$, $\widetilde{\mu}_j^+$, and $\widetilde{\bs{\xi}}_j$ established in Theorem \ref{thm:jakubowski} to show that the limit process $\widetilde{W}$ is a martingale.
Following \cite[Lemma 5.8]{Ondrejat2022}, it suffices to thow that
\begin{align}
\expectedt{\rkla{\widetilde{W}\trkla{t_2}-\widetilde{W}\trkla{t_1}}\prod_{i=1}^q\prod_{\iota=1}^o f_{i\iota}\rkla{\Psi_{i\iota}\rkla{\widetilde{\phi}_{s_\iota},\widetilde{\mu}_{s_\iota},\widetilde{W}_{s_\iota}}}}=0
\end{align}
for all times $0\leq s_1<\ldots<s_q\leq t_1<t_2\leq T$, all $f_{i\iota}\in\mathcal{A}_{s_\iota}$, where $\mathcal{A}_{s_\iota}$ is a subset of real-valued functions on
\begin{align}
\widehat{\mathcal{X}}_{s_\iota}:=C\trkla{\tekla{0,s_\iota};L^s\trkla{\Om}}\times L^2\trkla{0,s_\iota;L^2\trkla{\Om}}_{\weaktop}\times C\trkla{\tekla{0,s_\iota};L^2\trkla{\Om}}
\end{align}
for which the $\sigma$-algebra of $\mathcal{A}_{s_\iota}$ equals the Borel $\sigma$-algebra on $\widehat{\mathcal{X}}_{s_\iota}$.
Here, we denoted the restriction of a process $Y$ onto the subinterval $\tekla{0,s_\iota}$ by $Y_{s_\iota}$.
As explained in Example 5.9 in \cite{Ondrejat2022}, these functions can always be chosen in a manner such that they are continuous with respect to weakly converging sequences.
Hence, by combining the uniform integrability of $\widetilde{\bs{\xi}}_j$ established in Lemma \ref{lem:regWiener} and a Vitali argument, we obtain
\begin{align}
\begin{split}
&\expectedt{\rkla{\widetilde{W}\trkla{t_2}-\widetilde{W}\trkla{t_1}}\prod_{i=1}^q\prod_{\iota=1}^o f_{i\iota}\rkla{\Psi_{i\iota}\rkla{\widetilde{\phi}_{s_\iota},\widetilde{\mu}_{s_\iota},\widetilde{W}_{s_\iota}}}}\\
&~=\lim_{j\rightarrow\infty} \expectedt{\rkla{\widetilde{\bs{\xi}}_j\trkla{t_2}-\widetilde{\bs{\xi}}_j\trkla{t_1}}\prod_{i=1}^q\prod_{\iota=1}^o f_{i\iota}\rkla{\Psi_{i\iota}\rkla{\rkla{\widetilde{\phi}_j}_{s_\iota},\rkla{\widetilde{\mu}_j^+}_{s_\iota},\rkla{\widetilde{\bs{\xi}}_j}_{s_\iota}}}}\\
&~=\lim_{j\rightarrow\infty} \expectedt{\rkla{\widetilde{\bs{\xi}}_j\trkla{t_j^n}-\widetilde{\bs{\xi}}_j\trkla{t_j^m}}\prod_{i=1}^q\prod_{\iota=1}^o f_{i\iota}\rkla{\Psi_{i\iota}\rkla{\rkla{\widetilde{\phi}_j}_{s_\iota},\rkla{\widetilde{\mu}_j^+}_{s_\iota},\rkla{\widetilde{\bs{\xi}}_j}_{s_\iota}}}}\\
&\qquad+\lim_{j\rightarrow\infty} \expectedt{\rkla{\widetilde{\bs{\xi}}_j\trkla{t_2}-\widetilde{\bs{\xi}}_j\trkla{t_j^n}}\prod_{i=1}^q\prod_{\iota=1}^o f_{i\iota}\rkla{\Psi_{i\iota}\rkla{\rkla{\widetilde{\phi}_j}_{s_\iota},\rkla{\widetilde{\mu}_j^+}_{s_\iota},\rkla{\widetilde{\bs{\xi}}_j}_{s_\iota}}}}\\
&\qquad-\lim_{j\rightarrow\infty} \expectedt{\rkla{\widetilde{\bs{\xi}}_j\trkla{t_1}-\widetilde{\bs{\xi}}_j\trkla{t_j^m}}\prod_{i=1}^q\prod_{\iota=1}^o f_{i\iota}\rkla{\Psi_{i\iota}\rkla{\rkla{\widetilde{\phi}_j}_{s_\iota},\rkla{\widetilde{\mu}_j^+}_{s_\iota},\rkla{\widetilde{\bs{\xi}}_j}_{s_\iota}}}}\\
&~=\lim_{j\rightarrow\infty}\expectedt{\sum_{a=m+1}^n\sqrt{\tau_j}\sum_{k\in\mathds{Z}}\lambda_k\g{k}\widetilde{\xi}_j^{a,\tau_j}}=0\,.
\end{split}
\end{align}
Here, we compared $t_1$ and $t_2$ to grid points $t_j^n$ and $t_j^m$ satisfying $0\leq t_j^n-t_2\leq \tau$ and $0\leq t_j^m-t_1\leq\tau_j$ and used the continuity of $\widetilde{\bs{\xi}}_j$.
Defining $\widetilde{\beta}_k\trkla{t}:=\lambda_k^{-1}\iO\widetilde{W}\trkla{t,x}\g{k}\trkla{x}\dx$, we can use similar arguments to show that
\begin{align}
\widetilde{\beta}_k\trkla{t}\widetilde{\beta}_l\trkla{t}-\delta_{kl}t
\end{align}
with $\delta_{kl}$ denoting the usual Kronecker delta is a martingale.
Hence, by Levy's characterization of Brownian motions we obtain the result.
\end{proof}

Using similar arguments, we shall now show that the limit process of the left-hand side of \eqref{eq:timecont} is a martingale with respect to $\trkla{\widetilde{\mathcal{F}}_t}_{t\in\tekla{0,T}}$, identify its quadratic variation process and deduce a suitable representation as an It\^o-integral.\\
For fixed $v\in L^2\trkla{\Om}$, we consider the family of processes
\begin{align}\label{eq:defMjv}
\widetilde{M}_j^v\trkla{t}:=\iO\trkla{\widetilde{\phi}_j\trkla{t}-\widetilde{\phi}_j\trkla{0}}v\dx+\int_0^t\iO\widetilde{\mu}_j^+v\dx\ds\,.
\end{align}
A suitable candidate for the limit of $\widetilde{M}_j^v$ is the process
\begin{align}\label{eq:defMv}
\widetilde{M}^v\trkla{t}:=\iO\trkla{\widetilde{\phi}\trkla{t}-\widetilde{\phi}\trkla{0}}v\dx+\int_0^t\iO\widetilde{\mu}v\dx\ds\,.
\end{align}
\begin{lemma}\label{lem:martingale}
Let the assumptions of Lemma \ref{lem:convergence} hold true.
Then, the process $\widetilde{M}^v$ defined in \eqref{eq:defMv} with $v$ in $L^2\trkla{\Om}$ is a martingale with respect to the filtration $\trkla{\widetilde{\mathcal{F}}_t}_{t\in\tekla{0,T}}$.
\end{lemma}
\begin{proof}
Due to the strong convergence of $\widetilde{\phi}_j$ stated in \eqref{eq:conv:strong:cont} and the $\widetilde{\Prob}$-almost sure convergence of $\widetilde{\mu}_j^+$ in the weak topology together with a Vitali argument, we deduce the strong convergence of $\widetilde{M}_j^v\trkla{t}$ towards $\widetilde{M}^v\trkla{t}$ in $L^p\trkla{\widetilde{\Omega}}$ for $t\in\tekla{0,T}$.
Hence, we can repeat the arguments from the proof of Lemma \ref{lem:QWiener} and show that
\begin{align}
\begin{split}
\expectedt{\rkla{\widetilde{M}^v\trkla{t_2}-\widetilde{M}^v\trkla{t_1}}\mathfrak{f}^{qo}}=&\,\lim_{j\rightarrow\infty}\expectedt{\rkla{\widetilde{M}^v_j\trkla{t_j^n}-\widetilde{M}_j^v\trkla{t_j^m}}\mathfrak{f}_j^{qo}}\\
=&\,\lim_{j\rightarrow\infty}\expectedt{\sum_{a=m+1}^n\iO\Phi\hj\rkla{\widetilde{\phi}_j\trkla{t_j^{a-1}}}\sinctilde{a}v\dx}=0\,,
\end{split}
\end{align}
due to \ref{item:randomvariables}.
Here, we again used grid points $t_j^n$ and $t_j^m$ satisfying $0\leq t_j^n-t_2\leq\tau_j$ and $0\leq t_j^m-t_1\leq\tau_j$ and introduced the abbreviations
\begin{subequations}\label{eq:tmp:abbreviations}
\begin{align}
\mathfrak{f}^{qo}:=&\,\prod_{i=1}^q\prod_{\iota=1}^of_{i\iota}\rkla{\Psi_{i\iota}\rkla{\widetilde{\phi}_{s_\iota},\widetilde{\mu}_{s_\iota},\widetilde{W}_{s_\iota}}}\,,\\
\mathfrak{f}^{qo}_j:=&\,\prod_{i=1}^q\prod_{\iota=1}^of_{i\iota}\rkla{\Psi_{i\iota}\rkla{\rkla{\widetilde{\phi}_j}_{s_\iota},\rkla{\widetilde{\mu}_j^+}_{s_\iota},\rkla{\widetilde{\bs{\xi}}_j}_{s_\iota}}}
\end{align}
\end{subequations}
for times $0\leq s_1<\ldots<s_q\leq t_1<t_2$ to simplify the notation.
Hence, the limit process $\widetilde{M}^v$ is a martingale.
\end{proof}

In the next step, we identify the quadratic variation process $\skla{\!\!\skla{\widetilde{M}^v}\!\!}$ of $\widetilde{M}^v$, as well as the cross variation processes $\skla{\!\!\skla{\widetilde{M}^v,\widetilde{\beta}_k}\!\!}$.
Natural candidates for these processes are
\begin{subequations}\label{eq:variations}
\begin{align}
\skla{\!\!\skla{\widetilde{M}^v}\!\!}\trkla{t}:=&\,\int_0^t\sum_{k\in\mathds{Z}}\rkla{\iO \sigma\trkla{\widetilde{\phi}}\lambda_k\g{k} v\dx}^2\ds\,,\label{eq:def:quadvar}\\
\skla{\!\!\skla{\widetilde{M}^v,\widetilde{\beta}_k}\!\!}\trkla{t}:=&\int_0^t\iO\sigma\trkla{\widetilde{\phi}}\lambda_k\g{k}v\dx\ds\,.\label{eq:def:crossvar}
\end{align}
\end{subequations}
\begin{lemma}
The processes defined in \eqref{eq:variations} are the quadratic and cross variation processes of the martingale $\widetilde{M}^v$ defined \eqref{eq:defMv}.
\end{lemma}
\begin{proof}
We start by showing that $\rkla{\widetilde{M}^v}^2-\skla{\!\!\skla{\widetilde{M}^v}\!\!}$ is a martingale.
Using the abbreviations introduced in  \eqref{eq:tmp:abbreviations}, we need to show that
\begin{align}
\expectedt{\rkla{\rkla{\widetilde{M}^v\trkla{t_2}}^2-\rkla{\widetilde{M}^v\trkla{t_1}}^2-\skla{\!\!\skla{\widetilde{M}^v}\!\!}\trkla{t_2}+\skla{\!\!\skla{\widetilde{M}^v}\!\!}\trkla{t_1}}\mathfrak{f}^{qo}}=0\,.
\end{align}
We consider the following family $\rkla{\skla{\!\!\skla{\widetilde{M}^v_j}\!\!}}_{j\in\mathds{N}}$ of approximations defined via
\begin{align}
\skla{\!\!\skla{\widetilde{M}_j^v}\!\!}\trkla{t_j^m}:=\int_0^{t_j^m}\sum_{k\in\mathds{Z}\hj}\rkla{\iO\Ihj{\sigma\trkla{\widetilde{\phi}_j^-}\lambda_k\g{k}}v\dx}^2\ds
\end{align}
and establish the convergence of $\skla{\!\!\skla{\widetilde{M}_j^v}\!\!}\trkla{t_j^m}$ towards $\skla{\!\!\skla{\widetilde{M}^v}\!\!}\trkla{t}$ in $L^p\trkla{\widetilde{\Omega}}$ for $t_j^m\searrow t$.

\begin{multline}
\expectedt{\abs{\skla{\!\!\skla{\widetilde{M}^v}\!\!}\trkla{t}-\skla{\!\!\skla{\widetilde{M}_j^v}\!\!}\trkla{t_j^m}}^p}\\
\leq C\expectedt{\abs{\int_0^t\sum_{k\in\mathds{Z}}\rkla{\iO\sigma\trkla{\widetilde{\phi}}\lambda_k\g{k} v\dx}^2\ds-\int_0^t\sum_{k\in\mathds{Z}\hj}\rkla{\iO\sigma\trkla{\widetilde{\phi}_j^-}\lambda_k\g{k}v\dx}^2\ds}^p}\\
+C\expectedt{\abs{\int_t^{t_j^m}\sum_{k\in\mathds{Z}\hj}\rkla{\iO \sigma\trkla{\widetilde{\phi}_j^-}\lambda_k\g{k} v\dx}^2\ds}^p}\\
+C\expectedt{\abs{\int_0^{t_j^m}\sum_{k\in\mathds{Z}\hj}\ekla{\rkla{\iO\sigma\trkla{\widetilde{\phi}_j^-}\lambda_k\g{k}v\dx}^2-\rkla{\iO\Ihj{\sigma\trkla{\widetilde{\phi}_j^-}\lambda_k\g{k}}v\dx}^2}\ds}^p}\\
=:\mathfrak{R}_1+\mathfrak{R}_2+\mathfrak{R}_3\,.
\end{multline}
Combining \ref{item:color}, \ref{item:sigma}, and Lemma \ref{lem:convergence} shows that $\mathfrak{R}_1$ vanishes.
Assumptions \ref{item:color} and \ref{item:sigma} also provide $\mathfrak{R}_2\leq C\trkla{v}\tabs{t_j^m-t}^p\rightarrow 0$.
Hence, it remains to show $\mathfrak{R}_3\rightarrow0$.
Using the Lipschitz continuity of $\sigma$ (cf.~\ref{item:sigma}), standard error estimates for the interpolation operator (see e.g.~\cite{BrennerScott}), the error estimates collected in Lemma \ref{lem:interpolation}, and a standard inverse estimate, we compute
\begin{multline}\label{eq:tmp:convergenceinterpol}
\abs{\iO \sigma\trkla{\widetilde{\phi}_j^-}\lambda_k\g{k} v\dx-\iO\Ihj{\sigma\trkla{\widetilde{\phi}_j^-}\lambda_k\g{k}}v\dx}\\
\leq \abs{\iO\trkla{1-\Ihjop}\gkla{\sigma\trkla{\widetilde{\phi}_j^-}}\lambda_k\g{k} v\dx}+\abs{\iO\Ihj{\sigma\trkla{\widetilde{\phi}_j^-}}\lambda_k\trkla{1-\Ihjop}\gkla{\g{k}}v\dx}\\
+\abs{\iO\trkla{1-\Ihjop}\gkla{\Ihj{\sigma\trkla{\widetilde{\phi}_j^-}}\Ihj{\g{k}}}v\dx}\\
\leq Ch_j\norm{\nabla\widetilde{\phi}_j^-}_{L^2\trkla{\Om}}\tabs{\lambda_k}\norm{\g{k}}_{L^\infty\trkla{\Om}}\norm{v}_{L^2\trkla{\Om}} + C h_j\tabs{\lambda_k}\norm{\g{k}}_{W^{1,\infty}\trkla{\Om}}\norm{v}_{L^2\trkla{\Om}}\\
+Ch_j\tabs{\lambda_k}\norm{\g{k}}_{W^{1,\infty}\trkla{\Om}}\norm{v}_{L^2\trkla{\Om}}\,.
\end{multline}
Hence,
\begin{align}
\begin{split}
\mathfrak{R}_3\leq&\,C\widetilde{\mathds{E}}\left[\left|\int_0^{t_j^m}\sum_{k\in\mathds{Z}\hj}\abs{\iO\trkla{1+\Ihjop}\gkla{\sigma\trkla{\widetilde{\phi}_j^-}\lambda_k\g{k}}v\dx}\right.\right.\\
&\left.\left.\phantom{\sum_{k\in\mathds{Z}\hj}}\qquad\times h_j\tabs{\lambda_k}\norm{\g{k}}_{W^{1,\infty}\trkla{\Om}}\norm{v}_{L^2\trkla{\Om}}\rkla{1+\norm{\nabla\widetilde{\phi}_j^-}_{L^2\trkla{\Om}}} \right|^p\right]\\
\leq&\, C\trkla{v}\expectedt{\abs{h_j\sum_{k\in\mathds{Z}}\tabs{\lambda_k}^2\norm{\g{k}}_{W^{1,\infty}\trkla{\Om}}^2\int_0^{t_j^m}\rkla{1+\norm{\nabla\widetilde{\phi}_j^-}_{L^2\trkla{\Om}}}\ds}^p}\rightarrow 0\,.
\end{split}
\end{align}
Computations similar to the ones above also provide uniform bounds for higher moments of $\skla{\!\!\skla{\widetilde{M}^v_j}\!\!}$.
Having established the convergence in $L^p\trkla{\widetilde{\Omega}}$, we have $\widetilde{\Prob}$-almost sure convergence for a subsequence.
Hence, we can reuse the ideas from the proof of Lemma \ref{lem:martingale} and deduce for $0\leq t_j^n-t_2\leq\tau_j$ and $0\leq t_j^m-t_1\leq\tau_j$:
\begin{align}
\begin{split}
&\expectedt{\rkla{\rkla{\widetilde{M}^v\trkla{t_2}}^2-\rkla{\widetilde{M}^v\trkla{t_1}}^2-\skla{\!\!\skla{\widetilde{M}^v}\!\!}\trkla{t_2}+\skla{\!\!\skla{\widetilde{M}^v}\!\!}\trkla{t_1}}\mathfrak{f}^{qo}}\\
&=\lim_{j\rightarrow\infty}\expectedt{\rkla{\rkla{\widetilde{M}_j^v\trkla{t_j^n}}^2-\rkla{\widetilde{M}^v_j\trkla{t_j^m}}^2-\skla{\!\!\skla{\widetilde{M}^v_j}\!\!}\trkla{t_j^n}+\skla{\!\!\skla{\widetilde{M}^v_j}\!\!}\trkla{t_j^m}}\mathfrak{f}^{qo}_j}\\
&=\lim_{j\rightarrow\infty}\widetilde{\mathds{E}}\left[\left(\sum_{a=m+1}^n\rkla{\iO\Ihj{\Phi\hj\trkla{\widetilde{\phi}_j\trkla{t_j^{a-1}}}\sinctilde{a}}v\dx}^2\right.\right.\\
&\left.\left.\qquad\qquad-\sum_{a=m+1}^n\tau_j\sum_{k\in\mathds{Z}\hj}\rkla{\iO\Ihj{\sigma\trkla{\widetilde{\phi}_j\trkla{t_j^{a-1}}}\lambda_k\g{k}}v\dx}^2\right)\mathfrak{f}_j^{qo}\right]\\
&=0\,,
\end{split}
\end{align}
due to the mutual independence of the stochastic increments and $\expectedt{\tabs{\widetilde{\xi}_k^{n,\tau_j}}^2}=1$ (cf.~\ref{item:randomvariables}).
We shall now show that $\skla{\!\!\skla{\widetilde{M}^v,\widetilde{\beta}_k}\!\!}$ defined in \eqref{eq:def:crossvar} are indeed the cross variation processes by applying similar arguments.
Again, we start by defining suitable time discrete approximations for time points $t_j^m$ via
\begin{align}\label{eq:def:disc:crossvar}
\skla{\!\!\skla{\widetilde{M}_j^v,\sum_{a=1}^m\sqrt{\tau_j}\widetilde{\xi}_k^{a,\tau_j}}\!\!}:=\left\{\begin{array}{cl}
\int_0^{t_j^m}\iO\lambda_k\Ihj{\sigma\trkla{\widetilde{\phi}_j^-}\g{k}}v\dx\ds&\text{if~}k\in\mathds{Z}\hj\,,\\
0&\text{else}\,,
\end{array}\right.
\end{align}
and establish its convergence towards $\skla{\!\!\skla{\widetilde{M}^v,\widetilde{\beta}_k}\!\!}\trkla{t}$ for $j\rightarrow\infty$ and $t_j^m\searrow t$.
As $\lim_{j\rightarrow\infty}\mathds{Z}\hj=\mathds{Z}$, we can assume without loss of generality that $k\in\mathds{Z}\hj$.
Hence,
\begin{multline}
\expectedt{\abs{\int_0^t\iO\sigma\trkla{\widetilde{\phi}}\lambda_k\g{k} v\dx\ds-\int_0^{t_j^m}\iO\Ihj{\sigma\trkla{\widetilde{\phi}^-_j}\g{k}}v\dx\ds}^p}\\
\leq C\expectedt{\abs{\int_0^t\iO\rkla{\sigma\trkla{\widetilde{\phi}}-\sigma\trkla{\widetilde{\phi}_j^-}}\lambda_k\g{k}v\dx\ds}^p} \\
+C\expectedt{\abs{\int_t^{t_j^m}\iO\sigma\trkla{\widetilde{\phi}_j^-}\lambda_k\g{k}v\dx\ds}^p}\\
+ C\expectedt{\abs{\int_0^t\iO\trkla{1-\Ihjop}\gkla{\sigma\trkla{\widetilde{\phi}_j^-}\lambda_k\g{k}}v\dx\ds}^p}\\
=:\mathfrak{Q}_1+\mathfrak{Q}_2+\mathfrak{Q}_3\,.
\end{multline}
As before, $\mathfrak{Q}_1$ and $\mathfrak{Q}_2$ vanish due to \ref{item:color}, \ref{item:sigma}, and Lemma \ref{lem:convergence}.
The remaining term $\mathfrak{Q}_3$ vanishes due to \eqref{eq:tmp:convergenceinterpol}.\\
Hence, we can again deduce for $0\leq t_j^n-t_2\leq\tau_j$ and $0\leq t_j^m-t_1\leq\tau_j$ that
{\allowdisplaybreaks
\begin{multline}
\expectedt{\rkla{\widetilde{M}^v\trkla{t_2}\widetilde{\beta}_k\trkla{t_2}-\widetilde{M}^v\trkla{t_1}\widetilde{\beta}_k\trkla{t_1}-\skla{\!\!\skla{\widetilde{M}^v,\widetilde{\beta}_k}\!\!}\trkla{t_2}+\skla{\!\!\skla{\widetilde{M}^v,\widetilde{\beta}_k}\!\!}\trkla{t_1}}\mathfrak{f}^{qo}}\\
=\lim_{j\rightarrow\infty}\widetilde{\mathds{E}}\left[\left(\widetilde{M}^v_j\trkla{t_j^n}\sum_{a=1}^n\sqrt{\tau_j}\widetilde{\xi}_k^{a,\tau_j}-\widetilde{M}_j^v\trkla{t_j^m}\sum_{a=1}^m\sqrt{\tau_j}\widetilde{\xi}_j^{a,\tau_j}\right.\right.\\
\left.\left.-\int_{t_j^m}^{t_j^n}\iO\lambda_k\Ihj{\sigma\trkla{\widetilde{\phi}_j^-}\g{k}}v\dx\ds\right)\mathfrak{f}_j^{qo}\right]\\
=\lim_{j\rightarrow\infty}\expectedt{\rkla{\sum_{b=1}^n\iO\Ihj{\Phi\hj\trkla{\widetilde{\phi}_j\trkla{t_j^{b-1}}}\sinctilde{b}}v\dx\sum_{a=1}^n\sqrt{\tau_j}\widetilde{\xi}_k^{a,\tau_j}}\mathfrak{f}_j^{qo}}\\
-\lim_{j\rightarrow\infty}\expectedt{\rkla{\sum_{b=1}^m\iO\Ihj{\Phi\hj\trkla{\widetilde{\phi}_j\trkla{t_j^{b-1}}}\sinctilde{b}}v\dx\sum_{a=1}^m\sqrt{\tau_j}\widetilde{\xi}_k^{a,\tau_j}}\mathfrak{f}_j^{qo}}\\
-\lim_{j\rightarrow\infty}\expectedt{\rkla{\int_{t_j^m}^{t_j^n}\iO\lambda_k\Ihj{\sigma\trkla{\widetilde{\phi}_j^-}\g{k}}v\dx\ds}\mathfrak{f}_j^{qo}}\\
=\lim_{j\rightarrow\infty}\expectedt{\rkla{\sum_{b=m+1}^n\sum_{a=1}^n\iO\Ihj{\Phi\hj\trkla{\widetilde{\phi}_j\trkla{t_j^{b-1}}}\sinctilde{b}\sqrt{\tau_j}\widetilde{\xi}_k^{a,\tau_j}}v\dx}\mathfrak{f}_j^{qo}}\\
+\lim_{j\rightarrow\infty}\expectedt{\rkla{\sum_{b=1}^m\sum_{a=m+1}^n\iO\Ihj{\Phi\hj\trkla{\widetilde{\phi}_j\trkla{t_j^{b-1}}}\sinctilde{b}\sqrt{\tau_j} \widetilde{\xi}_k^{a,\tau_j}}v\dx}\mathfrak{f}_j^{qo}}\\
-\lim_{j\rightarrow\infty}\expectedt{\rkla{\int_{t_j^m}^{t_j^n}\iO\lambda_k\Ihj{\sigma\trkla{\widetilde{\phi}_j^-}\g{k}}v\dx\ds}\mathfrak{f}_j^{qo}}\\
=\lim_{j\rightarrow\infty}\expectedt{\rkla{\sum_{b=m+1}^n\iO\Ihj{\Phi\hj\trkla{\widetilde{\phi}_j\trkla{t_j^{b-1}}}\lambda_k\g{k}\widetilde{\xi}_k^{b,\tau_j}\tau_j\widetilde{\xi}_k^{b,\tau_j}}v\dx}\mathfrak{f}_j^{qo}}\\
-\lim_{j\rightarrow\infty}\expectedt{\rkla{\int_{t_j^m}^{t_j^n}\iO\lambda_k\Ihj{\sigma\trkla{\widetilde{\phi}_j^-}\g{k}}v\dx\ds}\mathfrak{f}_j^{qo}}=0\,.
\end{multline}}
Here, we used \eqref{eq:defMjv} and \ref{item:randomvariables}.
\end{proof}

Having obtained explicit expressions for the quadratic and cross variations of the martingale $\widetilde{M}^v$, we can follow the approach introduced in \cite{BrzezniakOndrejat, Ondrejat2010, HofmanovaSeidler} (see also \cite{HofmanovaRoegerRenesse,BreitFeireislHofmanova}) to show that $\widetilde{M}^v$ can be written as an It\^o-integral with the Wiener process $\widetilde{W}$.
\begin{lemma}\label{lem:itointegral}
Let the assumptions of Theorem \ref{thm:jakubowski} and Lemma \ref{lem:convergence} hold true. 
Then, the martingale $\widetilde{M}^v$ can be written as an It\^o-integral with the $\mathcal{Q}$-Wiener process $\widetilde{W}$, i.e.
\begin{align}
\widetilde{M}^v\trkla{t}=\int_0^t\sum_{k\in\mathds{Z}}\iO\sigma\trkla{\widetilde{\phi}}\lambda_k\g{k} v\dx\,\mathrm{d}\widetilde{\beta}_k\,.
\end{align}
\end{lemma}
\begin{proof}
As a martingale with vanishing quadratic variation is almost surely constant, it suffices to show that the quadratic variation of
\begin{align*}
\widetilde{M}^v\trkla{t}-\int_0^t\sum_{k\in\mathds{Z}}\iO\sigma\trkla{\widetilde{\phi}}\lambda_k\g{k} v\dx\,\mathrm{d}\widetilde{\beta}_k
\end{align*}
vanishes to establish the desired result.
We start by computing
\begin{multline}\label{eq:tmp:binomquadvar}
\skla{\!\!\skla{\widetilde{M}^v-\int_0^{\trkla{\cdot}}\sum_{k\in\mathds{Z}}\iO\sigma\trkla{\widetilde{\phi}}\lambda_k\g{k} v\dx\,\mathrm{d}\widetilde{\beta}_k}\!\!}\trkla{t}\\
=\skla{\!\!\skla{\widetilde{M}^v}\!\!}\trkla{t}+\skla{\!\!\skla{\int_0^{\trkla{\cdot}}\sum_{k\in\mathds{Z}}\iO\sigma\trkla{\widetilde{\phi}}\lambda_k\g{k} v\dx\,\mathrm{d}\widetilde{\beta}_k}\!\!}\trkla{t}\\
-2\skla{\!\!\skla{\widetilde{M}^v,\int_0^{\trkla{\cdot}}\sum_{k\in\mathds{Z}}\iO\sigma\trkla{\widetilde{\phi}}\lambda_k\g{k} v\dx\,\mathrm{d}\widetilde{\beta}_k}\!\!}\trkla{t}\,.
\end{multline}
Applying the cross-variation formula (see e.g.~\cite{KaratzasShreve2004}), we find for the last term on the right-hand side of \eqref{eq:tmp:binomquadvar}
\begin{align}
\skla{\!\!\skla{\widetilde{M}^v,\int_0^{\trkla{\cdot}}\sum_{k\in\mathds{Z}}\iO\sigma\trkla{\widetilde{\phi}}\lambda_k\g{k} v\dx\,\mathrm{d}\widetilde{\beta}_k}\!\!}\trkla{t}=\int_0^t\sum_{k\in\mathds{Z}}\iO\sigma\trkla{\widetilde{\phi}}\lambda_k\g{k} v\dx\,\mathrm{d}\skla{\!\!\skla{\widetilde{M}^v,\widetilde{\beta}_k}\!\!}\trkla{s}\,.
\end{align}
From \eqref{eq:def:crossvar} and \ref{item:sigma}, we obtain that the process $\tekla{0,T}\ni s\mapsto\skla{\!\!\skla{\widetilde{M}^v,\widetilde{\beta}_k}\!\!}\trkla{s}$ is absolutely continuous.
Hence, we have
\begin{align}
\mathrm{d}\skla{\!\!\skla{\widetilde{M}^v,\widetilde{\beta}_k}\!\!}=\lambda_k\iO\sigma\trkla{\widetilde{\phi}}\g{k}v\dx\ds
\end{align}
and consequently
\begin{align}
\skla{\!\!\skla{\widetilde{M}^v,\int_0^{\trkla{\cdot}}\sum_{k\in\mathds{Z}}\iO\sigma\trkla{\widetilde{\phi}}\lambda_k\g{k} v\dx\,\mathrm{d}\widetilde{\beta}_k}\!\!}\trkla{t}=\int_0^t\sum_{k\in\mathds{Z}}\tabs{\lambda_k}^2\rkla{\iO\sigma\trkla{\widetilde{\phi}}\g{k}v\dx}^2\ds\,.
\end{align}
Together with \eqref{eq:def:quadvar} and
\begin{align}
\skla{\!\!\skla{\int_0^{\trkla{\cdot}}\sum_{k\in\mathds{Z}}\iO\sigma\trkla{\widetilde{\phi}}\lambda_k\g{k} v\dx\,\mathrm{d}\widetilde{\beta}_k}\!\!}\trkla{t}=\int_0^t\sum_{k\in\mathds{Z}}\tabs{\lambda_k}^2\rkla{\iO\sigma\trkla{\widetilde{\phi}}\g{k} v\dx}^2\ds\,,
\end{align}
this concludes the proof.
\end{proof}

This provides the existence of martingale solutions and the convergence of discrete solutions of the linear stochastic SAV scheme.

\section{Pathwise uniqueness}\label{sec:uniqueness}
In this section, we complete the proof of Theorem \ref{thm:maintheorem} by establishing the pathwise uniqueness of the obtained solutions.
As a side effect, the pathwise uniqueness of the obtained solutions also allows us to extend the convergence results derived prior for subsequences to the complete sequence.\\
To show the pathwise uniqueness, we follow the approach in \cite{Liu2013}:
For two solutions $\widetilde{\phi}_1$ and $\widetilde{\phi}_2$ to \eqref{eq:solution} on the same stochastic basis $\trkla{\widetilde{\Omega},\widetilde{\mathcal{A}},\widetilde{\mathcal{F}},\widetilde{\Prob}}$ with the same $\mathcal{Q}$-Wiener process $\widetilde{W}$ sharing the same initial data $\phi_0$, we have the pathwise estimate
\begin{align}
\begin{split}
&\rkla{\Delta\widetilde{\phi}_1-\Delta\widetilde{\phi}_2,\widetilde{\phi}_1-\widetilde{\phi}_2}_{L^2\trkla{\Om}}-\rkla{F^\prime\rkla{\widetilde{\phi}_1}-F^\prime\rkla{\widetilde{\phi}_2},\widetilde{\phi}_1-\widetilde{\phi_2}}_{L^2\trkla{\Om}}\\
&\quad\leq -\norm{\nabla\widetilde{\phi}_1-\nabla\widetilde{\phi}_2}_{L^2\trkla{\Om}}^2 + C\norm{1+\abs{\widetilde{\phi}_1}^2+\abs{\widetilde{\phi}_2}^2}_{L^3\trkla{\Om}}\norm{\widetilde{\phi}_1-\widetilde{\phi}_2}_{H^1\trkla{\Om}}\norm{\widetilde{\phi}_1-\widetilde{\phi}_2}_{L^2\trkla{\Om}}\\
&\quad\leq C\rkla{1+\norm{\widetilde{\phi}_1}_{L^\infty\trkla{0,T;H^1\trkla{\Om}}}^2+\norm{\widetilde{\phi}_2}_{L^\infty\trkla{0,T;H^1\trkla{\Om}}}^2}\norm{\widetilde{\phi}_1-\widetilde{\phi}_2}_{L^2\trkla{\Om}}^2\\
&\qquad+ C\rkla{1+\norm{\widetilde{\phi}_1}_{L^\infty\trkla{0,T;H^1\trkla{\Om}}}^2+\norm{\widetilde{\phi}_2}_{L^\infty\trkla{0,T;H^1\trkla{\Om}}}^2}^2\norm{\widetilde{\phi}_1-\widetilde{\phi}_2}_{L^2\trkla{\Om}}^2\\
&\quad=:\rkla{C+\rho\rkla{\widetilde{\phi}_1,\widetilde{\phi_2}}}\norm{\widetilde{\phi}_1-\widetilde{\phi}_2}_{L^2\trkla{\Om}}^2\,,
\end{split}
\end{align}
due to the growth condition stated in \ref{item:potentialF} and Young's inequality.
Recalling \eqref{eq:Phi} and \ref{item:sigma}, we note that \eqref{eq:solution} satisfies a local monotonicity condition (cf.~(H2) in \cite{Liu2013}).
Hence, we can conclude by It\^o's formula
\begin{multline}
e^{-\int_0^t\trkla{C+\rho\trkla{\widetilde{\phi}_1,\widetilde{\phi}_2}}\ds} \norm{\widetilde{\phi}_1\trkla{t}-\widetilde{\phi}_2\trkla{s}}_{L^2\trkla{\Om}}^2\\
\leq 2\int_0^t e^{-\int_0^s\trkla{C+\rho\trkla{\widetilde{\phi}_1,\widetilde{\phi}_2}}\dr}\rkla{\widetilde{\phi}_1\trkla{s}-\widetilde{\phi}_2\trkla{s}, \Phi\rkla{\widetilde{\phi}_1\trkla{s}}\mathrm{d}\widetilde{W} - \Phi\rkla{\widetilde{\phi}_2\trkla{s}}\mathrm{d}\widetilde{W}}_{L^2\trkla{\Om}}\,, 
\end{multline}
which in particular provides
\begin{align}
\widetilde{\mathds{E}}\ekla{e^{-\int_0^t\trkla{C+\rho\trkla{\widetilde{\phi}_1,\widetilde{\phi}_2}}\ds} \norm{\widetilde{\phi}_1\trkla{t}-\widetilde{\phi}_2\trkla{s}}_{L^2\trkla{\Om}}^2}\leq 0
\end{align}
for $t\in\tekla{0,T}$.
As $\widetilde{\phi}_1$ and $\widetilde{\phi}_2$ are martingale solutions, they have the regularity stated in Theorem \ref{thm:maintheorem}.
In particular, we have 
\begin{align}
\norm{\widetilde{\phi}_1}_{L^\infty\trkla{0,T;H^1\trkla{\Om}}}+\norm{\widetilde{\phi}_2}_{L^\infty\trkla{0,T;H^1\trkla{\Om}}} <\infty
\end{align}
$\widetilde{\Prob}$-almost surely.
As a direct consequence, we obtain $\int_0^t \trkla{C+\rho\trkla{\widetilde{\phi}_1,\widetilde{\phi}_2}}\ds<\infty$ $\widetilde{\Prob}$-almost surely and hence $\widetilde{\phi}_1\trkla{t}=\widetilde{\phi}_2\trkla{t}$ $\widetilde{\Prob}$-almost surely for $t\in\tekla{0,T}$.
Therefore, the pathwise uniqueness follows from the path continuity of $\widetilde{\phi}_1$ and $\widetilde{\phi}_2$ in $L^2\trkla{\Om}$.

\section{Proof of Theorem \ref{thm:strongconvergence}: Convergence towards strong solutions}\label{sec:strong}
In this section, we shall assume that the stochastic increments introduced in assumption \ref{item:randomvariables} are given as increments of a $\mathcal{Q}$-Wiener process satisfying \ref{item:stochastic} and \ref{item:Wcolor}, and prove Theorem \ref{thm:strongconvergence}.
As increments of this $\mathcal{Q}$-Wiener process satisfy assumptions \ref{item:filtration}-\ref{item:color}, the results established in the previous sections remain valid.
In particular, we still have the pathwise uniqueness of martingale solutions.
Hence, we can prove $\Prob$-almost sure convergence using a generalization of the Gyöngy--Krylov characterization of convergence in probability (cf.~\cite{GyongyKrylov1996}) to the setting of quasi-Polish spaces that was established in \cite{BreitFeireislHofmanova}.\\
As shown before, there exists a sequence of stochastic processes defined on $\trkla{\Omega,\mathcal{A},\mathcal{F},\Prob}$ satisfying
\begin{subequations}
\begin{multline}
\iO\Ih{\rkla{\phi\h\tl\trkla{t}-\phi\h\tm}\psi\h}\dx+\trkla{t-t\no}\iO\Ih{\mu\h\tp\psi\h}\dx\\
=\frac{t-t\no}\tau\iO\Ih{\Phi\h\trkla{\phi\h\tm}\rkla{W\trkla{t\nn}-W\trkla{t\no}}}\dx\,
\end{multline}
for the given $\mathcal{Q}$-Wiener process $W$ with $\mu\h\tp$ depending on $\phi\h\tl$, $r\h\tl$, and $\Delta\h\phi\h\tp$ via
\begin{align}
\begin{split}
\mu\h\tp=&\,-\Delta\h\phi\h\tp+\frac{r\h\tp}{\sqrt{E\h\trkla{\phi\h\tm}}}\Ih{F^\prime\trkla{\phi\h\tm}}\\
&\,-\frac{r\h\tp}{4\ekla{E\h\trkla{\phi\h\tm}}^{3/2}}\iO\Ih{F^\prime\trkla{\phi\h\tm}\Phi\h\trkla{\phi\h\tm}\rkla{W\trkla{t\nn}-W\trkla{t\no}}}\dx\,\Ih{F^\prime\trkla{\phi\h\tm}}\\
&\,+\frac{r\h\tp}{2\sqrt{E\h\trkla{\phi\h\tm}}}\Ih{F^{\prime\prime}\trkla{\phi\h\tm}\Phi\h\trkla{\phi\h\tm}\rkla{W\trkla{t\nn}-W\trkla{t\no}}}
\end{split}
\end{align}
\end{subequations}
according to \eqref{eq:defmu}.
From this sequence of stochastic processes, we excerpt an arbitrary pair of subsequences $\trkla{\phi\hj^{\tau_j},r\hj^{\tau_j},\Delta\hj\phi\hj^{\tau_j,+},\mu\hj^{\tau_j,+}}_{j\in\mathds{N}}$ and $\trkla{\phi_{h_i}^{\tau_i},r_{h_i}^{\tau_i},\Delta_{h_i}\phi_{h_i}^{\tau_i,+},\mu_{h_i}^{\tau_i,+}}_{i\in\mathds{N}}$.
For this pair of subsequences, we consider their joint laws $\rkla{\nu_{i,j}}_{i,j\in\mathds{N}}$ on $\mathcal{Y}\times\mathcal{Y}$ with
\begin{align}
\mathcal{Y}:=C\trkla{\tekla{0,T};L^s\trkla{\Om}}\times L^2\trkla{0,T}_{\weaktop}\times L^2\trkla{0,T;L^2\trkla{\Om}}_{\weaktop}\times L^2\trkla{0,T;L^2\trkla{\Om}}_{\weaktop}\,.
\end{align}
We want to show that for the sequence of joint laws there exists a further subsequence which converges weakly to a probability measure $\nu$ such that
\begin{align}\label{eq:diagonallaw}
\nu\rkla{\trkla{x,y}\in\mathcal{Y}\times\mathcal{Y}\,:\,x=y}=1\,.
\end{align}
Then, according to \cite[Proposition A.4]{BreitFeireislHofmanova}, there exists a subsequence of $\trkla{\phi\h^{\tau},r\h^{\tau},\Delta\h\phi\h^{\tau,+},\mu\h^{\tau,+}}_{h,\tau}$ which converges $\Prob$-almost surely in the topology of $\mathcal{Y}$.\\
To establish \eqref{eq:diagonallaw}, we define the extended path space
\begin{align}
\widehat{\mathcal{X}}:=\mathcal{Y}\times\mathcal{Y}\times C\trkla{\tekla{0,T};L^2\trkla{\Om}}\,,
\end{align}
and consider the sequence 
\begin{align}
\rkla{z_{ij}}_{i,j\in\mathds{N}}:=\rkla{\phi\hj^{\tau_j},r\hj^{\tau_j},\Delta\hj\phi\hj^{\tau_j,+},\mu\hj^{\tau_j,+},\phi_{h_i}^{\tau_i},r_{h_i}^{\tau_i},\Delta_{h_i}\phi_{h_i}^{\tau_i,+},\mu_{h_i}^{\tau_i,+},W}_{i,j\in\mathds{N}}\,.
\end{align}
A slight adaption of the arguments in Section \ref{sec:compactness} shows that the joint laws of this sequence are tight on $\widehat{\mathcal{X}}$.
Hence, we can repeat the arguments of Theorem \ref{thm:jakubowski} to show that for a subsequence $\trkla{z_{i_\alpha j_\alpha}}_{\alpha\in\mathds{N}}$, there exists a sequence of random variables
\begin{align}
\rkla{\widetilde{z}_\alpha}_{\alpha\in\mathds{N}}:=\rkla{\widetilde{\phi}_{j_\alpha},\widetilde{r}_{j_\alpha},\Delta_{h_{j_\alpha}}\widetilde{\phi}_{j_\alpha}^+,\widetilde{\mu}_{j_\alpha}^+,\widetilde{\phi}_{i_\alpha},\widetilde{r}_{i_\alpha},\Delta_{h_{i_\alpha}}\widetilde{\phi}_{i_\alpha}^+,\widetilde{\mu}_{i_\alpha}^+,\widetilde{W}_\alpha}_{\alpha\in\mathds{N}}
\end{align}
defined on different probability space $\trkla{\widetilde{\Omega},\widetilde{\mathcal{A}},\widetilde{\Prob}}$ having the same law as $\trkla{z_{i_\alpha j_\alpha}}_{\alpha\in\mathds{N}}$ and converging towards random variables 
\begin{align}
\rkla{\widetilde{\phi}, \widetilde{r}, \widetilde{L}, \widetilde{\mu},\widehat{\phi},\widehat{r},\widehat{L},\widehat{\mu},\widetilde{W}}\,.
\end{align}
Arguments similar to the ones used in Lemma \ref{lem:convergence} yield $\widetilde{r}=\sqrt{\iO F\trkla{\widetilde{\phi}}\dx}$, $\widetilde{L}=\Delta\widetilde{\phi}$, $\widetilde{\mu}=-\Delta\widetilde{\phi}+F^\prime\trkla{\widetilde{\phi}}$, $\widehat{r}=\sqrt{\iO F\trkla{\widehat{\phi}}\dx}$, $\widehat{L}=\Delta\widehat{\phi}$, and $\widehat{\mu}=-\Delta\widehat{\phi}+F^\prime\trkla{\widehat{\phi}}$.
Repeating the arguments of Section \ref{sec:limit}, shows that $\trkla{\widetilde{\phi},\widetilde{W}}$ and $\trkla{\widehat{\phi},\widetilde{W}}$ are martingale solutions to the stochastic Allen--Cahn equation.
As both martingale solutions are subjected to the same initial conditions and share the same Wiener process $\widetilde{W}$ and a common filtration generated by $\widetilde{\phi}$, $\widehat{\phi}$, and $\widetilde{W}$, the pathwise uniqueness established in Section \ref{sec:uniqueness} provides $\widetilde{\phi}=\widehat{\phi}$ $\widetilde{\Prob}$-almost surely.
Hence, we also have $\widetilde{r}=\widehat{r}$, $\widetilde{L}=\widehat{L}$, and $\widetilde{\mu}=\widehat{\mu}$, which by equality of laws provides \eqref{eq:diagonallaw}.\\
Hence, we can deduce $\Prob$-almost sure convergence for a subsequence of $\trkla{\phi\h^{\tau},r\h^{\tau},\Delta\h\phi\h^{\tau,+},\mu\h^{\tau,+}}_{h,\tau}$ in the topology of $\mathcal{Y}$, i.e.~we have a convergence result similar to Theorem \ref{thm:jakubowski} without switching to a new probability space.
Repeating the arguments of the previous sections concludes the proof of Theorem \ref{thm:strongconvergence}.

\section{Numerical simulations}\label{sec:numerics}
In this section, we demonstrate the practicality of the proposed augmented SAV scheme \eqref{eq:discscheme} and the importance of the additional augmentation terms by comparing it to a standard SAV discretization and a scheme based on the discretization studied in \cite{MajeeProhl2018} for polynomial double-well potentials.
Hence, we set $F\trkla{\phi}:=\tfrac14\trkla{\phi^2-1}^2+10^{-5}$ and use the following discrete scheme for comparison:
\begin{multline}\label{eq:majeeprohl}
\iO\Ih{\trkla{\phi\h\nn-\phi\h\no}\psi\h}\dx+\tau\iO\nabla\phi\h\nn\cdot\nabla\psi\h\dx+\tau\iO\Ih{\tfrac12\trkla{\tabs{\phi\h\nn}^2-1}\trkla{\phi\h\nn+\phi\h\no}}\dx\\
=\iO\Ih{\Phi\h\trkla{\phi\h\no}\sinc{n}\psi\h}\dx\,.
\end{multline} 
To deal with the nonlinearity in \eqref{eq:majeeprohl}, Newton's method was employed.
For this time discretization convergence of order $1/2-\delta$ for arbitrarily small $\delta>0$ was established in \cite{MajeeProhl2018}.
Furthermore, we will compare our results to computations based on the standard SAV discretization
\begin{subequations}\label{eq:SAV}
\begin{multline}
\iO\Ih{\trkla{\phi\h\nn-\phi\h\no}\psi\h}\dx+\tau\iO\nabla\phi\h\nn\cdot\nabla\psi\h\dx+\tau\frac{r\h\nn}{\sqrt{E\h\trkla{\phi\h\no}}}\iO\Ih{F^\prime\trkla{\phi\h\no}\psi\h}\dx\\
=\iO\Ih{\Phi\h\trkla{\phi\h\no}\sinc{n}\psi\h}\dx\,,
\end{multline}
\begin{align}
r\h\nn-r\h\no=\frac{1}{2\sqrt{E\h\trkla{\phi\h\no}}}\iO\Ih{F^\prime\trkla{\phi\h\no}\trkla{\phi\h\nn-\phi\h\no}}\dx\,,
\end{align}
\end{subequations}
which is \eqref{eq:discscheme} without the additional augmentation terms.
This scheme was already used for numerical simulations in \cite{QiZhangXu2023}.
Yet, the convergence of the scheme has not been investigated analytically.\\
All schemes were implemented in the \texttt{C++} framework EconDrop (cf.~\cite{Campillo2012, Grun2013c, GrunGuillenMetzger2016, Metzger_2018b, Metzger2021a, Metzger2023}).
In order to obtain the pseudo-random numbers for the approximation of the $\mathcal{Q}$-Wiener process, version 4.24 of the TRNG library was used (for details we refer to \cite{BaukeMertens2007}).\\
In the following we shall study the evolution of an elliptical droplet on $\Om=\trkla{-2,2}^2$ subjected to periodic boundary conditions.
Initially, the elliptical droplet is origin centered with diameter 1.5 in $x$-direction and diameter 1 in $y$-direction (cf.~Fig.~\ref{fig:initial}).
\begin{figure}
\begin{center}
\includegraphics[width=0.3\textwidth]{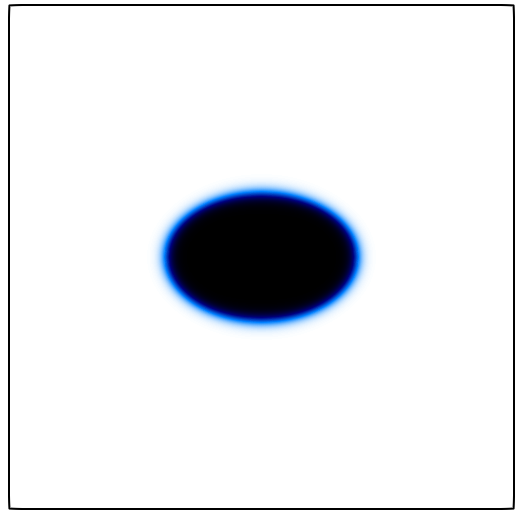}
\end{center}
\caption{Elliptical droplet used as initial condition for the numerical simulations.}
\label{fig:initial}
\end{figure}
In our study, we are interested in the effects of the size of time increments, the distributions of the pseudo-random numbers, and the strength of the stochastic source term.
In order to allow for comparability of the results, we leave the remaining parameters untouched.
In particular, we will use the following parameters:
We use the shifted polynomial double-well potential $F\trkla{\phi}:=\tfrac14\trkla{\phi^2-1}^2+10^{-5}$ and the parameter $\varepsilon$ which is related to the width of the interface is chosen as $\varepsilon=0.04$.
The spatial domain $\Om=\trkla{-2,2}^2$ is discretized using right-angled isosceles triangles with diameter $h=2^{-6}=0.015625$ which corresponds to $\dim\, \Uh=131,072$.
The approximations of the $\mathcal{Q}$-Wiener process are computed based on the smallest time increment $\tau_{\text{min}}$ considered in each subsection, i.e.~for $t^m=m\tau_{\text{min}}$ we use
\begin{align}\label{eq:defincremet}
\bs{\xi}\h^{m,\tau}\trkla{x,y}=\sum_{n=1}^m\sqrt{\tau_{\text{min}}}\sum_{k=-K}^{K}\sum_{l=-L}^{L} \lambda\trkla{k}\lambda\trkla{l} g_k^x\trkla{x} g_l^y\trkla{y} \xi_{k,l}^{n,\tau_{\text{min}}}\,.
\end{align}
Here, $g_k^x$ and $g_l^y$ are the $\sin$- and $\cos$-functions defined in Remark \ref{rem:eigenfunctions}, $\lambda\trkla{k}:=k^{-2}$ for $k\neq0$ and $\lambda\trkla{0}:=1$, and $\xi_{k,l}^{n,\tau_{\text{min}}}$ are pseudo-random numbers computed using the \texttt{xoshiro256plus} generator from the TRNG library.
Depending on the test case, these pseudo-random numbers either mimic independently, identically $\mathcal{N}\trkla{0,1}$-distributed random variables or independently, identically distributed Rademacher random variables (i.e.~$\Prob\trkla{\xi_{k,l}^{n,\tau_{\text{min}}}=1}=\Prob\trkla{\xi_{k,l}^{n,\tau_{\text{min}}}=-1}=0.5$).
$\Phi\h$ is defined as described in \eqref{eq:defPhih} with an indicator function $\sigma\trkla{\phi}:=\alpha\max\tgkla{1-\phi^2,0}$ of the interface.
Here, the parameter $\alpha$ describes the strength of the source term.
In each simulation, we consider 350 independent sample paths.
\subsection{Moderate noise}\label{subsec:moderate}
In our first batch of numerical experiments, we shall use a stochastic source term of moderate strength ($\alpha=2$) based on Rademacher distributed random variables and investigate the convergence behavior of solutions computed using the different discrete schemes.
As these discretizations differ only in the temporal direction, we will focus on the convergence with respect to time.
We shall consider  time-increments $\tau\in\tgkla{5\cdot 10^{-4}, 1\cdot10^{-3}, 2\cdot10^{-3}, 4\cdot10^{-3}, 1\cdot10^{-2}, 2\cdot10^{-2}}$ and compute the discrete approximations of the $\mathcal{Q}$-Wiener process using \eqref{eq:defincremet} with $\tau_\text{min}=5\cdot10^{-4}$ and $K=L=10$. \\
The evolution of the droplet over time is visualized in Fig.~\ref{fig:simpleEvolution}.
Here, the left column shows the evolution of the expected value of the phase-field parameter, while the middle and right column depict the evolution of two individual paths.
As a reference, the zero level set of a deterministic droplet is shown as an orange line.
For a better visual comparison of the different numerical schemes, we shall later discuss line plots of the phase-field profiles (cf.~Figs.~\ref{fig:linePlotExpected}, \ref{fig:linePlotPath14}, and \ref{fig:linePlotPath16}).
The area depicted in these line plots is indicated in Fig.~\ref{fig:simpleEvolution} by a red line.\\
While the expected value of the phase-field parameter remains origin centered, the droplet's barycenter can shift for individual paths.
As the noise terms can act as sources or as sinks, there exist paths where the droplet is growing and paths where the droplet is shrinking faster than anticipated by the deterministic evolution.
Yet, the droplet maintains a round shape with little deviation in the mean curvature.
As the deterministic Allen--Cahn equation is an approximation of the mean curvature flow, droplets will always try to evolve towards convex shapes with almost constant mean curvature.
The stochastic noise terms in this example are not chosen large enough to counteract this effect.\\
To compare the different numerical methods, we start by analyzing their numerical convergence properties.
Using the solutions obtained with the respective methods for $\tau_{\text{min}}=5\cdot 10^{-4}$ as reference solutions, we compute the experimental order of convergence for \eqref{eq:majeeprohl} and \eqref{eq:discscheme} individually.
The results are collected in Tab.~\ref{tab:moderateEOCrandomwalk}.
They indicate that the errors in the computations based on the augmented SAV-scheme \eqref{eq:discscheme} are slightly larger than the ones occurring in computations based on \eqref{eq:majeeprohl}.
The experimental order of convergence in these schemes is roughly the same for small time-increments. 
The order of convergence for the augmented SAV-scheme deteriorates slightly, however, if $\tau$ becomes too large.
This behavior can be expected as SAV-schemes are basically an additional approximation of the implicitly treated nonlinear terms in \eqref{eq:majeeprohl}.\\
The advantage of these approximation, on the other hand, lies in the linearization of the scheme, i.e.~SAV schemes reduce the computational costs to solving one linear system per time step.
The computational costs per path for different time increments are collected in the middle column of Tab.~\ref{tab:moderateComputationalCostsrandomwalk}.
These values indicate that using the augmented SAV scheme instead of \eqref{eq:majeeprohl}, where the nonlinearity is treated using Newton's method, allows for a speed-up by at least factor 2 in most cases.\\
Next, we investigate the importance of the augmentation terms, which were added to the standard SAV scheme to ensure its convergence.
Hence, we use the already extensively studied scheme \eqref{eq:majeeprohl} to compute a reference solution for each $\tau$ and compare the results obtained via the augmented SAV scheme \eqref{eq:discscheme} and the standard SAV scheme \eqref{eq:SAV} to it.
The result for this comparison at $t=4$ can be found in Tab.~\ref{tab:moderateComparisonNewtonSAVrandomwalk}.
These values clearly show that for vanishing $\tau$, \eqref{eq:majeeprohl} and the proposed augmented SAV scheme \eqref{eq:discscheme} will provide the same solutions.
The standard SAV scheme without the additional term $\Xi\h\nn$, however, does not provide the same results as \eqref{eq:majeeprohl} (despite the pathwise uniqueness of the solution to the stochastic Allen--Cahn equation).\\
To illustrate  these differences, we plotted the phase-field profiles along the red line shown in Fig.~\ref{fig:simpleEvolution}.
Here, solutions obtained using \eqref{eq:majeeprohl} are indicated by solid lines, solutions computed using the augmented SAV scheme \eqref{eq:discscheme} are indicated by dashed lines, and solutions based on \eqref{eq:SAV} are plotted using dotted lines.
In Fig.~\ref{fig:linePlotExpected}, the profile of the expected value of the phase-field parameter (cf.~left column of Fig.~\ref{fig:simpleEvolution}) at $t=2$ and $t=4$ is shown.
When comparing the different numerical schemes, one notices that the standard SAV scheme \eqref{eq:SAV} suggests a more centered and higher profile than \eqref{eq:majeeprohl}, while the proposed augmented SAV scheme \eqref{eq:discscheme} suggests a wider, but lower profile.
This can indicate that the augmented SAV scheme predicts a faster shrinking of the droplet than the other schemes.
As Fig.~\ref{fig:linePlotExpected} depicts the profile of the expected value, these differences can also indicate that there is more variation in the position of the droplet.\\
Profiles for individual paths can be found in Figs.~\ref{fig:linePlotPath14} \& \ref{fig:linePlotPath16}.
Fig.~\ref{fig:linePlotPath14} shows the phase-field profiles at $t=4$ for the path depicted in the middle column of Fig.~\ref{fig:simpleEvolution}.
These profile lines indicate that the standard SAV scheme \eqref{eq:SAV} provides a steeper interface profile than the implicit scheme \eqref{eq:majeeprohl}, while the augmented SAV scheme tends to a larger interfacial area. 
As the stochastic source term is located on the interface, this results in the solutions to the augmented SAV scheme being more susceptible to the source terms, which influence the growth and position of the droplet.
Fig.~\ref{fig:linePlotPath16} shows the droplet profiles of the path depicted in the right column of Fig.~\ref{fig:simpleEvolution} at $t=3$ and $t=4$.
Again, we note the interface profile suggested by the SAV scheme \eqref{eq:SAV} is steeper than the one predicted by the implicit scheme \eqref{eq:majeeprohl}, while the augmented SAV scheme suggests a flatter profile resulting in a faster shrinking of the droplet.
However, we want to emphasize that all three figures also clearly indicate that the differences between the augmented SAV scheme and \eqref{eq:majeeprohl} vanish for smaller time increments, while the differences between \eqref{eq:SAV} and \eqref{eq:majeeprohl} persist.
This verifies that the additional terms added in the augmented SAV scheme are indeed necessary -- not only for analytical results, but also for practical computations.\\

We conclude this section by repeating the simulations discussed above using $\mathcal{N}\trkla{0,1}$-distributed random variables for the increments in \eqref{eq:defincremet}.
The experimental order of convergence for \eqref{eq:majeeprohl} and the augmented SAV scheme \eqref{eq:discscheme} can be found in Tab.~\ref{tab:moderateEOCnormal}.
The computational costs per path are collected in Tab.~\ref{tab:moderateComputationalCostsrandomwalk}.
As there is no strict upper bound for $\mathcal{N}\trkla{0,1}$-distributed random variables, the increments in \eqref{eq:defincremet} can be larger, which explains the slightly higher computational costs and the slightly worse experimental order of convergence.
As shown in Tab.~\ref{tab:moderateComparisonNewtonSAVnormal}, the augmented SAV scheme \eqref{eq:discscheme} converges towards the same limit as the implicit scheme \eqref{eq:majeeprohl}, while the standard SAV scheme \eqref{eq:SAV} does not.

\begin{figure}
\begin{center}
\newcommand{\scale}{0.275}
\subfloat[][$t=0.5$]{
\includegraphics[width=\scale\textwidth]{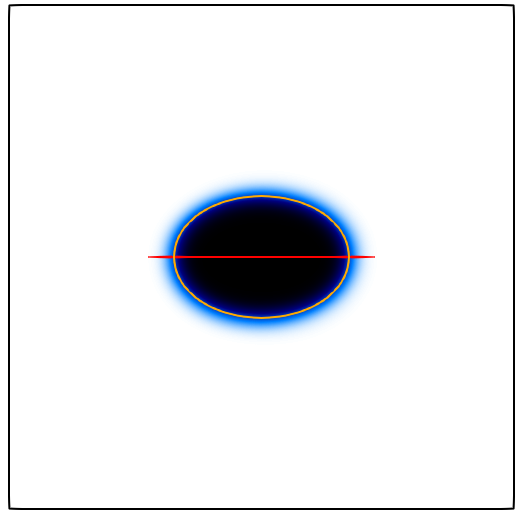}\hfill
\includegraphics[width=\scale\textwidth]{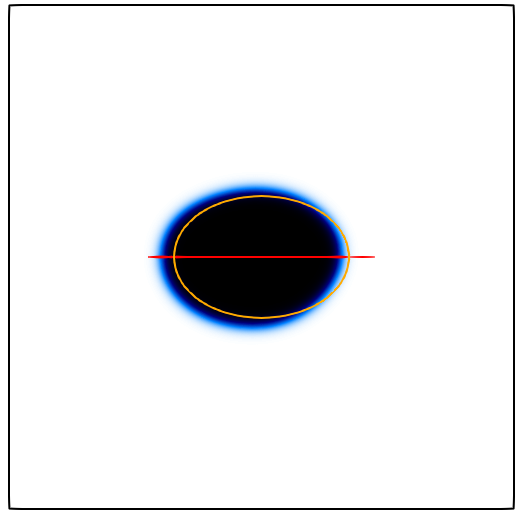}\hfill
\includegraphics[width=\scale\textwidth]{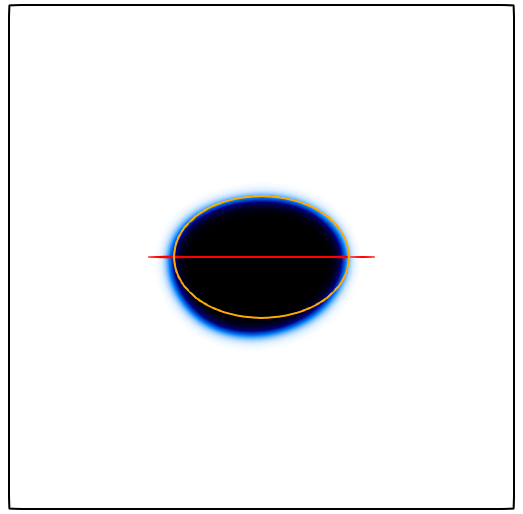}
}\\

\subfloat[][$t=1.5$]{
\includegraphics[width=\scale\textwidth]{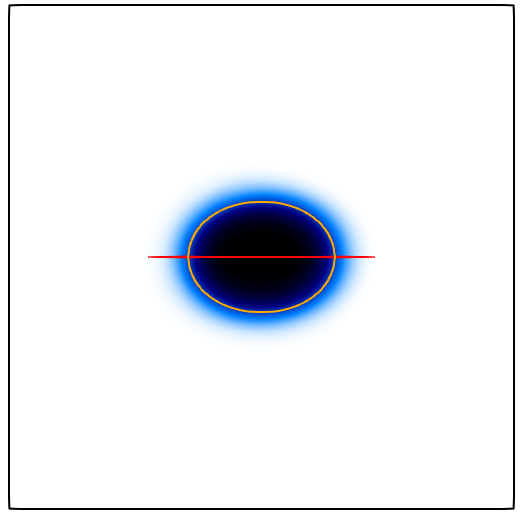}\hfill
\includegraphics[width=\scale\textwidth]{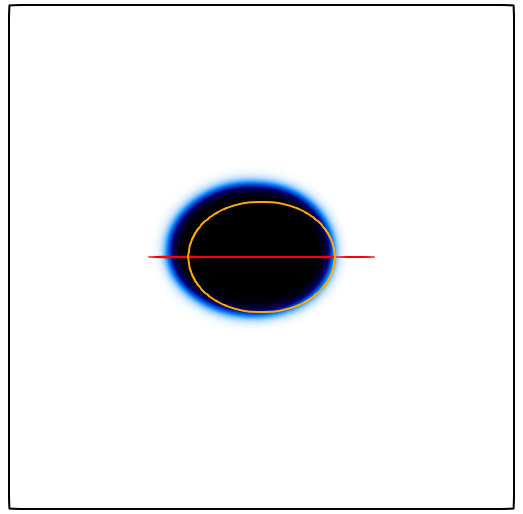}\hfill
\includegraphics[width=\scale\textwidth]{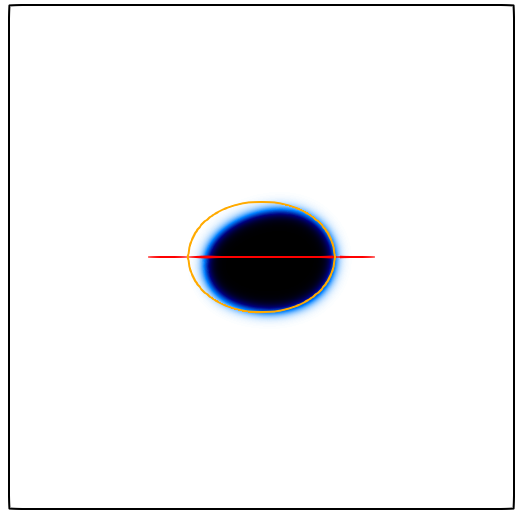}
}\\


\subfloat[][$t=3.0$]{
\includegraphics[width=\scale\textwidth]{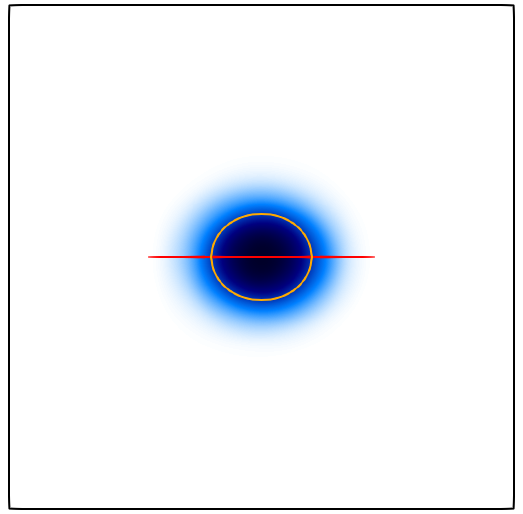}\hfill
\includegraphics[width=\scale\textwidth]{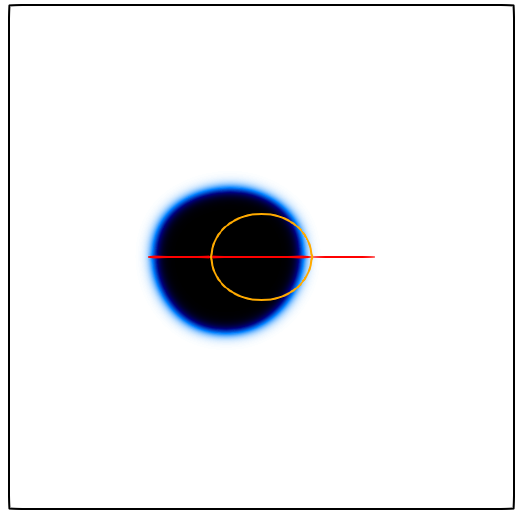}\hfill
\includegraphics[width=\scale\textwidth]{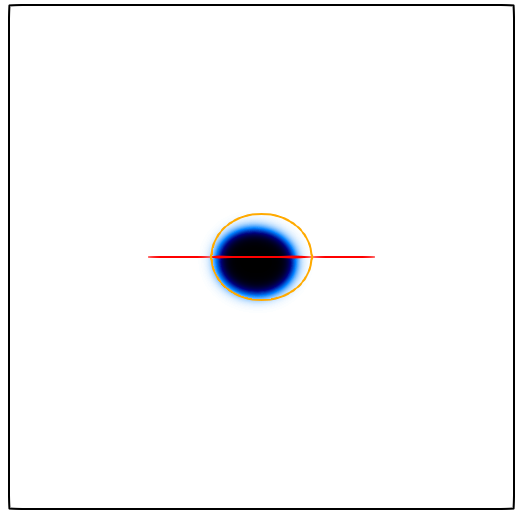}
}\\

\subfloat[][$t=4.0$]{
\includegraphics[width=\scale\textwidth]{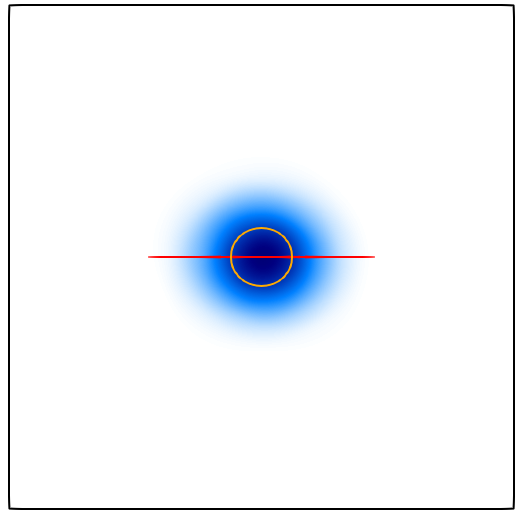}\hfill
\includegraphics[width=\scale\textwidth]{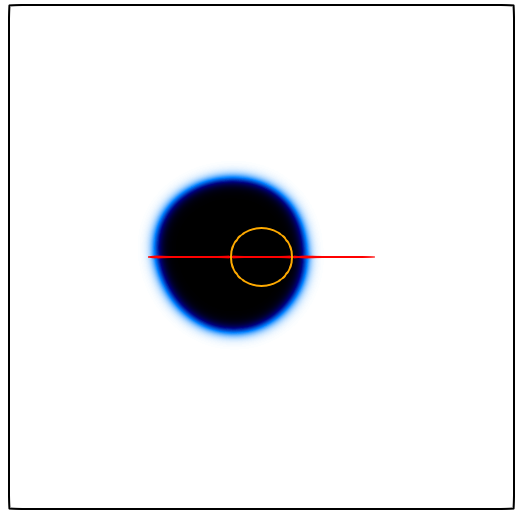}\hfill
\includegraphics[width=\scale\textwidth]{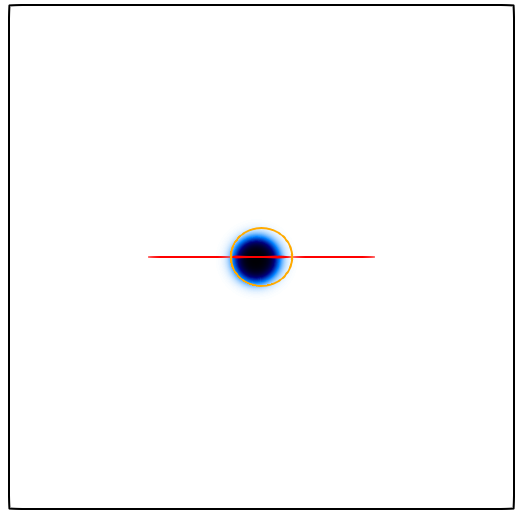}
}
\end{center}
\caption{Evolution of a droplet computed using scheme \eqref{eq:discscheme} with $\tau=5\cdot10^{-4}$. The orange line indicates the zero level set of the deterministic evolution and the red line indicates the area depicted in the line plots in Figs.~\ref{fig:linePlotExpected}, \ref{fig:linePlotPath14}, and \ref{fig:linePlotPath16}. The left column depicts the expected value, while the middle and the right column show two individual sample paths.}
\label{fig:simpleEvolution}
\end{figure}

\begin{table}
\begin{tabular}{c|cc|cc}
&\multicolumn{2}{c}{nonlinear \eqref{eq:majeeprohl}}&\multicolumn{2}{c}{augmented~SAV \eqref{eq:discscheme}}\\
$\tau$ & $\sqrt{\expected{\norm{\cdot}_{L^2\trkla{\Om}}^2\big\vert_{t=4}}}$& EOC& $\sqrt{\expected{\norm{\cdot}_{L^2\trkla{\Om}}^2\big\vert_{t=4}}}$& EOC\\
\hline
$1\cdot10^{-3}$&2.11E-02&--&2.83E-02&--\\
$2\cdot10^{-3}$&5.62E-02&1.42&7.68E-02&1.44\\
$4\cdot10^{-3}$&1.17E-01&1.05&1.59E-01&1.05\\
$1\cdot10^{-2}$&2.64E-01&0.89&3.39E-01&0.83\\
$2\cdot10^{-2}$&4.56E-01&0.79&5.70E-01&0.75
\end{tabular}
\caption{Experimental order of convergence w.r.t.~$\tau$ for the nonlinear scheme \eqref{eq:majeeprohl} and the augmented SAV method \eqref{eq:discscheme} using Rademacher random variables and $\alpha=2$. }
\label{tab:moderateEOCrandomwalk}
\end{table}

\begin{table}
\begin{tabular}{cc|ll|ll}
&&\multicolumn{2}{c}{Rademacher}&\multicolumn{2}{c}{$\mathcal{N}\trkla{0,1}$}\\
&$\tau$& nonlinear& aug.~SAV & nonlinear & aug.~SAV\\
\hline
{~\color{colorONE}$\blacksquare$}&$5\cdot10^{-4}$&$\approx60$ min& $\approx 34$ min    & $\approx 71.5$ min & $\approx 35$ min\\
{~\color{colorTWO}$\blacksquare$}&$1\cdot10^{-3}$&$\approx40$ min& $\approx 18$ min    & $\approx 47$ min & $\approx 20$ min\\
{~\color{colorTHREE}$\blacksquare$}&$2\cdot10^{-3}$&$\approx28$ min& $\approx 11$ min    & $\approx 31.5$ min & $\approx 13$ min \\
{~\color{colorFOUR}$\blacksquare$}&$4\cdot10^{-3}$&$\approx 22$ min& $\approx 7.5$ min  & $ \approx 22$ min & $\approx 8$ min\\
{~\color{colorFIVE}$\blacksquare$}&$1\cdot10^{-2}$&$\approx 12$ min& $\approx 5$ min    & $\approx 14.5$ min & $\approx 5.5$ min\\
{~\color{colorSIX}$\blacksquare$}&$2\cdot10^{-2}$&$\approx 10.5$ min& $\approx 3.5$ min& $\approx 10 $ min & $\approx 4$ min
\end{tabular}
\caption{Computational costs per path for the simulations in Sec.~\ref{subsec:moderate}.}
\label{tab:moderateComputationalCostsrandomwalk}
\end{table}

\begin{figure}
\subfloat[][$t=2$]{
\includegraphics[width=\textwidth]{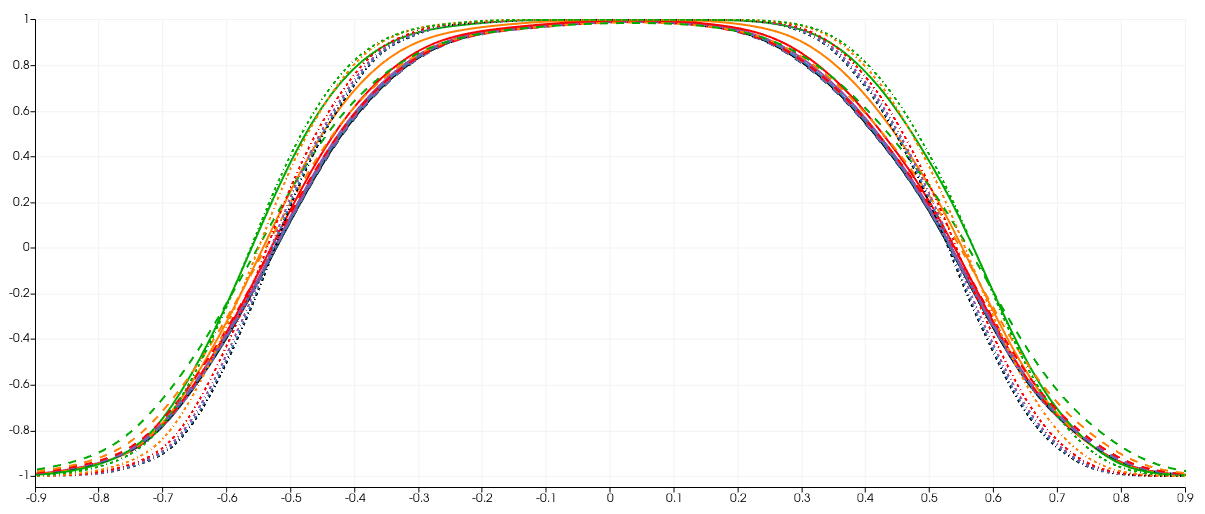}
}\\

\subfloat[][$t=4$]{
\includegraphics[width=\textwidth]{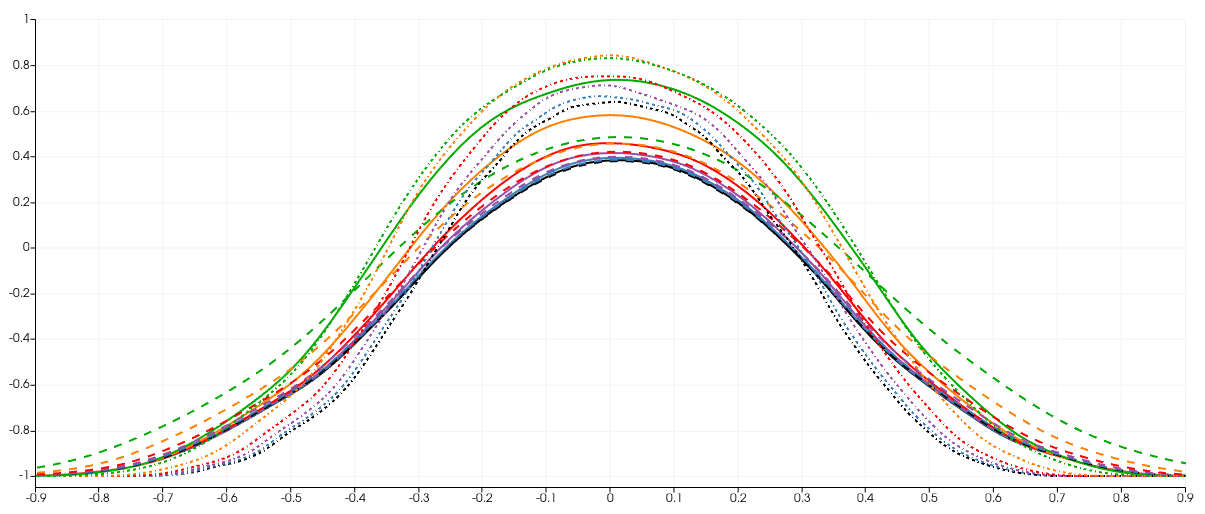}
}
\caption{Line plot of expected values (based on 350 samples) of the phase-field profile using different numerical schemes (implicit scheme \eqref{eq:majeeprohl} solid, augmented SAV \eqref{eq:discscheme} dashed, SAV \eqref{eq:SAV} dotted). Time-increments are color-coded ({~\color{colorONE}$\blacksquare$} $\tau=5\cdot10^{-4}$, {~\color{colorTWO}$\blacksquare$} $\tau=1\cdot10^{-3}$, {~\color{colorTHREE}$\blacksquare$} $\tau=2\cdot10^{-3}$,
{~\color{colorFOUR}$\blacksquare$} $\tau=4\cdot10^{-3}$,
{~\color{colorFIVE}$\blacksquare$} $\tau=1\cdot10^{-2}$, {~\color{colorSIX}$\blacksquare$} $\tau=2\cdot10^{-2}$).}
\label{fig:linePlotExpected}
\end{figure}

\begin{figure}
\includegraphics[width=\textwidth]{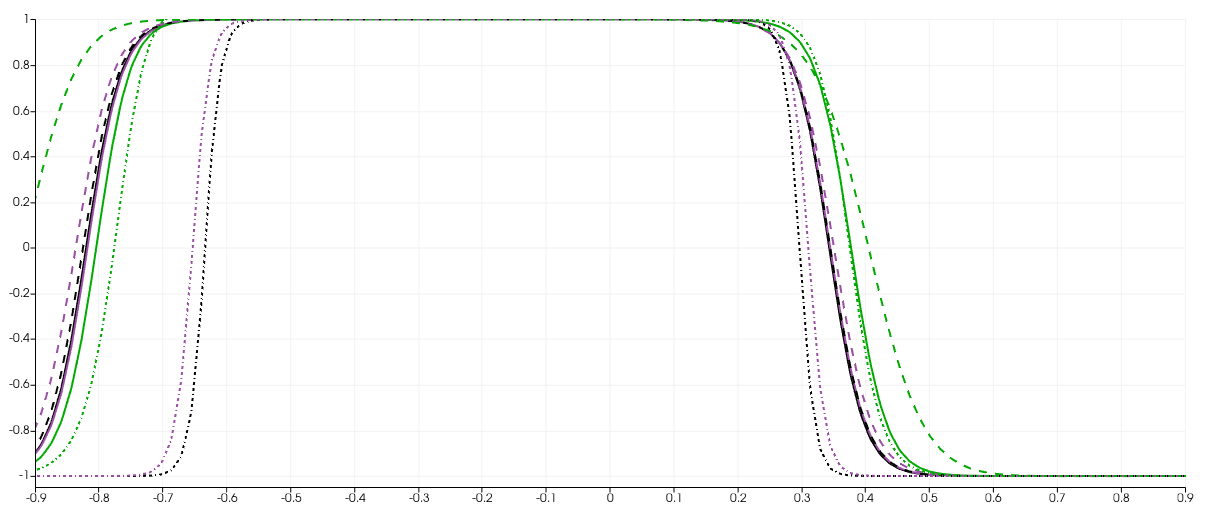}
\caption{Line plots of phase-field profiles at $t=4$ for one sample path using different numerical schemes (implicit scheme \eqref{eq:majeeprohl} solid, augmented SAV \eqref{eq:discscheme} dashed, SAV \eqref{eq:SAV} dotted) and different time increments ({~\color{colorONE}$\blacksquare$} $\tau=5\cdot10^{-4}$, {~\color{colorTHREE}$\blacksquare$} $\tau=2\cdot10^{-3}$, {~\color{colorSIX}$\blacksquare$} $\tau=2\cdot10^{-2}$).}
\label{fig:linePlotPath14}
\end{figure}

\begin{figure}
\subfloat[][$t=3$]{
\includegraphics[width=\textwidth]{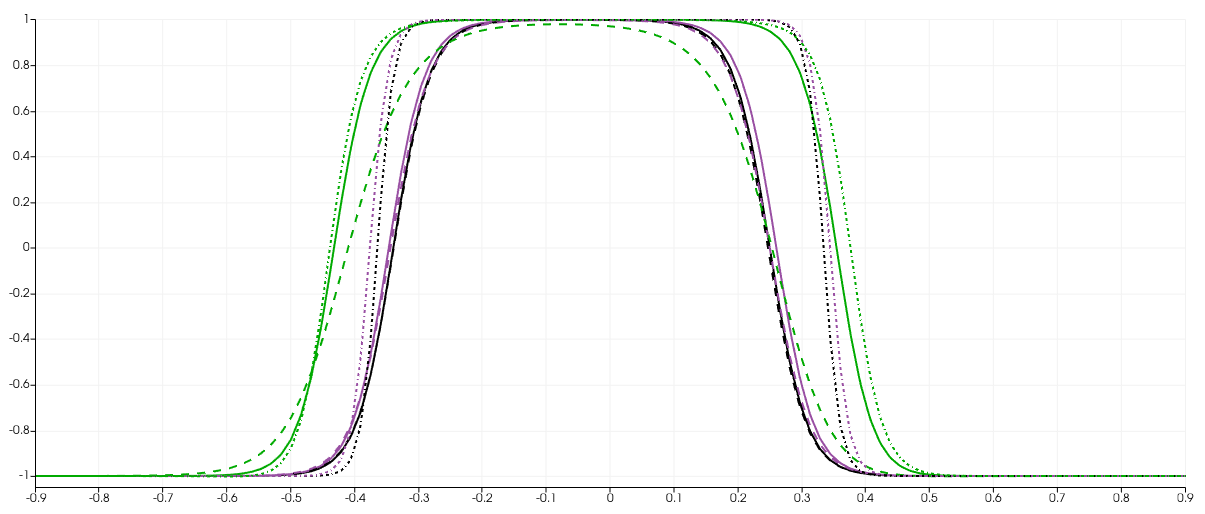}
}

\subfloat[][$t=4$]{
\includegraphics[width=\textwidth]{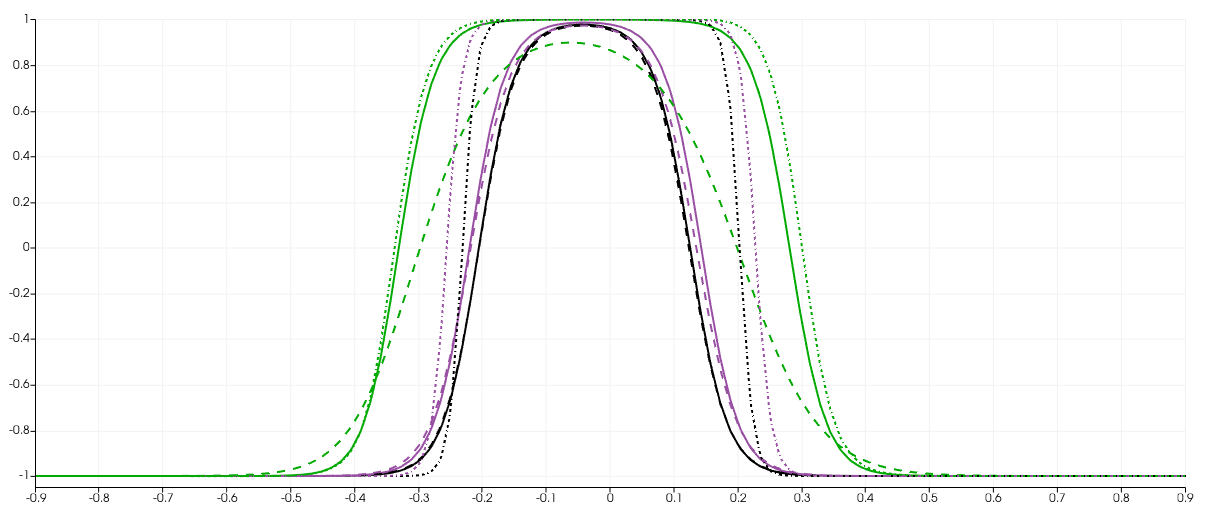}
}
\caption{Line plots of phase-field profiles for one sample path using different numerical schemes (implicit scheme \eqref{eq:majeeprohl} solid, augmented SAV \eqref{eq:discscheme} dashed, SAV \eqref{eq:SAV} dotted) and different time increments ({~\color{colorONE}$\blacksquare$} $\tau=5\cdot10^{-4}$, {~\color{colorTHREE}$\blacksquare$} $\tau=2\cdot10^{-3}$, {~\color{colorSIX}$\blacksquare$} $\tau=2\cdot10^{-2}$).}
\label{fig:linePlotPath16}
\end{figure}

\begin{table}
\begin{tabular}{c|cc|cc}
&\multicolumn{2}{c}{augmented~SAV}&\multicolumn{2}{c}{standard SAV}\\
$\tau$ & $\sqrt{\expected{\norm{\cdot}_{L^2\trkla{\Om}}^2\big\vert_{t=4}}}$& EOC& $\sqrt{\expected{\norm{\cdot}_{L^2\trkla{\Om}}^2\big\vert_{t=4}}}$& EOC\\
\hline
$5\cdot10^{-4}$& 2.05E-02 &--&   5.55E-01&--\\
$1\cdot10^{-3}$& 4.05E-02 &0.98& 5.43E-01&-0.03\\
$2\cdot10^{-3}$& 7.75E-02 &0.94& 5.22E-01&-0.06\\
$4\cdot10^{-3}$& 1.43E-01 &0.88& 4.91E-01&-0.09\\
$1\cdot10^{-2}$& 2.98E-01 &0.81& 4.00E-01&-0.22\\
$2\cdot10^{-2}$& 5.11E-01 &0.78& 1.64E-01&-1.29
\end{tabular}
\caption{Comparison between SAV methods and the nonlinear scheme \eqref{eq:majeeprohl} using Rademacher random variables and $\alpha=2$.}
\label{tab:moderateComparisonNewtonSAVrandomwalk}
\end{table}

\begin{table}
\begin{tabular}{c|cc|cc}
&\multicolumn{2}{c}{nonlinear}&\multicolumn{2}{c}{augmented~SAV}\\
$\tau$ & $\sqrt{\expected{\norm{\cdot}_{L^2\trkla{\Om}}^2\big\vert_{t=4}}}$& EOC& $\sqrt{\expected{\norm{\cdot}_{L^2\trkla{\Om}}^2\big\vert_{t=4}}}$& EOC\\
\hline
$1\cdot10^{-3}$&2.43E-02&--&3.18E-02&--\\
$2\cdot10^{-3}$&6.40E-02&1.39&7.96E-02&1.32\\
$4\cdot10^{-3}$&1.39E-01&1.12&1.51E-01&0.93\\
$1\cdot10^{-2}$&3.21E-01&0.91&2.88E-01&0.70\\
$2\cdot10^{-2}$&5.39E-01&0.75&4.31E-01&0.58\\
\end{tabular}
\caption{Experimental order of convergence w.r.t.~$\tau$ for the nonlinear scheme \eqref{eq:majeeprohl} and the augmented SAV method \eqref{eq:discscheme} using $\mathcal{N}\trkla{0,1}$-random variables and $\alpha=2$.}
\label{tab:moderateEOCnormal}
\end{table}

\begin{table}
\begin{tabular}{c|cc|cc}
&\multicolumn{2}{c}{augmented~SAV}&\multicolumn{2}{c}{standard SAV}\\
$\tau$ & $\sqrt{\expected{\norm{\cdot}_{L^2\trkla{\Om}}^2\big\vert_{t=4}}}$& EOC& $\sqrt{\expected{\norm{\cdot}_{L^2\trkla{\Om}}^2\big\vert_{t=4}}}$& EOC\\
\hline
$5\cdot10^{-4}$& 3.02E-02 &--&   3.31E-01&--\\
$1\cdot10^{-3}$& 5.40E-02 &0.84& 3.21E-01&-0.04\\
$2\cdot10^{-3}$& 9.02E-02 &0.74& 3.11E-01&-0.04\\
$4\cdot10^{-3}$& 1.41E-01 &0.64& 3.04E-01&-0.03\\
$1\cdot10^{-2}$& 2.30E-01 &0.54& 3.03E-01&-0.00\\
$2\cdot10^{-2}$& 3.21E-01 &0.48& 2.75E-01&-0.14
\end{tabular}
\caption{Comparison between SAV methods and the nonlinear scheme \eqref{eq:majeeprohl} using $\mathcal{N}\trkla{0,1}$-random variables and $\alpha=2$.}
\label{tab:moderateComparisonNewtonSAVnormal}
\end{table}

\subsection{Strong noise}\label{subsec:strong}
In this section, we discuss the performance of the different numerical schemes for larger stochastic source terms.
In particular, we set $\alpha=10$.
As this requires smaller time increments, we approximate the $\mathcal{Q}$-Wiener process using \eqref{eq:defincremet} with $\tau_{\text{min}}=5\cdot10^{-5}$ and $L=K=3$, and compare numerical solutions for 
$\tau\in\tgkla{5\cdot 10^{-5}, 1\cdot10^{-4},2\cdot10^{-4},4\cdot10^{-4},1\cdot10^{-3},2\cdot10^{-3}}$ in the time interval $\tekla{0,2}$.
Similar to the last section, we start by considering Rademacher-distributed increments.
The evolution of the expected value of the phase-field parameter is depicted in the left column of Fig.~\ref{fig:strongEvolution}.
The other columns of Fig.~\ref{fig:strongEvolution} show the droplet evolution for two independent sample paths.
Again, the orange line indicates the zero level set of the deterministic evolution.
When comparing Fig.~\ref{fig:strongEvolution} and Fig.~\ref{fig:simpleEvolution}, the difference in the droplet shapes is striking.
The less regular droplet shapes with large variations in the mean curvature indicate that in our second example the stochastic source terms are sufficiently large to counteract the regularizing effects of the mean curvature flow.\\
As in the last section, an experimental order of convergence was computed for the nonlinear scheme \eqref{eq:majeeprohl} and the augmented SAV scheme \eqref{eq:discscheme} using their respective solutions for $\tau_{\text{min}}=5\cdot10^{-5}$ as reference.
The results are collected in Tab.~\ref{tab:strongEOCrandomwalk}.
In comparison to the results presented in the previous section, the experimental order of convergence is smaller and deteriorates more for larger time increments.
Again, a pathwise comparison of the solutions obtained from the different schemes (cf.~Tab.~\ref{tab:strongComparisonNewtonSAVrandomwalk}) shows that the augmented SAV scheme still converges towards the same solution as \eqref{eq:majeeprohl} while the SAV scheme without augmentation \eqref{eq:SAV} does not.\\
In order to visualize the differences between these schemes, the profile of the expected value of the phase-field parameter along the line $(-2,0)$---$(2,0)$ (cf.~red line in Fig.~\ref{fig:strongEvolution}) is depicted in Fig.~\ref{fig:stronglinePlotExpected} for different time increments.
Again, solutions based on \eqref{eq:majeeprohl} are plotted using solid lines, solution obtained from the augmented SAV scheme \eqref{eq:discscheme} are plotted using dashed lines, and the solutions computed via \eqref{eq:SAV} are shown as dotted lines.
These profile lines show good agreement between the expected values computed via \eqref{eq:majeeprohl} and \eqref{eq:discscheme}.
The standard SAV scheme without augmentation on the other hand provides different results. 
The profiles for the individual paths depicted in the middle and right columns of Fig.~\ref{fig:strongEvolution} are shown in Figs.~\ref{fig:stronglinePlotPath103} \& \ref{fig:stronglinePlotPath104}.
These profiles indicate that using the augmented SAV scheme with time increments that are too large results in unwanted oscillations.
These oscillations can be attributed to the fact that SAV schemes are not based on the original energy $\iO\tfrac\varepsilon2\tabs{\nabla\phi}^2+\tfrac{1}{4\varepsilon} \trkla{1-\phi^2}^2\dx$ but only on the modified energy $\iO\tfrac{\varepsilon}2\tabs{\nabla\phi}^2\dx +\tfrac1\varepsilon\tabs{r}^2$.
Hence, the main mechanism to keep $\phi$ close to the interval $\tekla{-1,+1}$ is severely weakened.
For smaller time increments however, the profiles computed with the augmented SAV scheme align with the ones obtained via \eqref{eq:majeeprohl}.
The phase-field profiles computed with \eqref{eq:SAV} again predict a sharper transition between the phases.
As the stochastic source terms are placed in the interfacial area, the solutions to \eqref{eq:SAV} are less susceptible to the noise.
Hence, there are less oscillations for large time increments but there is also no convergence towards the correct solution for small increments.
A simulation of the path depicted in the right column of Fig.~\ref{fig:strongEvolution} using \eqref{eq:SAV} with the smallest considered size of the time increment $\tau_{\text{min}}=5\cdot10^{-5}$ is depicted in Fig.~\ref{fig:SAVsimple}.
Here, it is noticeable that the interfacial area in Fig.~\ref{fig:SAVsimple} is almost not visible and the predicted evolution of the droplet deviates significantly from the one computed with the other schemes.\\
Concerning the computational costs for the different schemes, the data collected in Tab.~\ref{tab:strongComputationalCosts} suggest that using the augmented SAV scheme instead of the nonlinear approach \eqref{eq:majeeprohl} with Newton's method leads to a speed-up by more than factor 2. \\
When using $\mathcal{N}\trkla{0,1}$-distributed random variables for the increments in \eqref{eq:defincremet} we obtain similar results.
As it can be seen in Tab.~\ref{tab:strongEOCnormal}, the deviations from the reference solutions computed with $\tau_{\text{min}}=5\cdot10^{-5}$ are larger for the augmented SAV scheme.
However, as Tab.~\ref{tab:strongComparisonNewtonSAVnormal} shows, the deviations between the solutions to the nonlinear scheme \eqref{eq:majeeprohl} and the ones to the augmented SAV scheme vanish for $\tau\rightarrow0$, which indicates convergence towards the correct solution.
Solutions computed via \eqref{eq:SAV}, on the other hand, do not converge towards the same limit.

\begin{figure}
\begin{center}
\newcommand{\scale}{0.27}
\subfloat[][$t=0.5$]{
\includegraphics[width=\scale\textwidth]{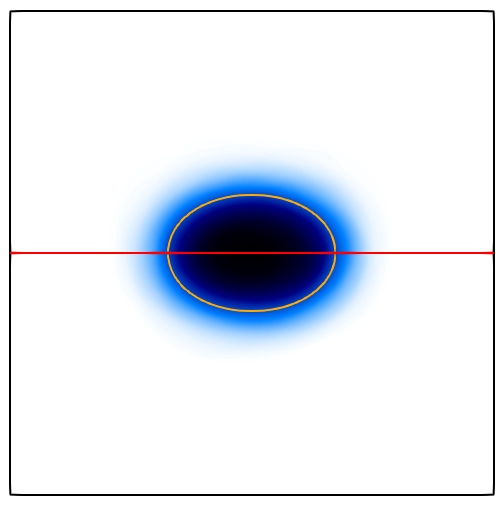}\hfill
\includegraphics[width=\scale\textwidth]{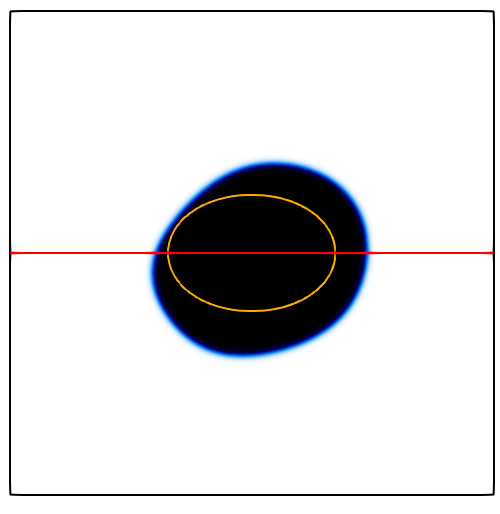}\hfill
\includegraphics[width=\scale\textwidth]{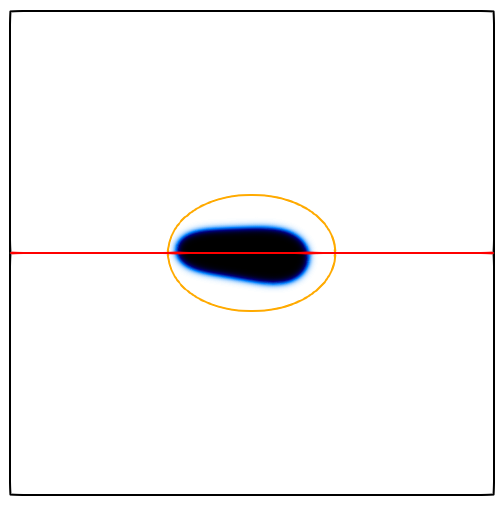}
}\\

\subfloat[][$t=1.0$]{
\includegraphics[width=\scale\textwidth]{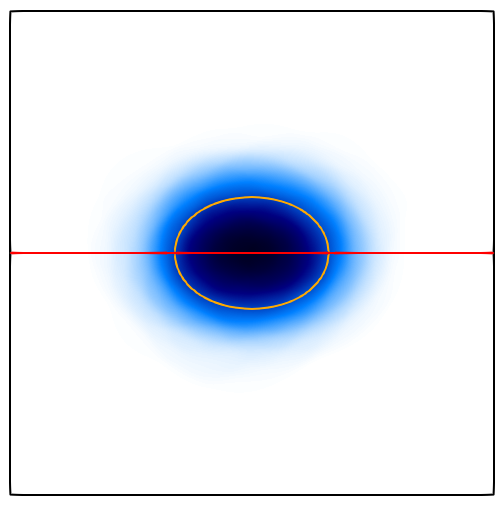}\hfill
\includegraphics[width=\scale\textwidth]{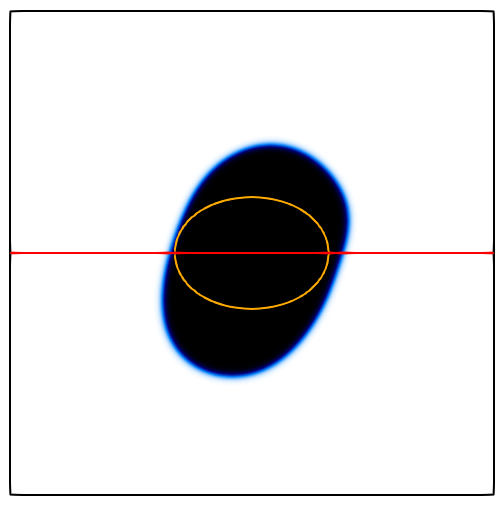}\hfill
\includegraphics[width=\scale\textwidth]{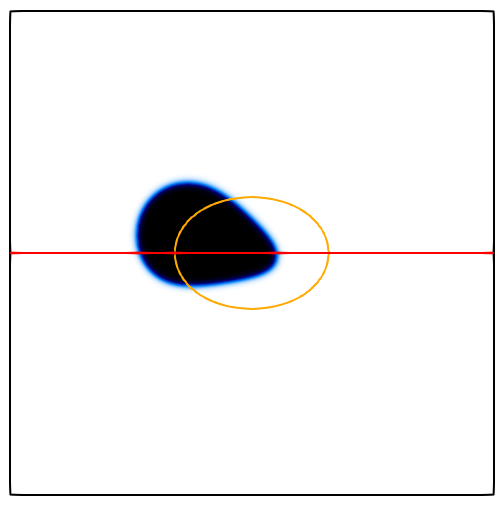}
}\\

\subfloat[][$t=1.5$]{
\includegraphics[width=\scale\textwidth]{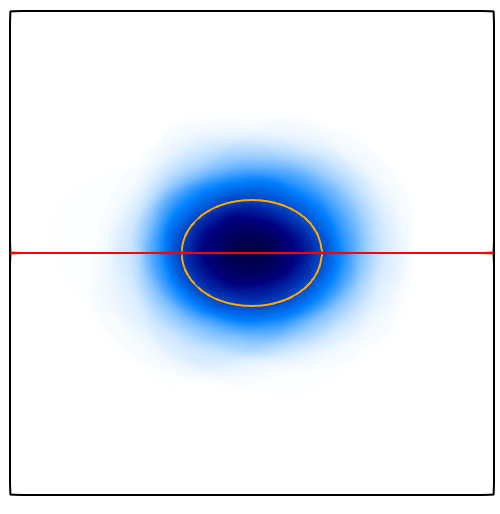}\hfill
\includegraphics[width=\scale\textwidth]{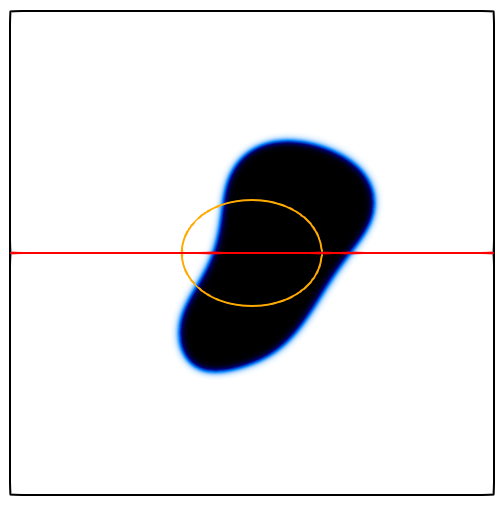}\hfill
\includegraphics[width=\scale\textwidth]{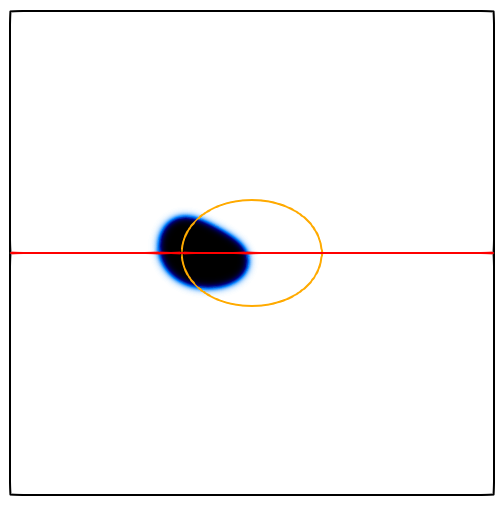}
}\\

\subfloat[][$t=2.0$]{
\includegraphics[width=\scale\textwidth]{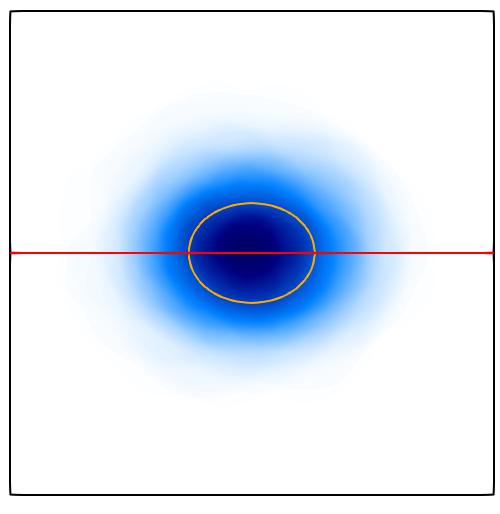}\hfill
\includegraphics[width=\scale\textwidth]{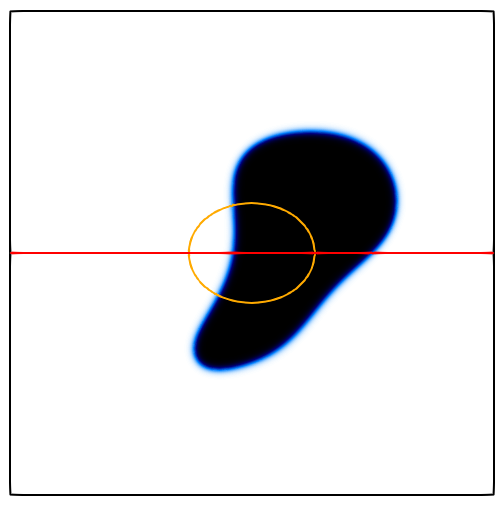}\hfill
\includegraphics[width=\scale\textwidth]{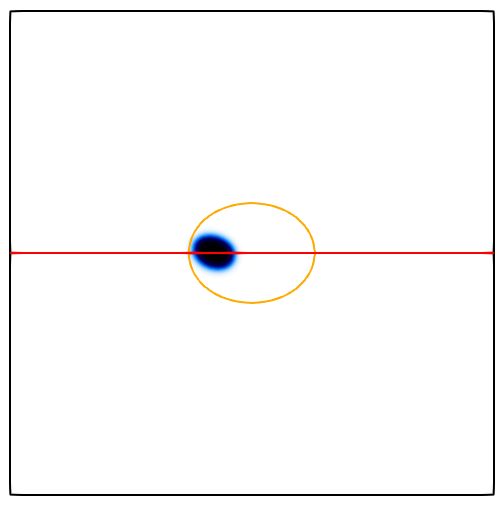}
}
\end{center}
\caption{Evolution of a droplet computed using scheme \eqref{eq:discscheme} with $\tau=5\cdot10^{-5}$. The orange line indicates the zero level set of the deterministic evolution and the red line indicates the area depicted in the line plots in Figs.~\ref{fig:stronglinePlotExpected}, \ref{fig:stronglinePlotPath103}, and \ref{fig:stronglinePlotPath104}. The left column depicts the expected value, while the middle and the right column show two individual sample paths.}
\label{fig:strongEvolution}
\end{figure}

\begin{table}
\begin{tabular}{c|cc|cc}
&\multicolumn{2}{c}{nonlinear}&\multicolumn{2}{c}{augmented~SAV}\\
$\tau$ & $\sqrt{\expected{\norm{\cdot}_{L^2\trkla{\Om}}^2\big\vert_{t=2}}}$& EOC& $\sqrt{\expected{\norm{\cdot}_{L^2\trkla{\Om}}^2\big\vert_{t=2}}}$& EOC\\
\hline
$1\cdot10^{-4}$&1.09E-01&--&2.57E-01&--\\
$2\cdot10^{-4}$&1.91E-01&0.80&5.30E-01&1.04\\
$4\cdot10^{-4}$&2.68E-01&0.49&7.73E-01&0.54\\
$1\cdot10^{-3}$&4.04E-01&0.45&1.04E-00&0.33\\
$2\cdot10^{-3}$&5.20E-01&0.36&1.30E-00&0.31
\end{tabular}
\caption{Experimental order of convergence w.r.t.~$\tau$ for the nonlinear scheme \eqref{eq:majeeprohl} and the augmented SAV method \eqref{eq:discscheme} using Rademacher random variables and $\alpha=10$.}
\label{tab:strongEOCrandomwalk}
\end{table}

\begin{table}
\begin{tabular}{c|cc|cc}
&\multicolumn{2}{c}{augmented~SAV}&\multicolumn{2}{c}{standard SAV}\\
$\tau$ & $\sqrt{\expected{\norm{\cdot}_{L^2\trkla{\Om}}^2\big\vert_{t=2}}}$& EOC& $\sqrt{\expected{\norm{\cdot}_{L^2\trkla{\Om}}^2\big\vert_{t=2}}}$& EOC\\
\hline
$5\cdot10^{-5}$& 1.86E-01 &--&   1.53E-00&--\\
$1\cdot10^{-4}$& 3.53E-01& 0.85& 1.51E-00&-0.02\\
$2\cdot10^{-4}$& 5.98E-01& 0.76& 1.49E-00&-0.02\\
$4\cdot10^{-4}$& 7.81E-01& 0.39& 1.45E-00&-0.04\\
$1\cdot10^{-3}$& 1.03E-00& 0.30& 1.36E-00&-0.07\\
$2\cdot10^{-3}$& 1.26E-00& 0.30& 1.20E-00&-0.17
\end{tabular}
\caption{Comparison between SAV methods and the nonlinear scheme \eqref{eq:majeeprohl} using Rademacher random variables and $\alpha=10$.}
\label{tab:strongComparisonNewtonSAVrandomwalk}
\end{table}

\begin{table}
\begin{tabular}{cc|ll|ll}
&&\multicolumn{2}{c}{Rademacher}&\multicolumn{2}{c}{$\mathcal{N}\trkla{0,1}$}\\
&$\tau$& nonlinear& aug.~SAV & nonlinear & aug.~SAV\\
\hline
{~\color{colorONE}$\blacksquare$}&$5\cdot10^{-5}$&$\approx216$ min& $\approx 86$ min    & $\approx 214.5$ min & $\approx 84.5$ min\\
{~\color{colorTWO}$\blacksquare$}&$1\cdot10^{-4}$&$\approx125$ min& $\approx 51$ min    & $\approx 120$ min & $\approx 51.5$ min\\
{~\color{colorTHREE}$\blacksquare$}&$2\cdot10^{-4}$&$\approx73$ min& $\approx 28$ min    & $\approx 75$ min & $\approx 28$ min \\
{~\color{colorFOUR}$\blacksquare$}&$4\cdot10^{-4}$&$\approx 45$ min& $\approx 17$ min  & $ \approx 45$ min & $\approx 17$ min\\
{~\color{colorFIVE}$\blacksquare$}&$1\cdot10^{-3}$&$\approx 28$ min& $\approx 9.5$ min    & $\approx 28.5$ min & $\approx 9$ min\\
{~\color{colorSIX}$\blacksquare$}&$2\cdot10^{-3}$&$\approx 20$ min& $\approx 6.5$ min& $\approx 20 $ min & $\approx 6$ min
\end{tabular}
\caption{Computational costs per path for the simulations in Sec.~\ref{subsec:strong}.}
\label{tab:strongComputationalCosts}
\end{table}

\begin{figure}
\subfloat[][$t=1$]{
\includegraphics[width=\textwidth]{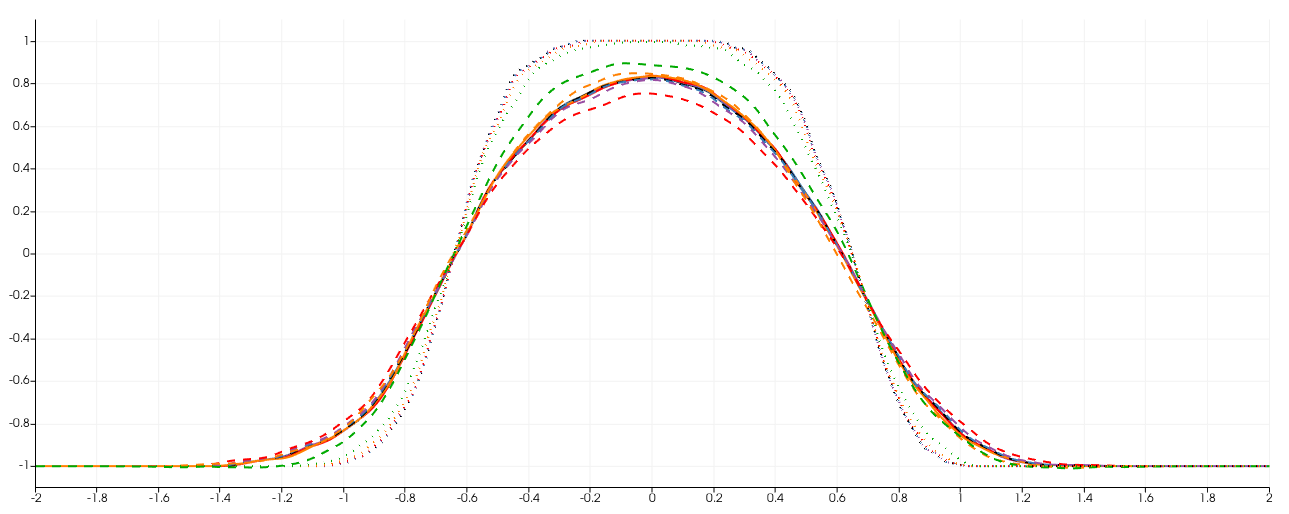}}\\

\subfloat[][$t=2$]{
\includegraphics[width=\textwidth]{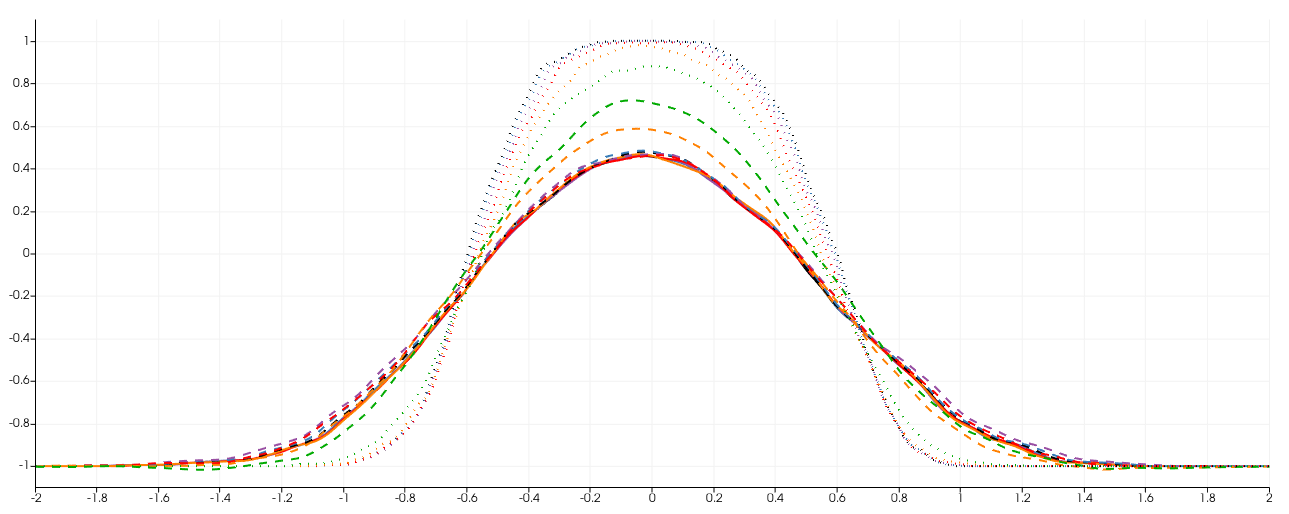}
}
\caption{Line plot of expected values (based on 350 samples) of the phase-field profile using different numerical schemes (implicit scheme \eqref{eq:majeeprohl} solid, augmented SAV \eqref{eq:discscheme} dashed, SAV \eqref{eq:SAV} dotted). Time-increments are color-coded:{~\color{colorONE}$\blacksquare$} $\tau=5\cdot10^{-5}$, {~\color{colorTWO}$\blacksquare$} $\tau=1\cdot10^{-4}$, {~\color{colorTHREE}$\blacksquare$} $\tau=2\cdot10^{-4}$,
{~\color{colorFOUR}$\blacksquare$} $\tau=4\cdot10^{-4}$,
{~\color{colorFIVE}$\blacksquare$} $\tau=1\cdot10^{-3}$, {~\color{colorSIX}$\blacksquare$} $\tau=2\cdot10^{-3}$.}
\label{fig:stronglinePlotExpected}
\end{figure}

\begin{figure}
\subfloat[][$t=1$]{
\includegraphics[width=\textwidth]{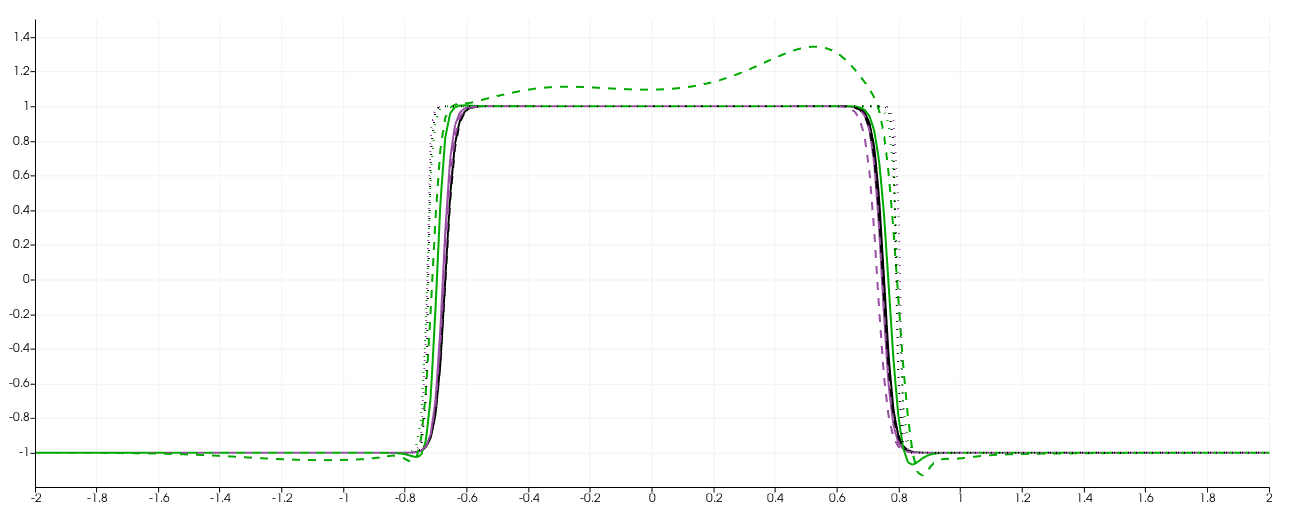}}\\

\subfloat[][$t=2$]{
\includegraphics[width=\textwidth]{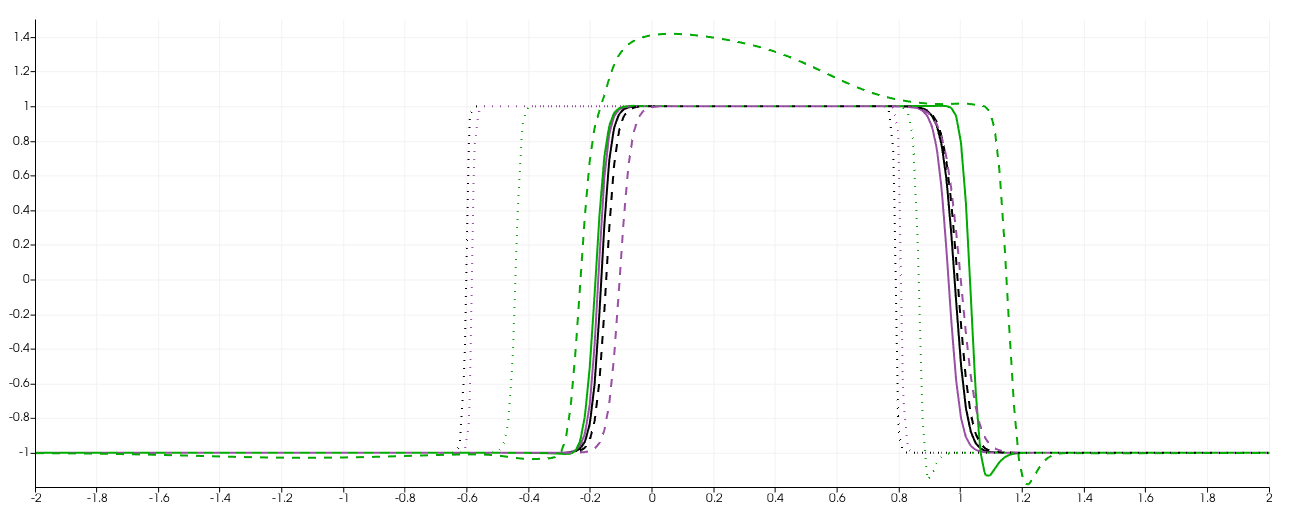}}
\caption{Line plots of phase-field profiles for one sample path using different numerical schemes (implicit scheme \eqref{eq:majeeprohl} solid, augmented SAV \eqref{eq:discscheme} dashed, SAV \eqref{eq:SAV} dotted) and different time increments ({~\color{colorONE}$\blacksquare$} $\tau=5\cdot10^{-5}$, {~\color{colorTHREE}$\blacksquare$} $\tau=2\cdot10^{-4}$, {~\color{colorSIX}$\blacksquare$} $\tau=2\cdot10^{-3}$).}
\label{fig:stronglinePlotPath103}
\end{figure}

\begin{figure}
\includegraphics[width=\textwidth]{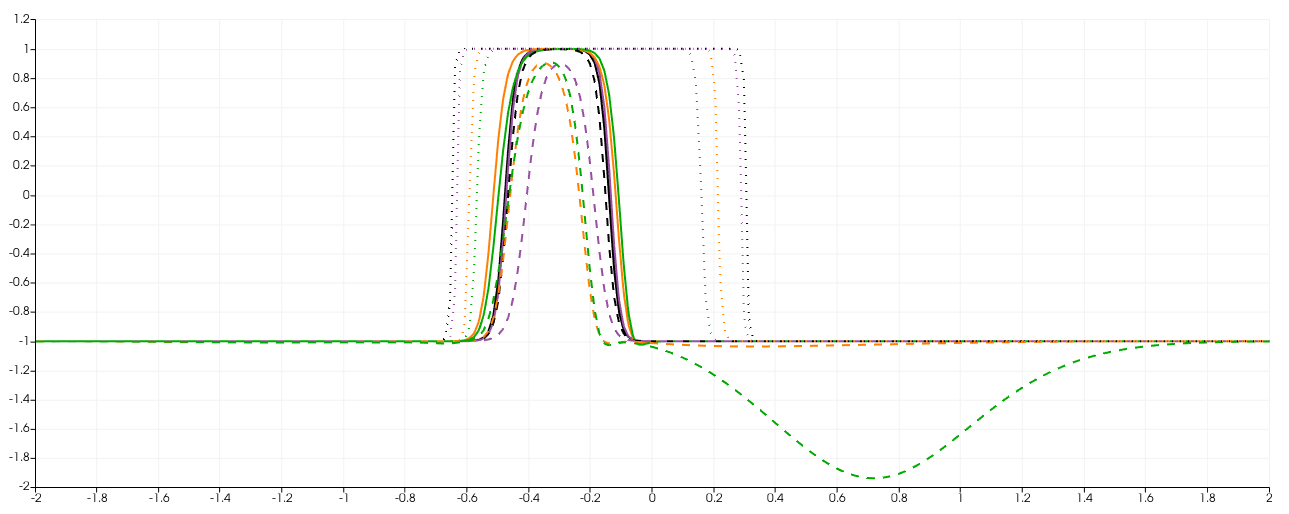}
\caption{Line plots of phase-field profiles for one sample path at $t=2$ using different numerical schemes (implicit scheme \eqref{eq:majeeprohl} solid, augmented SAV \eqref{eq:discscheme} dashed, SAV \eqref{eq:SAV} dotted) and different time increments ({~\color{colorONE}$\blacksquare$} $\tau=5\cdot10^{-5}$, {~\color{colorTHREE}$\blacksquare$} $\tau=2\cdot10^{-4}$, {~\color{colorFIVE}$\blacksquare$} $\tau=1\cdot10^{-3}$, {~\color{colorSIX}$\blacksquare$} $\tau=2\cdot10^{-3}$).}
\label{fig:stronglinePlotPath104}
\end{figure}

\begin{figure}
\begin{center}
\newcommand{\scale}{0.24}
\subfloat[][$t=0.5$]{
\includegraphics[width=\scale\textwidth]{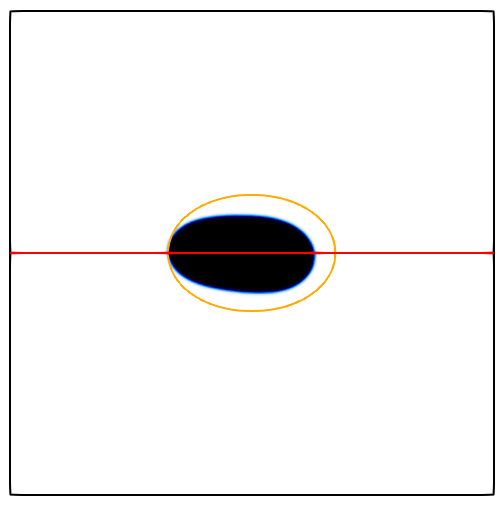}
}
\subfloat[][$t=1.0$]{
\includegraphics[width=\scale\textwidth]{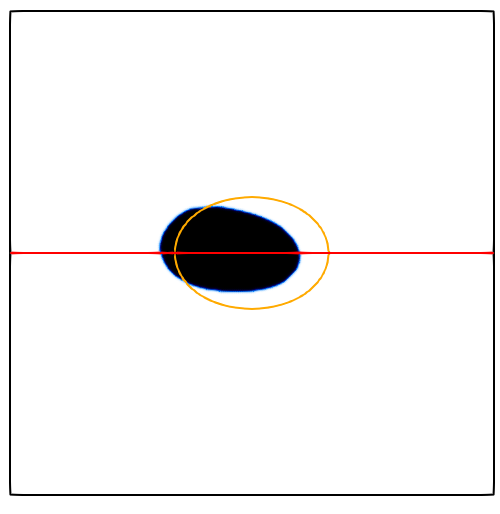}
}
\subfloat[][$t=1.5$]{
\includegraphics[width=\scale\textwidth]{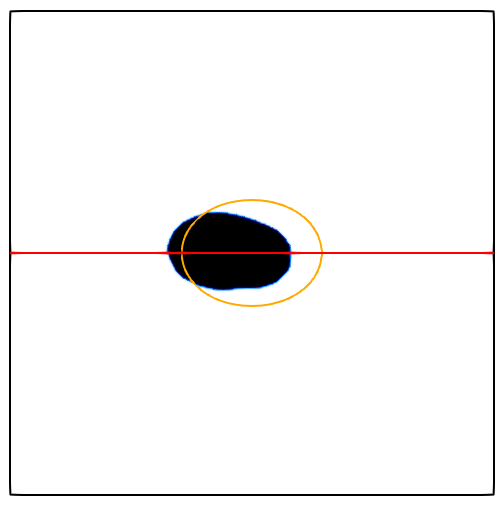}
}
\subfloat[][$t=2.0$]{
\includegraphics[width=\scale\textwidth]{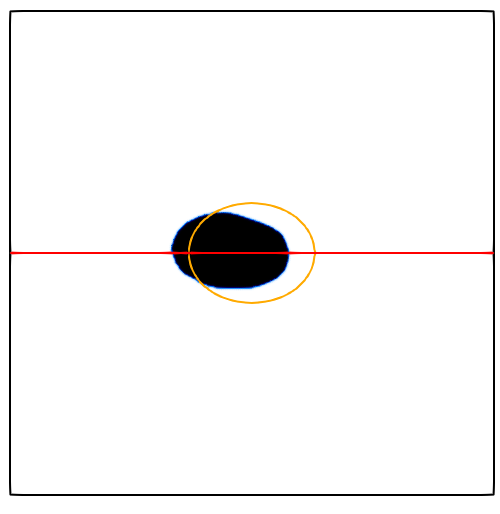}
}
\end{center}
\caption{Evolution of a droplet computed using scheme \eqref{eq:SAV} with $\tau=5\cdot10^{-5}$. The depicted path corresponds to the right column of Fig.~\ref{fig:strongEvolution}.}
\label{fig:SAVsimple}
\end{figure}

\begin{table}
\begin{tabular}{c|cc|cc}
&\multicolumn{2}{c}{nonlinear}&\multicolumn{2}{c}{augmented~SAV}\\
$\tau$ & $\sqrt{\expected{\norm{\cdot}_{L^2\trkla{\Om}}^2\big\vert_{t=2}}}$& EOC& $\sqrt{\expected{\norm{\cdot}_{L^2\trkla{\Om}}^2\big\vert_{t=2}}}$& EOC\\
\hline
$1\cdot10^{-4}$& 1.39E-01&--& 2.69E-01&--\\
$2\cdot10^{-4}$& 2.03E-01& 0.55& 5.43E-01& 1.01\\
$4\cdot10^{-4}$& 2.93E-01& 0.53& 7.99E-01& 0.56\\
$1\cdot10^{-3}$& 4.18E-01& 0.39& 1.05E-00& 0.30\\
$2\cdot10^{-3}$& 5.05E-01& 0.27& 1.30E-00& 0.32
\end{tabular}
\caption{Experimental order of convergence w.r.t.~$\tau$ for the nonlinear scheme \eqref{eq:majeeprohl} and the augmented SAV method \eqref{eq:discscheme} using $\mathcal{N}\trkla{0,1}$-random variables and $\alpha=10$.}
\label{tab:strongEOCnormal}
\end{table}

\begin{table}
\begin{tabular}{c|cc|cc}
&\multicolumn{2}{c}{augmented~SAV}&\multicolumn{2}{c}{standard SAV}\\
$\tau$ & $\sqrt{\expected{\norm{\cdot}_{L^2\trkla{\Om}}^2\big\vert_{t=2}}}$& EOC& $\sqrt{\expected{\norm{\cdot}_{L^2\trkla{\Om}}^2\big\vert_{t=2}}}$& EOC\\
\hline
$5\cdot10^{-5}$& 2.29E-01 &--&   1.14E-00&--\\
$1\cdot10^{-4}$& 3.79E-01& 0.73& 1.52E-00&0.42\\
$2\cdot10^{-4}$& 6.20E-01& 0.71& 1.50E-00&-0.03\\
$4\cdot10^{-4}$& 8.10E-01& 0.39& 1.46E-00&-0.04\\
$1\cdot10^{-3}$& 1.03E-00& 0.26& 1.37E-00&-0.07\\
$2\cdot10^{-3}$& 1.27E-00& 0.30& 1.21E-00&-0.18
\end{tabular}
\caption{Comparison between SAV methods and the nonlinear scheme \eqref{eq:majeeprohl} using $\mathcal{N}\trkla{0,1}$-random variables and $\alpha=10$.}
\label{tab:strongComparisonNewtonSAVnormal}
\end{table}

\begin{appendix}

\section{Proof of Lemma \ref{lem:BDG}}\label{sec:prelim}
\begin{proof}[Proof of Lemma \ref{lem:BDG}]
We start by establishing \eqref{eq:bdg1} for $p\in2\mathds{N}$.
The general case for $p\in[1,\infty)$ then follows by Hölder's inequality.
Introducing a sequence of independent Rademacher random variables $\trkla{\widehat{\epsilon}_k}_k$ on a new probability space $\trkla{\widehat{\Omega},\widehat{\mathcal{A}},\widehat{\Prob}}$, we note that $\trkla{\xi_k^{n,\tau}}_k$ and $\trkla{\widehat{\epsilon}_k \xi_k^{n,\tau}}$ are equal in distribution and obtain
\begin{align}
\expected{\norm{\Phi\h\no\sinc{n}}_H^p}=\tau^{p/2}\int_\Omega\norm{ \sum_{k\in\mathds{Z}\h}\lambda_k\Phi\h\no\g{k}\xi_k^{n,\tau}}_H^p\mathrm{d}\Prob=\tau^{p/2}\int_\Omega\int_{\widehat{\Omega}}\norm{ \sum_{k\in\mathds{Z}\h}\lambda_k\Phi\h\no\g{k}\xi_k^{n,\tau}\widehat{\epsilon}_k}_H^p\mathrm{d}\widehat{\Prob}\mathrm{d}\Prob\,.
\end{align}
Applying the Kahane--Khintchine inequality (see e.g.~\cite{pisier_2016,Hytonen2016}) pathwise for all $\omega\in\Omega$ yields the estimate
\begin{multline}
\expected{\norm{\Phi\h\no\sinc{n}}_H^p}\leq C_p\tau^{p/2}\int_\Omega \rkla{\int_{\widehat{\Omega}}\norm{\sum_{k\in\mathds{Z}\h}\lambda_k\Phi\h\no\g{k}\xi_k^{n,\tau}\widehat{\epsilon}_k}_H^2\mathrm{d}\widehat{\Prob}}^{p/2}\mathrm{d}\Prob\\
=C_p\tau^{p/2}\int_\Omega \rkla{\sum_{k\in\mathds{Z}\h}\norm{\lambda_k\Phi\h\no\g{k}}_H^2\abs{\xi_k^{n,\tau}}^2}^{p/2}\mathrm{d}\Prob=:\trkla{\star}
\end{multline}
with a constant $C_p$ depending on $p$, but not on $h$ or $\tabs{\mathds{Z}\h}$, respectively.
As by assumption $p/2\in\mathds{N}$, we can introduce a multi-index $\alpha\in\tgkla{0,\ldots,p/2}^{\tabs{\mathds{Z}\h}}$ and obtain due to the polynomial theorem
\begin{multline}
\trkla{\star}=C_p\tau^{p/2}\sum_{\tabs{\alpha}=p/2} K_\alpha\expected{\prod_{k\in\mathds{Z}\h}\norm{\lambda_k\Phi\h\no\g{k}}_H^{2\alpha_k}\abs{\xi_k^{n,\tau}}^{2\alpha_k}}\\
 \leq \overline{C}_p\tau^{p/2}\sum_{\tabs{\alpha}=p/2} K_\alpha\expected{\prod_{k\in\mathds{Z}\h}\norm{\lambda_k\Phi\h\no\g{k}}_H^{2\alpha_k}}=\overline{C}_p\tau^{p/2}\expected{\rkla{\sum_{k\in\mathds{Z}\h}\norm{\lambda_k\Phi\h\no\g{k}}_H^2}^{p/2}}\,
\end{multline}
with $K_\alpha$ being the multinomial coefficients.
Here we used \ref{item:randomvariables} to deduce $\displaystyle\expected{\prod_{k\in\mathds{Z}\h}\abs{\xi_k^{n,\tau}}^{2\alpha_k}}\leq C_p$ for all $\tabs{\alpha}=p/2$ with a constant $C_p$ independent of $h$.\\
To establish \eqref{eq:bdg2}, we adapt the ideas of Lemma 2.8 in \cite{Ondrejat2022} and obtain due to Doob's maximal inequality and the Burkholder--Rosenthal inequality (cf.~\cite[Theorem 5.50]{pisier_2016})
\begin{align}\label{eq:rosenthal}
\begin{split}
\expected{ \max_{0\leq l\leq m} \norm{\sum_{n=1}^l\Phi\h\no\sinc{n}}_H^p}\leq&\, C\expected{\rkla{\sum_{n=1}^m\expected{\norm{\Phi\h\no\sinc{n}}_H^2\,\big\vert\,\mathcal{F}_{t\no}^\tau}}^{p/2}}\\
&+C\expected{\max_{1\leq n\leq m}\norm{\Phi\h\no\sinc{n}}_H^p}\,,
\end{split}
\end{align}
where the constant $C$ depends on $p$, but not on $h$.
Due to \ref{item:increments} and \ref{item:randomvariables}, we obtain for the first term on the right-hand side of \eqref{eq:rosenthal}
\begin{align}
\begin{split}
&\expected{\rkla{\sum_{n=1}^m\expected{\norm{\Phi\h\no\sinc{n}}_H^2\,\big\vert\,\mathcal{F}_{t\no}^\tau}}^{p/2}}\\
&= \expected{\rkla{\sum_{n=1}^m\tau \expected{\sum_{k,l\in\Zh}\rkla{\lambda_k\Phi\h\no\g{k},\lambda_l\Phi\h\no\g{l}}_H\xi_k^{n,\tau}\xi_l^{n,\tau}\,\big\vert\,\mathcal{F}_{t\no}^\tau}}^{p/2}}\\
&=\expected{\rkla{\sum_{n=1}^m\tau \sum_{k\in\Zh}\norm{\lambda_k\Phi\h\no\g{k}}_{H}^2}^{p/2}}\leq \expected{\trkla{m\tau}^{(p-2)/2}\sum_{n=1}^m\tau\rkla{\sum_{k\in\Zh}\norm{\lambda_k\Phi\h\no\g{k}}_H^2}^{p/2}}\,.
\end{split}
\end{align}
Applying \eqref{eq:bdg1} to the second term, we obtain
\begin{align}
\begin{split}
\expected{\max_{1\leq n\leq m}\norm{\Phi\h\no\sinc{n}}_H^p}\leq C\sum_{n=1}^m\tau^{p/2}\expected{\rkla{\sum_{k\in\Zh}\norm{\lambda_k\Phi\h\no\g{k}}_H^2}^{p/2}}\,,
\end{split}
\end{align}
which provides \eqref{eq:bdg2}.

\end{proof}

\end{appendix}


\bibliographystyle{amsplain}
\providecommand{\bysame}{\leavevmode\hbox to3em{\hrulefill}\thinspace}
\providecommand{\MR}{\relax\ifhmode\unskip\space\fi MR }
\providecommand{\MRhref}[2]{%
  \href{http://www.ams.org/mathscinet-getitem?mr=#1}{#2}
}
\providecommand{\href}[2]{#2}

\end{document}